%% file: drinheis.tex
\documentclass[preprint,12pt]{elsarticle}

\usepackage{amssymb}
\usepackage[
colorlinks=true,
linkcolor=black, 
anchorcolor=black,
citecolor=black, 
urlcolor=black, 
]{hyperref}
\usepackage{import}
\usepackage{amsmath}
\usepackage{amsthm}
\usepackage{mathabx}
\usepackage[english]{babel}
\usepackage{url}
\usepackage{amsfonts}
\usepackage{geometry}
\usepackage[small,nohug,heads=vee]{diagrams}
\diagramstyle[labelstyle=\scriptstyle]
\newarrow{Mono}{C}{-}{-}{-}{>}

\def\acts{\mathrel{\reflectbox{$\righttoleftarrow$}}}
\newcommand{\vin}{\rotatebox[origin=c]{-90}{$\acts$}}

\newcommand{\bigslant}[2]{{\raisebox{.2em}{$#1$}\left/\raisebox{-.2em}{$#2$}\right.}}

\newcommand{\leftsub}[2]{{\vphantom{#2}}_{#1}{#2}}
\newcommand{\leftexp}[2]{{\vphantom{#2}}^{#1}{#2}}
\newcommand{\leftexpsub}[3]{{\vphantom{#3}}^{#1}_{#2}{#3}}

\newcommand{\lmod}[1]{#1\text{-}\mathbf{Mod}}
\newcommand{\rmod}[1]{\mathbf{Mod}\text{-}#1}
\newcommand{\lcomod}[1]{#1\text{-}\mathbf{CoMod}}
\newcommand{\rcomod}[1]{\mathbf{CoMod}\text{-}#1}

\newcommand{\un}[1]{\underline{#1}}
\newcommand{\ov}[1]{\overline{#1}}

\newcommand{\Ann}{\operatorname{Ann}}

\newcommand{\coev}{\operatorname{coev}}

\newcommand{\Drin}{\operatorname{Drin}}
\newcommand{\ev}{\operatorname{ev}}

\newcommand{\Heis}{\operatorname{Heis}}

\newcommand{\Hom}{\operatorname{Hom}}
\newcommand{\ide}{\operatorname{Id}}

\newcommand{\Ind}{\operatorname{Ind}}

\newcommand{\Isom}{\operatorname{Isom}}

\newcommand{\Nat}{\operatorname{Nat}}

\newcommand{\coop}{\overline{\operatorname{cop}}}
\newcommand{\oop}{\overline{\operatorname{op}}}

\newcommand{\Res}{\operatorname{Res}}

\newcommand{\triv}{\operatorname{triv}}

\newcommand{\Alg}{\mathbf{Alg}}
\newcommand{\BiAlg}{\mathbf{BiAlg}}
\newcommand{\BiMod}{\mathbf{BiMod}}

\newcommand{\Cat}{\mathbf{Cat}}
\newcommand{\CoAlg}{\mathbf{CoAlg}}

\newcommand{\Hopf}{\mathbf{Hopf}}
\newcommand{\Mod}{\mathbf{Mod}}

\newcommand{\Vect}{\mathbf{Vect}_k}
\newcommand{\fVect}{\mathbf{Vect}_k^\text{fd}}
\newcommand{\Set}{\mathbf{Set}}

\newcommand{\g}{\mathfrak{g}}

\newcommand{\Ug}{U_q(\mathfrak{g})}
\newcommand{\Uq}{U_q(\mathfrak{sl}_2)}

\providecommand{\cal}[1]{\mathcal{#1}}

\providecommand{\fr}[1]{\mathfrak{#1}}

\providecommand{\op}[1]{\operatorname{#1}}

\newcommand{\mC}{\mathbb{C}}

\newcommand{\mZ}{\mathbb{Z}}
\newcommand{\mN}{\mathbb{N}}

\newcommand{\mA}{\mathbb{A}}

\newcommand{\cC}{\mathcal{C}}
\newcommand{\cD}{\mathcal{D}}
\newcommand{\cB}{\mathcal{B}}

\newcommand{\cH}{\mathcal{H}}
\newcommand{\cO}{\mathcal{O}}
\newcommand{\cI}{\mathcal{I}}

\newcommand{\cV}{\mathcal{V}}
\newcommand{\cW}{\mathcal{W}}
\newcommand{\cM}{\mathcal{M}}

\newcommand{\cZ}{\mathcal{Z}}

\numberwithin{equation}{section}
\numberwithin{figure}{section}

\newtheorem{theorem}{Theorem}[subsection]

\newtheorem{proposition}[theorem]{Proposition}
\newtheorem{corollary}[theorem]{Corollary}
\newtheorem{lemma}[theorem]{Lemma}

\newdefinition{definition}[theorem]{Definition}
\newdefinition{notation}[theorem]{Notation}
\newdefinition{setting}[theorem]{Setting}
\newdefinition{convention}[theorem]{Convention}

\newdefinition{example}[theorem]{Example}
\newdefinition{remark}[theorem]{Remark}
\newdefinition{exercise}[theorem]{Exercise}
\newdefinition{warning}[theorem]{Warning}
\newdefinition{question}[theorem]{Question}

\journal{Journal of Pure and Applied Algebra}

\begin{document}

\begin{frontmatter}


\title{Braided Drinfeld and Heisenberg Doubles}
\author{Robert Laugwitz\corref{cor1}}
\ead{laugwitz@maths.ox.ac.uk}
\ead[url]{http://people.maths.ox.ac.uk/~laugwitz/}
\cortext[cor1]{Corresponding Author}
\address{Mathematical Institute, University of Oxford, Andrew Wiles Building, Radcliffe Observatory Quarter, Woodstock Road, Oxford, OX2 6GG}


\begin{abstract}
In this paper, the Drinfeld center of a monoidal category is generalized to a class of mixed Drinfeld centers. This gives a unified picture for the Drinfeld center and a natural Heisenberg analogue. Further, there is an action of the former on the latter. This picture is translated to a description in terms of Yetter-Drinfeld and Hopf modules over quasi-bialgebras in a braided monoidal category. Via braided reconstruction theory, intrinsic definitions of braided Drinfeld and Heisenberg doubles are obtained, together with a generalization of the result of Lu (1994) that the Heisenberg double is a 2-cocycle twist of the Drinfeld double for general braided Hopf algebras.
\end{abstract}


\end{frontmatter}




\section{Introduction}

\subsection{Motivation}

The Drinfeld double was originally introduced as the quantum double by Drinfeld in \cite{Dri}. The construction was generalized to quasi-Hopf algebras in \cite{Maj3}, and to braided Hopf algebras in \cite{Maj2}. 

There are different ways to motivate the introduction of the Drinfeld double. For example, it gives a way to construct morphism $\Psi\colon V\otimes V\to V\otimes V$ satisfying the \emph{Yang-Baxter equation}
\begin{equation}\label{yangbaxter}
(\Psi\otimes \ide_V)(\ide_V\otimes \Psi)(\Psi\otimes \ide_V)=(\ide_V\otimes \Psi)(\Psi\otimes \ide_V)(\ide_V\otimes \Psi).
\end{equation}
It also gives a natural way of associating to a Hopf algebra a quasitriangular Hopf algebra (i.e. one that is almost cocommutative).

The Heisenberg double can be given a similar interpretation, where the Yang-Baxter equation is replaced by the \emph{Pentagon equation} (see \cite{Kas})
\begin{equation}\label{pentagon}
(\Phi\otimes \ide_V)(\ide_V\otimes \Phi)(\Phi\otimes \ide_V)=(\Phi\otimes \ide_V)(\ide_V\otimes \Phi).
\end{equation}
Every module $V$ over the Heisenberg double comes with a map $\Phi\colon V\otimes V\to V\otimes V$ satisfying this equation.

In \cite{Kas} it was also shown that solutions for (\ref{yangbaxter}) can be obtained from solutions for Eq. (\ref{pentagon}). One application of the results of this paper is to show that given a solution $(V,\Psi)$ to (\ref{yangbaxter}) and $(W,\Phi)$ to (\ref{pentagon}), $(V\otimes W,\Psi^{(1)}\otimes \Phi^{(1)}\otimes \Psi^{(2)}\otimes \Phi^{(2)})$ again has the structure of a solution to the pentagon equation. For finite-dimensional Hopf algebras, this follows from a twisting result of \cite{Lu} which is generalized to the context of braided Hopf algebras (so-called \emph{braided groups} in the braided cocommutative case in \cite{Maj6,Maj8} and other papers), in \ref{twistaction} and, more generally, monoidal categories in Theorem \ref{cataction}.

To argue why it is beneficial to define a braided version of the Drinfeld and Heisenberg double, it is helpful to consider an example. Let $B$ is the coordinate ring $\cO_X$ for $X=\mA^n$. In this case, $\Heis(B)$ is the ring of differential operators on $X$, $D_X$. However, if we compute the Drinfeld double of $B$, then this simply gives $\cO_X[\partial_1,\ldots,\partial_n]=\cO_{T^*X}$. This is a commutative and cocommutative Hopf algebra. A more interesting object is obtained by considering $B$ as a Hopf algebra in the category of YD-modules over the group $C_2$. This can be seen as a \emph{super} algebra version of the coordinate ring. Computing the braided Drinfeld double $\Drin_{\Drin{C_2}}(B)$ gives a non-commutative Hopf algebra in which the commutator relation
\begin{equation}
[\partial_i,x_j]=(1-\delta_1-\delta_{-1})\delta_{i,j}
\end{equation}
holds (see \ref{demonstrationexpl}). From the point of view that the Drinfeld double gives examples of ``Quantum groups" it is more natural to have a non-commutative and non-cocommutative Hopf algebra. Note that computing the Heisenberg double over $\Drin(C_2)$ gives $D_X\otimes kC_2$, so the answer for the Heisenberg double is essentially not changed as the commutator relation in the braided Heisenberg double remains
\begin{equation}
[\partial_i,x_j]=\delta_{i,j}.
\end{equation}

Another application of the braided version is to give a clean description of the quantum groups $\Ug$ as Drinfeld doubles (see \cite{Maj2}), while it was already observed in \cite{Dri} that $\Ug$ is a quotient of the Drinfeld double of $U_q(\fr{n}_+)\rtimes U_q(\fr{t}).$ In our unified picture, we now have a natural Heisenberg analogue for the quantum groups, for which the commutator relation is 
\begin{equation}
[E_i,F_j]=\frac{K^{-\tfrac{i\cdot i}{2}}}{q_i^{-1}-q_i}.
\end{equation}
This algebra has no finite-dimensional representations (see \ref{quantumgroupaction}).

Another generalization included in the constructions of this paper is to allow quasi-Hopf algebras (introduced in \cite{Dri2}), which we consider in a braided monoidal category. Including this direction of generalization will enable us to consider examples such as the twisted Drinfeld double $\Drin^\omega(G)$  of a group $G$ which is of relevance in mathematical physics as it occurs as data associated to particular orbifolds in Rational Conformal Field theory (see \cite{DPR}). In \cite{Maj3}, a construction of the Drinfeld double of a quasi-Hopf algebra is given via reconstruction theory. We add a Heisenberg analogue to this picture and generalize it to quasi-Hopf algebras in a braided monoidal category, thus combining the two directions of generalization (braiding and twisting).

In a derived setting (see e.g. \cite{BFN}), the Drinfeld center gives a categorical version of \emph{Hochschild cohomology} which is defined as
\[
\mathbf{HH}^\bullet(\cM)=\mathbf{HH}^\bullet(\cM,\cM):=\cZ(\cM)=\BiMod_\cM(\cM^{\op{reg}},\cM^{\op{reg}}).
\]
From this point of view, in Section \ref{catback}, we consider a more general form of Drinfeld centers, which can be viewed as categorical Hochschild cohomology with values in other bimodule categories $\cV$ over $\cM$,
\[
\mathbf{HH}^\bullet(\cM,\cV):=\BiMod_\cM(\cM,\cV).
\]
Even though the constructions are given in a non-derived setting in the (2,1)-meta category $\textbf{Cat}$ of categories, because of the intrinsic nature of the definitions in terms of bimodule categories, it is possible to obtain higher-categorical analogues working in the meta-category $(\infty,1)$-category of $(\infty,1)$-categories $\mathbf{Cat}_\infty$ instead using an appropriate formalism of bimodule categories (such as e.g. \cite{Lur2}). Hochschild cohomology with coefficients in any $\cM$-bimodule $\cV$ will give module categories over $\mathbf{HH}^\bullet(\cM)$, and relative versions of these with respect to $\cM$ living over a braided monoidal category $\cB$. In particular, the Hopf center from Definition \ref{hopfcenter} -- which is the Heisenberg analogue of the Drinfeld center -- can be interpreted as $\mathbf{HH}^\bullet_\cB(\cM,\leftexp{\op{reg}}{\cM}^{\triv})$ in a derived setting. Here, coefficients are taken in the $\cM$-bimodule $\leftexp{\op{reg}}{\cM}^{\triv}$ which has the regular action on the left and the trivial action (given by the underlying tensor-product in $\cB$) on the right.

The original motivation for the author to write this paper lies in the applications of the categorical action of $\lmod{\Drin_H(C,B)}$ on $\lmod{\Heis_H(C,B)}$ to the \emph{rational Cherednik algebras} $H_{t,c}(G)$ of \cite{EG}. For this, the morphisms from \cite{BB}
\[
M_c\colon H_{0,c}(G)\longrightarrow \Heis_{\mC G}(U(\fr{yb}^*_G),U(\fr{yb}_G)),
\]
where $U(\fr{yb}_G)$ is a generalization of the algebra $U(\fr{tr}_n)$ from \cite{BEER} to general complex reflection groups (cf. \cite{Lau4} for the notation), are used. Analogues also exist for parameters $t\neq 0$, or the restricted rational Cherednik algebras. The morphisms $M_c$ can be used to restrict the categorical action studied in this paper to gives actions
\[
\triangleright_{t,c} \colon \lmod{\Drin_{\mC G}(U(\fr{yb}^*_G),U(\fr{yb}_G))}\otimes \lmod{H_{t,c}(G)}\longrightarrow \lmod{H_{t,c}(G)}.
\]
It is work in progress to study these actions using category $\cO$ techniques as in \cite{GGOR}.

\subsection{Summary}

The discussion of Drinfeld and Heisenberg doubles in this paper is done at different levels of generality which are structured by the sections:
\begin{diagram}
&&{\text{Section~\ref{catback}:}\atop  {{\text{$B$ quasi-Hopf algebra}}\atop {\text{in $\cB$}}}}&&\\
&\ldTo&&\rdTo&\\
{\text{Section~\ref{doubles}:}\atop {\text{$B$ Hopf algebra in $\lmod{H}$}\atop \text{(or $\lcomod{A^{\oop}}$})}}&&&&{\text{Section~\ref{twistedchapter}:}\atop {\text{$B$ quasi-Hopf algebra}\atop \text{in $\Vect$.}}}
\end{diagram}
In Section \ref{catback} we work on the level of a monoidal category $\cM$ living over a braided monoidal category $\cB$. That is, there exist functors $F\colon \cM\to \cB$ and $P\colon \cB\to \cM$ such that $FP\cong \ide_{\cB}$. It is not necessary to make further assumptions on the braided monoidal category, such as a $k$-linear structure, although in typical examples, the categories will have fiber functors to the category of $k$-vector spaces for a field $k$.

We start by giving the most general definition of a \emph{mixed relative Drinfeld center} in Section \ref{relativecentersection}. This is done by translating the data of morphisms of $\cM$-bimodules $\cM^{\op{reg}}\to \leftexp{G_1}{\cV}^{G_2}$ to pairs $(V,c)$, where $V$ is an object of $\cV$ and $c\in \Nat^{\otimes}(G_1\otimes V\to V\otimes G_2)$. Here, $\leftexp{G_1}{\cV}^{G_2}$ is the $\cM$-bimodule where the left module structure is induced by pulling the regular action back along $G_1$, and the right one along $G_2$. We will focus on two classes of examples: the relative Drinfeld center $\cZ_{\cB}(\cM)$, and its ``Heisenberg analogue" (called the Hopf center) $\cH_{\cB}(\cM)$. The Drinfeld center corresponds to the pair of functors $(G_1,G_2)=(\ide_\cM,\ide_\cM)$, the Hopf center to $(G_1,G_2)=(\triv,\ide_\cM)$, where $\triv:=PF$. The main observation is a natural action of the Drinfeld center on the Hopf center.

For the purposes of this paper, we introduce a slightly more general version of Majid's braided reconstruction theory in Section \ref{reconstruction}, working with quasi-Hopf algebra objects in $\cB$. This generalizes work of \cite{Hae}. We further give a categorical interpretation of the concept of quasitriangularity in Section \ref{quasitriangularity}.

In Section \ref{ydsection}, we consider monoidal categories of the form $\cM=\lmod{B}(\cB)$ for a quasi-bialgebra (or quasi-Hopf algebra) object $B$ in $\cB$. For such $\cM$, the Drinfeld center can be reformulated as the category of Yetter-Drinfeld modules, while the Hopf center consists of Hopf modules. In the case of the Drinfeld center, this is well-known for Hopf algebras. A version for braided Hopf algebras is due to \cite{Maj2}, and a version for quasi-Hopf algebras in $\Vect$ can be found in \cite{Maj3}. Working with strict Hopf algebras (trivial 3-cycles) many formulas simplify as summarized in Section \ref{strictsection}.

As preparation for the definition of the braided Drinfeld and Heisenberg double requires working with two dually paired braided Hopf algebra $C,B$ (see Section \ref{duallypaired} for the conventions used). We embed the Drinfeld and Hopf center into larger categories of left $C$ and right $B$-modules which satisfy compatibility conditions resembling those of Yetter-Drinfeld (respectively Hopf) modules (see Section \ref{pairedydsect}). We also discuss a reformulation of Majid's concept of weak quasitriangularity in \ref{BCquasitriangular}. This concept is needed to obtain the quantum groups as examples of braided Drinfeld doubles as in \cite{Maj2}.

To summarize, we give five formulations of the action of the relative Drinfeld center on the relative Hopf center, for $C,B$ dually paired Hopf algebras in $\cB$:
\begin{diagram}
\BiMod_\cM^\cB(\cM^{\op{reg}},\cM^{\op{reg}})&\cong& \cZ_\cB(\cM)&\cong & \leftsub{B^{\oop}}{\cal{YD}}^B(\cB)&\hookrightarrow& \leftexp{C}{\cal{YD}}^B(\cB)&\cong &\lmod{\Drin_H(C,B)}\\
\vin&&\vin&&\vin&&\vin&&\vin\\
\BiMod_\cM^\cB(\cM^{\op{reg}},\leftexp{\op{reg}}{\cM}^{\triv})&\cong&\cH_\cB(\cM)&\cong &\leftsub{B^{\oop}}{\cal{H}}^B(\cB)&\hookrightarrow&\leftexp{C}{\cal{H}}^B(\cB)&\cong &\lmod{\Heis_H(C,B)}
\end{diagram}

In Section \ref{quantumgroupaction}, we consider the example of the quantum groups $\Ug$. In this example, our result will give a categorical action of the category of $\Ug$-modules on the category of $D_q(\fr{g})$-modules, which can be interpreted as a category of quantum differential operators.

We now have a machinery to also define twisted braided Drinfeld and Heisenberg doubles. There exist different concepts of twist in the literature which will appear at different places in this exposition. To provide an overview:
\begin{itemize}
\item A (right) 2-cocycle twist of a Hopf algebra is a way of obtaining new algebras $H_\sigma$ from the datum of a (braided) Hopf algebra $H$ together with a 2-cocycle $\sigma$. This is used in \ref{cocycletwists} to twist the braided Drinfeld double, giving the braided Heisenberg double.
\item The \emph{Drinfeld twist} of a (quasi-)Hopf algebra goes back to \cite{Dri3}. It provides a way to change the monoidal structure of a category of representations over a quasi-bialgebra in an equivalent way by means of conjugation by an element $F$ of $B\otimes B$ (see e.g. \cite{GM,Maj1}). We include this concept in \ref{drinfeldtwist}.
\item A twisted version of the Drinfeld double of a group algebra can be found in the literature (see e.g. \cite{DPR,Maj3,Wil}). Underlying its notion of twist is the idea that a commutative bialgebra can be viewed as a quasi-bialgebra with respect to any 3-cycle. From this point of view, twisted versions of commutative Hopf algebras can be introduced in larger generality (see \ref{drinfeldtwist}).
\end{itemize}

Applying the categorical action to the case of the twisted Hopf algebra $k^\omega[G]$ of functions on a group, there is a categorical action of modules over the twisted Drinfeld double $\Drin^\omega(G)$, which is the category of $G^{\op{ad}}$-equivariant $\omega$-twisted vector bundles on the category of $\omega$-twisted $G^{\op{reg}}$-equivariant twisted vector bundles on $G$. This is the topic of Section~\ref{twistedgroups} which concludes this paper.

\subsection{Hints on Reading this Paper}

The basic structure of this paper is a transgression from category theory (Section~\ref{catback}) to representation theory of algebras (Sections \ref{doubles} and \ref{twistedchapter}). The link is given by braided reconstruction theory (Section \ref{reconstruction}).

The exposition is significantly easier if one works with strict monoidal categories $\cM$ (i.e. Hopf algebras via reconstruction theory). For the purpose of considering twisted Drinfeld doubles, we include the formulas for the more general case of non-strict monoidal categories and quasi-Hopf algebras. The readers only interested in the strict case can safely skip to Section~\ref{doubles} and when looking up the relevant proofs in Section~\ref{catback} treat associativity and rigidity isomorphisms as identities.

Throughout Section~\ref{catback}, we find it most effective to do the proofs using graphical calculus. This however requires to work with a strict monoidal base categories $\cB$ (while $\cM$ may still be non-strict). We refer to Mac Lane's coherence theorem to justifying giving many proofs on this level. Often, in the proofs the computations are not given in detail. It is an essential standing exercise in reading this paper to always draw diagrams for all statements and proofs that come up.

If the reader is only interested in the Drinfeld and Heisenberg doubles as (Hopf) algebras, Section \ref{presentations} is a good point to start. In this section, concrete examples are provided as well.

\subsection{Some Notational Conventions}
In this paper, $k$ always denotes a field. The category of finite-dimensional $k$-vector spaces is denoted by $\fVect$, the category of possibly infinite-dimensional $k$-vector spaces by $\Vect$. We denote the symmetric monoidal category of (co)algebras in $\Vect$ by $\Alg$ (respectively $\CoAlg$).

More generally, the concept of an algebra, $\Alg(\cM)$, and coalgebra, $\CoAlg(\cM)$, and their modules can be defined in any monoidal category $\cM$. To illustrate the idea, the multiplication is a morphism $m\colon A\otimes A\to A$ in $\cM$. It satisfies associativity and unitarity with respect to $1\colon I\to A$ which is a morphism in $\cM$. These properties are commutative squares and can be expressed in $\cM$. For more details on this approach see e.g. \cite[9.2.11ff.]{Maj1} or \cite{Maj6}. Given an algebra (or coalgebra) object in a monoidal category $\cM$, we denote the category of left $A$-modules (respectively comodules) in $\cM$ by $\lmod{A}(\cM)$ (respectively $\lcomod{A}(\cM)$) and right $A$-modules by $\rmod{A}(\cM)$. If $\cM=\Vect$, we omit mentioning the category $\cM$ and simply write $\lmod{A}$ (respectively $\lcomod{A}$).

Functors of monoidal categories are always strong monoidal (sometimes strict monoidal). For compositions of morphisms or functors, we write $f\circ g$ simply as $fg$. In the whole paper, $\cB$ will denote a strict monoidal braided category\footnote{For an introduction to braided monoidal categories see e.g. \cite{JS} or \cite{Maj1}.}, with braiding $\Psi$. The monoidal category $\cM$ typically lives over $\cB$. That is, there exists a monoidal (fiber) functor $\cM\to \cB$. We do not require $\cM$ to be strict monoidal itself.

\subsection{Bialgebra and Hopf Algebra Objects}

In order to define bialgebras and Hopf algebras, one needs a braided monoidal category $\cB$ with braiding $\Psi$. We will always treat the base category $\cB$ as strict monoidal. The categories $\Alg(\cB)$ and $\CoAlg(\cB)$ of algebra and coalgebra objects then have a monoidal structure fibered over $\cB$. We denote the product on $A\otimes B$ by $m_{A\otimes B}$ for two algebra objects $A$, $B$ in $\cB$. That is, $m_{A\otimes B}=(m_A\otimes m_B)(\ide_A\otimes \Psi_{A,B}\otimes \ide_B)$. Inductively, denote by $m_{B^{\otimes n}}$ the product on $B^{\otimes n}$. Dually, we denote the coalgebra structure on $C\otimes D$ for two coalgebras $C$, $D$ in $\cB$ by $\Delta_{C\otimes D}$. We will occasionally use the notation $m^k$, for the map $m(m\otimes \ide)\ldots (m\otimes \ide)\colon B^{\otimes k+1}\to B$ obtained by applying $m$ $k$ times.

In $\cB$, we can define bialgebra objects as simultaneous algebras and coalgebras satisfying the bialgebra  condition
\begin{equation}
\Delta m= (m\otimes m)(\ide\otimes \Psi\otimes \ide)(\Delta\otimes \Delta).
\end{equation}
We call a bialgebra (resp. Hopf algebra) object in $\cB$ a \emph{braided} bialgebra (or Hopf algebra). In order for $\BiAlg(\cB)$ to be monoidal, a braiding is not sufficient, but a symmetric monoidal structure is. However, the category $\lmod{B}(\cB)$ is monoidal using the comultiplication. It is important to use this more general definition to study main examples such as the quantum groups (or more generally, Nichols algebras) later. It is also important that we do not restrict ourselves to finite-dimensional Hopf algebras over $k$.

\subsection{Bialgebras vs. Quasi-Bialgebras}\label{quasihopf}

Let $B$ be a bialgebra in $\cB$. Then the category $\lmod{B}(\cB)$ is strict monoidal with fiber functor over $\Vect$. That is, the underlying morphisms in $\cB$ of the associativity transformation $\alpha\colon \otimes(\otimes \times \ide )\to \otimes (\ide \times \otimes)$ are identity morphisms. In some cases, one requires a higher level of generality (for example, when working with twists of Hopf algebras as in \cite{Dri2}) and wants to drop the assumption of $\cM$ being strict. The natural notion arising via reconstruction theory (see \ref{reconstruction}) is that of a \emph{quasi}-bialgebra. Following \cite{Maj3}, we require that there exists an invertible element $\phi\in B\otimes B\otimes B$ (the \emph{coassociator}) such that
\begin{equation}
m_{B\otimes B\otimes B}(\ide_B\otimes \Delta\otimes \phi)\Delta=m_{B\otimes B\otimes B}(\phi\otimes \Delta\otimes \ide_B)\Delta.
\end{equation}
Such an element $\phi$ needs to satisfy the 3-cycle condition of a non-abelian homology theory (see e.g. \cite[Section~6]{Maj4}, \cite[2.3]{Maj1}):
\begin{equation}\label{3cocycle}
\begin{split}&m_{B\otimes B\otimes B}(m_{B\otimes B\otimes B}\otimes \ide_{B\otimes B\otimes B})(1\otimes\phi\otimes \ide_B\otimes \Delta\otimes \ide_B\otimes\phi\otimes 1) \phi\\&=m_{B\otimes B\otimes B}(\ide_{B\otimes B}\otimes \Delta\otimes \Delta \otimes \ide_{B\otimes B})(\phi\otimes \phi)
\end{split}
\end{equation}
The counitary property still holds as in the bialgebra case, given that $(\ide\otimes \varepsilon\otimes \ide)\phi=1\otimes 1$. For a quasi-bialgebra $B$, the categories $\lmod{B}(\cB)$ (and $\rmod{B}(\cB)$, $\lcomod{B}(\cB)$, $\rcomod{B}(\cB)$) are monoidal.

If $B$ is a (quasi)-Hopf algebra object in $\cB$, then $\lmod{B}(\cB)$ is rigid given that $\cB$ is rigid. That is, left dual objects exist\footnote{See e.g. \cite[Section~9.3]{Maj1}.} and are denoted by $V^*$ for $V\in \cB$. That is for example the case if $\cB=\fVect$. For infinite-dimensional modules $V$, we can still give the finite dual $V^\circ= \lbrace \delta_v\mid v\in V\rbrace$ a module structure, but there is no coevaluation map. We observe that left dual objects are unique up to canonical isomorphism.

When working with quasi-Hopf algebras, the antipode axioms are valid only up to elements $a, b \in H$ (cf. e.g. \cite[XV.5]{K}, or  Section \ref{reconstruction} in the braided setting):
\begin{equation}\label{quasiantipode}
m^2(S\otimes a\otimes \ide)\Delta=a\epsilon,\text{  and  }m^2(\ide\otimes b\otimes S)\Delta=b\epsilon.
\end{equation}
This requires the compatibility conditions
\begin{equation}\label{antipodecyclecond}
m^5(\ide\otimes b\otimes S\otimes a\otimes\ide)\phi=1,\text{  and  } m^5(S\otimes a\otimes \ide\otimes b\otimes S)\phi^{-1}=1,
\end{equation}
with the coassociators. The formulas may be more clear when drawn as diagrams (using graphical calculus) or using generalized Sweedler's notation\footnote{We will stick to \cite[Preliminaries]{Maj3}) for conventions about Sweedler's notation (from \cite{Swe}). We will use these conventions, including the Einstein sum convention from Section \ref{doubles} onward.}. The notion of a quasitriangular quasi-Hopf algebra is also spelled out in \cite[2.4]{Maj1}. We also include a version for braided quasi-Hopf algebras in \ref{reconstruction}.

\subsection{Dually Paired Hopf Algebras}\label{duallypaired}

Let $\cB$ be a braided monoidal category with braiding $\Psi$ (recall $\cB$ is always treated as strict monoidal in this paper). In this section, we want to discuss what notion of dually paired Hopf algebras is suitable for our purposes in the remainder of the paper. Unlike working in the category of finite dimensional vector spaces $\Vect$, a dual may not necessarily exist in this more general setting. We assume that $C,B$ are braided Hopf algebras in $\cB$ with a pairing, in the sense that there exists an evaluation map $\ev\colon C\otimes B \to I$ to the unit $I$ in $\cB$ compatible with the structure. This displays $C$ as the left (categorical) dual of $B$. That is, using graphical calculus\footnote{We stick to the conventions of \cite{Maj1} about graphical calculus of Hopf algebra objects in $\cB$. The drawings are created using inkscape.} the conditions from Figure~\ref{dualityaxiomshopf} hold. If $C,B$ are Hopf algebras, then we further assume that the antipodes are invertible and the duality $\ev(S\otimes \ide)=\ev(\ide\otimes S)$ holds. 

\begin{figure}
\begin{center}
\small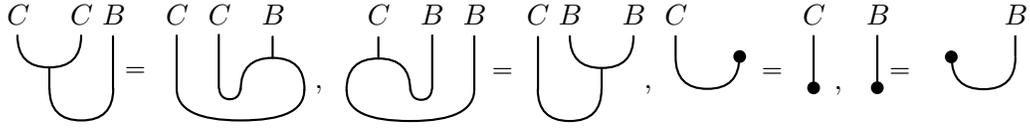
\caption{Dually paired Hopf algebras}
\label{dualityaxiomshopf}
\end{center}
\end{figure}

\begin{remark} Note that we do not restrict ourselves to treating finite-dimensional Hopf algebras here. In particular, a coevaluation map $\coev\colon I\to B\otimes C$ may not exist as a morphism in $\cB$. In later applications, we use infinite-dimensional Hopf algebras and their restricted duals, where the coevaluation map exists as a formal power series rather than a linear map given that the pairing is perfect and dual bases exist (\ref{rmatricessect}). The restricted dual of a Hopf algebra $H$ is denoted by $H^\circ$ and consists of those functions that vanish on a (two-sided) ideal of $H$ of finite codimension. This is not necessarily equal to the finite dual of an infinite-dimensional vector space and the pairing is not necessarily perfect.
\end{remark}

\begin{remark}
The situation for quasi-bialgebras is asymmetric. The dual of a quasi-bialgebra has a multiplication which is not strictly associative.
\end{remark}

Let us denote the braided category $\cB$ with \emph{inverse} braiding $\Psi^{-1}$ by $\overline{\cB}$. In $\cB$, the categories $\lmod{B}(\cB)$, $\lcomod{B}(\cB)$ of left (co)modules (and their right versions) are monoidal. Given a dual pair $B, C$ as above, we further observe that there exists functors of monoidal categories
\[
\leftsub{B}{\Phi}\colon \lcomod{B}(\cB)\to \lmod{\leftexp{cop}{C}}(\overline{\cB}),\qquad
\Phi_{C}\colon \rcomod{C}(\cB)\to \rmod{\leftexp{cop}{B}}(\overline{\cB}).
\]
Here, $\leftexp{cop}{C}$ denotes $C$ with co-opposite coproduct ${\Psi}^{-1}\Delta$. This is a bialgebra (resp. Hopf algebra) in the braided monoidal category $\overline{\cB}$ (with antipode $S^{-1}$). The functor $\leftsub{B}{\Phi}$ maps a comodule with coaction $\delta\colon V\to B\otimes V$ to $V$ with action $(\ev\otimes \ide)(\ide\otimes \delta)$, and $\Phi_C$ is defined analogously. We find it helpful to check such statements using graphical calculus, in which the braiding $\Psi$ and its inverse $\Psi^{-1}$ are denoted by 
\[
\begin{array}{ccc}
\Psi\quad=\vcenter{\hbox{\small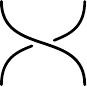}}&\quad\text{and}&\Psi^{-1}\quad=\vcenter{\hbox{\small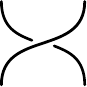}}.
\end{array}
\]
Note that if $C,B$ are finite-dimensional, the functors $\leftsub{B}{\Phi}$ and $\Phi_C$ are part of equivalences of categories. In general, this is not the case. If we restrict to situations in which the braided monoidal category $\cB$ admits a monoidal functor $\cB\to \Vect$ to vector spaces over a field $k$, then we can talk about the pairing $\ev$ being perfect. If that is the case, the above functors will be fully faithful. In the general situation we \emph{define} the pairing to be perfect if the functors $\leftsub{B}{\Phi}$ and $\Phi_C$ are fully faithful.

\begin{remark}
The bialgebras $\leftexp{cop}{C}$ and $\leftexp{cop}{B}$ are dually paired in a different way in $\cB$ (see Figure \ref{strangepairing}). Here, it is important to distinguish whether we express a functional identity in $\cB$ or in $\overline{\cB}$. In this example, the second term is expressed using symbols in $\ov{\cB}$, while the third term is the same expression written in $\cB$.
\end{remark}

\begin{figure}
\begin{center}
\small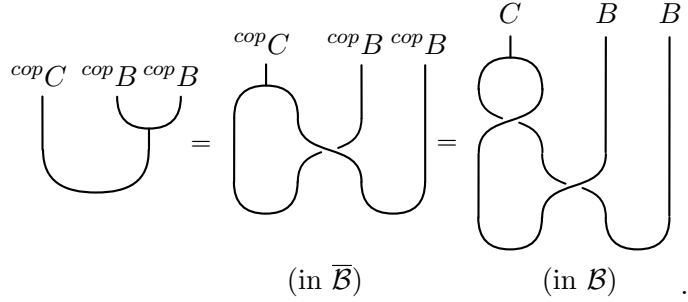.
\caption{Pairing of co-opposite Hopf algebras}
\label{strangepairing}
\end{center}
\end{figure}

\subsection{Acknowledgement}

I am grateful to my PhD advisor Dr Kobi Kremnizer for his guidance and help. I further like to thank Dr Yuri Bazlov and Prof Christoph Schweigert for teaching me the subject, and Prof Arkady Berenstein for helpful conversations. Diagrams are created using the package available at \url{http://www.paultaylor.eu/diagrams/}.

This research is supported by the EPSRC grant EP/I033343/1 \emph{Motivic Invariants and Categorification}\footnote{\url{http://gow.epsrc.ac.uk/NGBOViewGrant.aspx?GrantRef=EP/I033343/1}}.


\section{The Categorical Picture}\label{catback}

In this section, we introduce the two main categories of interest in this paper in purely categorical terms. The first one, the \emph{Drinfeld center} $\cZ(\cM)$ of a monoidal category, is well known. The other one, the \emph{Hopf center} $\cH(\cM)$ is well-known in the case where $\cM=\lmod{H}$ is the category of modules over an ordinary Hopf algebra where it can be described as the category of Hopf modules over $H$. For the more general case $\cM=\lmod{H}(\cB)$ where $\cB$ is a braided monoidal category and $H$ a bialgebra object in it, see e.g. \cite{Bes}. We present a new description of  $\cH(\cM)$ as a special case of a \emph{mixed} Drinfeld center construction.

In the strict monoidal case, applying techniques from reconstruction theory, which is a generalization of Tannaka-duality, one can recover certain categories $\cC\to \cV$,  where $\cV$ is monoidal, as module (or comodule) categories $\lmod{H}(\cV)\to \cV$ in the category $\cV$. This is used later to describe the categories $\cZ(\cM)$ and $\cH(\cM)$ in the case where $\cM=\rmod{B}(\lmod{H})$ for a bialgebra (or Hopf algebra) $B$ in the braided monoidal category $\lmod{H}$ as module categories leading to the definition of the \emph{braided} Drinfeld and Heisenberg doubles. For this, we need another description of the categories $\cZ(\cM)$ and $\cH(\cM)$ for $\cM=\lmod{B}(\cB)$ in terms of simultaneous modules and comodules over $B$ satisfying certain compatibility conditions, leading to Yetter-Drinfeld and Hopf modules in \ref{ydsection}. Finally, we give a categorical explanation of different concepts of quasitriangularity of a braided Hopf algebra in \ref{quasitriangularity}.


\subsection{Bimodules over Monoidal Categories}

In this section, we discuss the 2-category of $\cM$-bimodule morphisms over a monoidal category. This will later serve us to define relative Drinfeld and Heisenberg centers in \ref{relativecentersection}. One application is a natural strictification of a monoidal category (see Corollary \ref{strictification}).

We now work in the meta-2-category of categories $\Cat$ under suitable locally smallness assumptions. This category is symmetric monoidal with respect to the Cartesian product of categories denoted by $\times$. Let $(\cM$, $\otimes$) be a monoidal category (not necessarily strict, but strictly unital with unit object $I$). Such categories can be thought of as unitary monoid objects in $\Cat$. We denote the associativity isomorphism by $\alpha \colon \otimes(\otimes \times \ide)\Longrightarrow \otimes (\ide \times \otimes)$.

Recall that a \emph{braided} monoidal category $\cV$ is a monoidal category with a natural isomorphism $\Psi\colon \otimes\to \otimes^{op}$ which satisfies the $\otimes$-compatibilities
\begin{align}
\Psi_{V\otimes W, X}&=\alpha_{X,V,W}(\Psi_{V,X}\otimes \ide_W)\alpha^{-1}_{V,X,W}(\ide_V\otimes \Psi_{W,X})\alpha_{V,W,X},\\
\Psi_{V,W\otimes X}&=\alpha^{-1}_{W,X,V}(\ide_W\otimes \Psi_{V,X})\alpha_{W,V,X}(\Psi_{V,W}\otimes \ide_X)\alpha^{-1}_{V,W,X}.
\end{align}

Let $\cM$ and $\cV$ be monoidal categories. We further recall that a monoidal functor (or fiber functor) $G\colon \cM\to \cV$ is a functor such that there exists a natural isomorphism
\[
\mu=\mu_G\colon G\otimes \stackrel{\sim}{\Longrightarrow} \otimes(G\times G),
\]
which is compatible with the associativity isomorphism of $\cM$ and $\cV$, i.e. for any three objects $X$, $Y$ and $Z$ in $\cM$ we have
\begin{equation}
\alpha_{G(X),G(Y),G(Z)}(\mu_{X,Y}\times \ide_{G(Z)})\mu_{X\otimes Y,Z}=(\ide_{G(X)\times\mu_{Y,Z}})\mu_{X,Y\otimes Z}G(\alpha_{X,Y,Z}).
\end{equation}

Similarly to defining bimodules of unitary monoids (or rings), we can consider the category $\BiMod_\cM$ of $\cM$-bimodule objects in $\Cat$. Objects are categories which have a left and a right $\cM$-action which commute up to a coherent natural isomorphism. We usually denote the left action by $\triangleright \colon \cM \times \cV \to \cV$ and the right action by $\triangleleft \colon \cV\times \cM\to \cV$. The action coherences are natural isomorphisms
\begin{align*}
\chi\colon \triangleright(\ide_\cM\times \triangleright)&\stackrel{\sim}{\Longrightarrow}\triangleright(\otimes\times  \ide_\cV),\qquad \triangleright(I\times \ide_\cV)=\ide_\cV,\\
\xi \colon\triangleleft(\triangleleft \times \ide_\cM)&\stackrel{\sim}{\Longrightarrow}\triangleleft(\ide_\cV\times \otimes),\qquad \triangleleft(\ide_\cV\times I)=\ide_\cV.
\end{align*}
Observe that the modules are considered to be strictly unital mainly to simplify the exposition.
The coherence commuting the left and right action is a natural isomorphism
\[
\zeta\colon \triangleleft(\triangleright \times \ide_\cM)\stackrel{\sim}{\Longrightarrow} \triangleright(\ide_\cM \times \triangleleft).
\]
We require compatibilities between the coherences $\chi, \xi$ and $\zeta$. Without spelling them out in detail, the idea is that whenever two combinations of the functors $\otimes$, $\triangleright$ and $\triangleleft$ can be transformed into one another using different combinations of the transformations $\alpha, \chi, \xi$ and $\zeta$, then these different combinations have to be equal. This uses Mac Lane's coherence theorem \cite[VII.2]{Mac} which shows that elementary coherence axioms of minimal tensor order are sufficient to explain all coherences.

Morphisms in the category $\BiMod_\cM$ commute with the left and right actions up to coherent natural isomorphism. That is, a functor $F\colon \cV\to \cW$ is a morphism of bimodules if there exist natural isomorphisms
\[
\rho_F \colon F\triangleleft\stackrel{\sim}{\Longrightarrow}\triangleleft (F\times \ide_\cM),\qquad \lambda_F\colon F\triangleright \stackrel{\sim}{\Longrightarrow} \triangleright (\ide_\cM\times F).
\]
We often just write $\rho$, $\lambda$ if only one morphism $F$ is considered. These again have to be compatible with the natural isomorphisms $\alpha, \chi, \xi$ and $\zeta$.

The category $\BiMod_\cM$ has the structure of a 2-category. 2-morphisms are natural transformations $\tau\colon F\Rightarrow G$ of module morphisms which commute with the bimodule coherences, i.e. $\rho_G \tau_{\triangleleft}= (\triangleleft(\tau\times \ide_{\ide_\cM}))\rho_F\colon F\triangleleft \Rightarrow \triangleleft(G\times \ide_\cM)$ and $\lambda_G \tau_{\triangleright}= (\triangleright(\ide_{\ide_\cM}\times \tau))\lambda_F\colon F\triangleright \Rightarrow \triangleright (\ide_\cM\times G)$. It is helpful to write these conditions as commutative diagrams for any object $(X,M)\in \cV\times \cM$:
\[
\begin{array}{ccc}
\begin{diagram}
F(X\triangleleft M)& \rTo^{\tau_{X\triangleleft M}} &G(X\triangleleft M)\\
\dTo^{(\rho_F)_{X,M}}&&\dTo^{(\rho_G)_{X,M}}\\
F(X)\triangleleft M &\rTo^{\triangleleft(\tau_X\times \ide_M)}&G(X)\triangleleft M,
\end{diagram}
&\qquad&
\begin{diagram}
F(M\triangleright X)& \rTo^{\tau_{M\triangleright X}} &G(M\triangleright X)\\
\dTo^{(\lambda_F)_{M,X}}&&\dTo^{(\lambda_G)_{M,X}}\\
M\triangleright F(X) &\rTo^{\triangleright(\ide_M\times\tau_X)}&M\triangleright G(X).
\end{diagram}
\end{array}
\]

\begin{example}\label{bimoduleexamples}
$~$\nopagebreak
\begin{enumerate}
\item[(i)]
The regular $\cM$-bimodule $\cM^{\text{reg}}$ is defined using $\otimes\colon \cM\otimes \cM\to \cM$ as module structure (both left and right), and $\xi=\alpha$, $\chi=\alpha^{-1}$, $\zeta=\alpha$ as structure maps.
\item[(ii)]
The trivial $\cM$-bimodule on $\cM$ is given by $X\triangleright Y=Y$ and $X\triangleleft Y=Y$, and trivial action on morphisms too. We say that this bimodule is obtained by pulling the regular bimodule structure back along the functor $I\colon \cM\to \cM$ (factoring through the terminal and initial monoidal category $\cI$ with one element and morphism). Here, $\xi=\chi=\zeta=\ide$.
\item[(iii)]
More generally, for any pair of functors $G_1, G_2\colon \cM\to \cV$, we can give $\cV$ a $\cM$-bimodule structure $\leftexp{G_1}{\cV}^{G_2}$ where the left action is induced by pulling the regular bimodule structure on $\cV$ back along $G_1$, i.e.
$X\triangleright V=G_1(X)\otimes V,$
and the right action is induced by $G_2$ in the same way. For $\mu_i\colon G_i(\otimes)\Rightarrow \otimes (G_i\times G_i)$, we have
\begin{align*}
\chi&:= (\mu_{G_1}^{-1}\otimes \ide)\alpha^{\cV},
&\xi:&= (\ide\otimes \mu_{G_2}^{-1})(\alpha^{\cV})^{-1},
&\zeta&:= \alpha^{\cV}.
\end{align*}
\end{enumerate}
\end{example}

\begin{lemma}\label{monoidallemma}
Let $\cV$ be an $\cM$-bimodule. Then $\BiMod_\cM(\cV,\cV)$ is a strict monoidal category via composition of functors and composition of structure maps, i.e.
\begin{align*}
\rho^{\psi\phi}_{A,B}&=\rho^{\psi}_{\phi(A),B}\psi(\rho_{A,B}^{\phi}),
&\lambda^{\psi\phi}_{A,B}&=\lambda^{\psi}_{\phi(A),B}\psi(\lambda_{A,B}^{\phi}).
\end{align*}
\end{lemma}

\begin{lemma}\label{braidedlemma}
The category $\BiMod_{\cM}(\cM^{\op{reg}},\cM^{\op{reg}})$ is braided monoidal with braiding $\Psi_{\psi,\phi}$ given for an object $A$ of $\cM$ by
\[
\psi\phi(A)=\psi\phi(A\otimes I)\stackrel{\psi(\lambda_{A,I}^{\phi})}{\longrightarrow}\psi(A\otimes \phi(I))\stackrel{\rho^{\psi}_{A,\phi(I)}}{\longrightarrow}\psi(A)\otimes \phi(I)\stackrel{\lambda^{\phi}_{\psi(A),I}}{\longrightarrow}\phi(\psi(A)\otimes I)=\phi\psi(A).
\]
\end{lemma}

\begin{lemma}\label{forgetfulbimodule}
Consider an $\cM$-bimodule structure on $\cM$ itself where either the left or the right action is given by the regular action and denote this bimodule by $\cM'$. Then the functor
\[
F\colon \BiMod_{\cM}(\cM',\cM')\to \cM, \quad \phi\mapsto \phi(I)
\]
is monoidal.
\end{lemma}
\begin{proof}
Assume that the left $\cM$-action of $\cM'$ is regular. Then for two morphisms of bimodules $\phi, \psi \colon \cM' \to \cM'$,
\[
\psi\phi(I)=\psi(\phi(I)\otimes I)=\psi(\phi(I)\triangleright I)\stackrel{\lambda^{\psi}_{\phi(I),I}}{\longrightarrow} \phi(I)\triangleright \psi(I)=\phi(I)\otimes \psi(I)
\]
is an isomorphism showing that $F$ is monoidal.
\end{proof}

\begin{lemma}\label{bimodulesmitself}
For any monoidal category where $I$ is a terminal object, the functor $F$ from Lemma \ref{forgetfulbimodule} is part of an equivalence of categories
\[
\BiMod_\cM(\leftexp{\ide_\cM}{\cM}^I,\leftexp{\ide_\cM}{\cM}^I)\cong \cM.
\]
\end{lemma}
\begin{proof}
We define the inverse $\Ind\colon \cM\to \BiMod_\cM(\leftexp{\ide_\cM}{\cM}^I,\leftexp{\ide_\cM}{\cM}^I)$ by mapping an object $X$ to the functor $\Ind(X)$ given by
\[
Y\mapsto \Ind(X)(Y):=Y\otimes X=Y\triangleright X.
\]
For morphisms, we use the transformation $\Ind(F)_X:=\ide_X\otimes f$. The structure transformations are
\begin{align*}
\lambda^{\Ind(X)}_{A,B}&:=\chi^{-1}_{A,B,X}\colon \Ind(X)(A\triangleright B)\to A\triangleright \Ind(X)(B), \\
\rho^{\Ind(X)}_{A,B}&:=\ide_{A\otimes X}\colon \Ind(X)(A\triangleleft B)\to \Ind(X)(A)\triangleleft B.
\end{align*}

In order to show that the two constructions are mutually inverse to each other, we first show that for any $\phi \colon \leftexp{\ide_\cM}{\cM}^I \to \leftexp{\ide_\cM}{\cM}^I$, $\rho_{X,Y}=\ide_{\phi(X)}$. For this, consider the square
\begin{diagram}
\phi((X\triangleleft Y)\triangleleft X)&\rTo^{\rho_{X\triangleleft Y,Z}}&\phi(X\triangleleft Y)\triangleleft Z&\rTo^{\rho_{X,Y\triangleleft Z}}&(\phi(X)\triangleleft Y)\triangleleft Z\\
\dTo^{\ide}&&&&\dTo^{\ide}\\
\phi(X\triangleleft (Y\otimes Z))&&\rTo^{\rho_{X,Y\otimes Z}}&&\phi(X)\triangleleft (Y\otimes Z).
\end{diagram}
This square commutes by the coherence between $\rho$ and $\xi$ (which is $\ide$ in this case). Thus we obtain that
\begin{equation}
\rho_{X,Z}\rho_{X,Y}=\rho_{X\triangleleft Y,Z}\rho_{X,Y\triangleleft Z}=\rho_{X,Y\otimes Z}.
\end{equation}
Hence, as these are isomorphisms, $\rho_{X,I}=\ide_{\phi(X)}$.
Further, using the naturality square of $\rho$ in the second component, we obtain that $\rho_{X,Y}=\rho_{X,Y'}$ whenever there exists a morphism $f\colon Y\to Y'$. Under the assumption that $I$ is terminal, we always have an isomorphism $X\to I$ and hence $\lambda=\ide$.

Returning to the proof of the equivalence, it is clear that $F\Ind=\ide_{\cM}$. We claim that $\lambda_{-,I} \colon \ide\to \Ind F$ is a natural isomorphism. As $\lambda=\ide$, it remains to check that the square
\begin{diagram}
\phi(X\triangleright Y)&\rTo^{\lambda_{X,Y}}&X\triangleright \phi(Y)\\
\dTo^{\lambda_{X\otimes Y,I}}&&\dTo^{X\otimes \lambda_{Y,I}}\\
&&\\
{\Ind(\phi(I))(X\triangleright Y)\atop =(X\otimes Y)\triangleright \phi(I)}&\rTo^{\chi_{X,Y,\phi(I)}^{-1}}&{X\triangleright \Ind(\phi(I))(Y)\atop =X\triangleright (Y\triangleright \phi(I)}
\end{diagram}
commutes. But this follows from the coherence of $\chi$ and $\lambda$.
\end{proof}

It is not strictly necessary to assume that $I$ is terminal. It is sufficient that the graph of the category (objects as vertices, morphisms as edges) is connected. If, for example, $\cM=\lmod{B}(\cB)$ for $B$ a quasi-Hopf algebra object in $\cB$, then $\cM$ has $I$ as terminal object provided that $B$ does. If $\cM$ is additive, then the graph of the category is connected (via the zero morphisms). Note that
\[
\BiMod_\cM(\leftexp{\ide_\cM}{\cM}^I,\leftexp{\ide_\cM}{\cM}^I)\cong \Mod_{\cM}(\cM^{reg}, \cM^{reg})\cong \cM
\]
by the Lemma. The result can be interpreted as the following strictification:

\begin{corollary}\label{strictification}
Any monoidal category $\cM$ with connected graph is equivalent to a strict monoidal category.
\end{corollary}
\begin{proof}
We use the equivalence of Lemma \ref{bimodulesmitself}. By Lemma \ref{monoidallemma}, the monoidal category $\BiMod_\cM(\leftexp{\ide_\cM}{\cM}^I,\leftexp{\ide_\cM}{\cM}^I)$ is strict.
\end{proof}

\begin{example}
If $\cM$ is a monoidal category with connected graph, then
\[
\BiMod_{\cM}(\leftexp{I}{\cM}^I,\leftexp{I}{\cM}^I)\cong\mathbf{Fun}(\cM,\cM).
\]
This can be seen using the observation in the proof of Lemma \ref{bimodulesmitself} showing that under the  given assumption on $\cM$, $\rho=\lambda=\ide$. Hence any functor is a bimodule morphism for the trivial bimodule.
\end{example}

Let us fix an $\cM$-bimodules $\cV$ and consider the category of bimodule morphisms $\BiMod_\cM(\cM^{\text{reg}},\cV)$. To represent the data of a morphism $G\colon \cM^{\text{reg}} \to \cV$ in a more compact way, denote the image $G(I)$ of the $\otimes$-unit $I$ by $V$. Then it is sufficient -- by Proposition \ref{centerdata} below -- to consider the composite coherence (called \emph{centralizing isomorphism})
\[
c_X:=V\triangleleft X\stackrel{\rho^{-1}_{I,X}}{\to}G(I\otimes X)=G(X)=G(X\otimes I)\stackrel{\lambda_{X,I}}{\to}G(X)\otimes V,
\]
for any object $X$ of $\cM$. The natural isomorphism obtained this way will be denoted by $c\colon V \triangleleft\ide_\cM \to \ide_\cM\triangleright V$.
Note that $(V,c)$ is \emph{monoidal} in the sense that
\begin{equation}\label{monadicityc}
c_{X\otimes Y}=\chi_{X,Y}(c_X\otimes \ide_Y)\zeta_{X,V,Y}(c_X\otimes \ide_Y)\xi^{-1}_{X,Y}.
\end{equation}

A morphism of $\cM$-bimodules, $\vartheta\colon (V,c_V)\to (W,c_W)$ gives a morphism $\vartheta\colon V\to W$ satisfying that the square
\begin{diagram}
V\otimes G(X)&\rTo^{\vartheta\otimes \ide}&W\otimes G(X)\\
\dTo^{c_{V,X}}&&\dTo^{c_{W,X}}\\
G(X)\otimes V&\rTo^{\ide\otimes \vartheta}&G(X)\otimes W
\end{diagram}
commutes for any object $X$ of $\cM$.

We denote the set of such natural isomorphisms $c\colon V\triangleleft \ide_\cM\to \ide_\cM\triangleright V$ obeying the $\otimes$-compatibility (\ref{monadicityc}) by $\Isom^\otimes(V\triangleleft \ide_{\cM}, \ide_\cM\triangleright V)$. We use the notation
\begin{equation}\label{badnotation}
\Isom^{\otimes}(\cV\triangleleft \ide_\cM,\ide_\cM\triangleright \cV)
\end{equation}
to denote the category of pairs $(V,c)$, where $V$ varies over the objects $V\in \cV$, introduced above.

\begin{proposition}\label{centerdata}
There is an equivalence of categories
\[
\Isom^{\otimes}(\cV\triangleleft \ide_\cM,\ide_\cM\triangleright \cV)\cong \BiMod_\cM(\cM^{\text{reg}},\cV)
\]
\end{proposition}

\begin{proof}
First, we show that given a morphism of $\cM$-bimodules $G\colon \cM^{\text{reg}}\to \cV$, we can recover the data of $G$, $\rho_G$ and $\lambda_G$ from the pair $(V,c)$. We can set $G':= X\triangleright V$. Then $G'\cong G$ via $\rho_{X,I}^{-1}$. We can recover $\lambda_{X,Y}$ as
\[
\chi^{-1}_{X,Y,V}\colon G'(X\otimes Y)=(X\otimes Y)\triangleright V\to X\triangleright (Y\triangleright V)=X\triangleright G'(X).
\]
The natural isomorphism $\rho_{X,Y}$ can be recovered from $(V,c)$ by considering the composite $\zeta^{-1}_{X,Y}(X\triangleright c_Y)\chi_{X,Y}^{-1}$.

One checks that, given any pair $(V,c)$ as above, the procedure described in the first part gives an $\cM$-bimodule morphism $G_{(V,c)}$. Clearly, $G_{(V,c)}(I)=V$ and if we apply the above procedure to define the centralizing isomorphism, we recover $c$. The requirement of $(V,c)$ being monoidal implies the required compatibilities of $\rho$ and $\lambda$.
\end{proof}

\begin{remark}
If $\cB$ and $\cM$ are additive, abelian or $k$-linear, then these properties are inherited by the categories $\BiMod_\cM(\cV,\cW)$.
\end{remark}
In the next section, bimodule categories will be used to define different kinds of centers of monoidal categories.


\subsection{Mixed Relative Drinfeld Centers}\label{relativecentersection}

The Drinfeld center is a canonical way of associating a braided monoidal category to a monoidal category $\cM$. It can be defined, using the work of the previous subsection, via Hom-sets of $\cM$-bimodules. 

\begin{definition}\label{drinfeldcenter}
Let $G\colon \cM\to \cV$ be a monoidal functor. Then $\cV$ is a $\cM$-bimodule using the functor $G$, i.e. for objects $M\in \cM$, $V\in \cV$, we have $M\triangleright V:= G(M)\otimes V$. The right action is defined analogously and the resulting $\cM$-bimodule is denoted by $\cV^{G}$. We say that the action on $\cV^{G}$ is \emph{induced} by pullback of the regular bimodule structure on $\cV$ along $G$. This is $\leftexp{G}{\cV}^G$ from Example~\ref{bimoduleexamples}.

The \emph{Drinfeld center} of $\cM$ with respect to $G$ is defined as
\[
\cZ^G(\cM):=\Isom^{\otimes}(\cV\otimes G,G\otimes \cV).
\]
The special case $G=\ide\colon \cM\to \cM$ is denoted by $\cZ(\cM)$ and referred to as the \emph{Drinfeld center of $\cM$}. That is, $\cZ(\cM)=\Isom^{\otimes}(\cM\otimes \ide_\cM,\ide_\cM\otimes \cM)$.
\end{definition}

We want to emphasize two special cases that will be the main categories of interest in this paper: One is the Drinfeld center $\cZ(\cM)$ which is equivalent to $\BiMod_{\cM}(\cM^{\op{reg}},\cM^{\op{reg}})$ by \ref{centerdata}.
To define the other one, the Hopf center $\cH(\cM)$, we have to generalize the definition of the Drinfeld center in two different ways. On one hand, we will provide an appropriate relative setting with respect to a fiber functor $F\colon \cM\to \cB$. On the other hand, we will allow to use two functors $G_1$ and $G_2$ instead of just one. This however leads to the loss of monoidacity.

\begin{setting}\label{FPsetting}
First, we generalize to the setting over a braided monoidal category. Namely, we will always consider the monoidal category $\cM$ together with a fiber functor $F\colon \cM \to \cB$ where $\cB$ is a braided monoidal category, which has a left section $P \colon \cB \to \cM$ such that there exists $\tau\colon FP\stackrel{\sim}{\Longrightarrow} \ide_\cB$. We assume that $\tau$ is a monoidal natural transformation.

Further, we will consider \emph{mixed} centers. For this we assume given two monoidal functors $G_1,G_2\colon \cM\to \cV$ which factor through the fiber functor $F$, i.e. for $i=1,2$ we have commutative diagrams of functors
\begin{diagram}
&&\cM&&\\
&\ldTo^{G_i}&&\rdTo^{F}&\\
\cV&&\rTo^{F'}&&\cB.
\end{diagram}
\end{setting}

Note that for any monoidal category, we can always consider it over the trivial braided monoidal category $\cI$ with one element and one morphism. As this is both terminal and initial, we have unique functors $T\colon \cM\to \cI$ and $I\colon \cI\to \cM$, and $\tau=\ide_\cI\colon PF\cong IT$. 

\begin{definition}\label{mixedrelativedef}
In the same setting as above, the \emph{mixed relative Drinfeld center} of $\cM$ over $\cB$ w.r.t. $G_1$, $G_2$ is defined to be
\[
\leftexp{G_1}{\cZ}^{G_2}_{\cB}(\cM):=\Isom_\cB^\otimes(\cV\otimes G_2,G_1\otimes \cV),
\]
where $\Isom_\cB^\otimes(\cV\otimes G_2,G_1\otimes \cV)$ is the full subcategory of $\Isom^\otimes(\cV\otimes G_2,G_1\otimes \cV)$ on objects $(V,c)$ of $\Isom^\otimes(V\otimes G_2,G_1\otimes V)$ which are \emph{$(F,P)$-admissible}. That is,
\[
\Psi_{F'(V),X}=(\tau_{X}\times\ide_{F'(V)})\mu^{F'}_{G_1P(X),V}F'(c_{P(X)})(\mu^{F'}_{V,G_2P(X)})^{-1}(\ide_{F'(V)}\times \tau^{-1}_X),
\]
for each object $X$ of $\cB$, where $\Psi$ is the braiding on $\cB$. Note that we use $F'G_i=F$ for $i=1,2$ for these compositions to be well-defined.

If $G=G_1=G_2$, then we denote $\cZ^{G}_{\cB}(\cM):=\leftexp{G}{\cZ}^G_{\cB}(\cM)$ and refer to it as the \emph{relative Drinfeld center} of $\cM$ over $\cB$. If even $G_1=G_2=\ide_\cM$, then we denote
\[
\cZ_{\cB}(\cM):=\cZ^{\ide_\cM}_{\cB}(\cM)
\]
and refer to it as the \emph{Drinfeld center} of $\cM$ over $\cB$.
\end{definition}

We are mainly interested in the case where $G_1=G_2=\ide_{\cM}$ and $F=F'$ strictly monadic. In this situation, $(F,P)$-admissible means that the underlying morphism of the braiding with an object in the image of the functor $P$ is the same as the braiding in $\cB$. As special cases, we recover $\cZ^{G}(\cM)$ as defined in \ref{drinfeldcenter} as $\cZ_{(I,I)}^{G}(\cM)$. 

\begin{example}
Let $\cM=\rmod{B}(\cB)$ with $P=\triv$ the functor giving an object of $\cB$ the trivial module structure, $F$ the forgetful functor. Consider the category $\cZ^{\triv}_\cB(\cM)$. This category consists of objects $(V,\delta)$, where $V$ is a right $B$-module and $\delta$ is a left $B^{\oop}$-comodule, such that the action and coaction commute. Here, $\delta$ is obtained as $c_B(\ide_B\otimes 1)$ from the commutativity isomorphism.

If, to specify further, $B$ is a finite-dimensional bialgebra over $k$, then $\cZ^{\triv}_\cB(\cM)$ is equivalent to the category of modules over the bialgebra $B^*\otimes B$.
\end{example}

\begin{remark} If $\cB=\Vect$ for a quasitriangular Hopf algebra, many authors do not impose the admissibility condition as under certain representability conditions (e.g. $\cM=\lmod{B}$), any object in the center will be admissible for $F$ being the forgetful functor and $P=\triv$ the functor mapping a vector space to the trivial $B$-module on it. See e.g. \cite[Lemma 2.1]{Maj3} for such a proof, which relies on the existence of elements in a vector space. In our general setting, elements of objects do not exist, hence the admissibility assumption. This condition is a generalization of the assumption used in \cite[Section 3.6]{Bes}.

At the general level, admissibility will be crucial in the proof of Proposition~\ref{equivalence} which is the main result of Section \ref{ydsection} where we describe the centers in terms of Yetter-Drinfeld and Hopf modules.\end{remark}

We can apply this more general definition to the pair of functors $G_1=\ide_\cM$, and $G_2=PF$. The second center of interest in this paper can now be defined.

\begin{definition}\label{hopfcenter}
We define the \emph{relative Hopf center} for $G\colon\cM\to \cV$ to be
\[
\cH^G_{\cB}(\cM):=\leftexp{G}{\cZ}^{GPF}_{\cB}(\cM).
\]
In particular, the \emph{Hopf center} of $\cM$ over $\cB$ is
\[
\cH_{\cB}(\cM):=\cH^{\ide_\cM}_{\cB}(\cM).
\]
\end{definition}

In Section~\ref{doubles}, we explore the relationship between the categories $\cZ(\cM)$ and $\cH(\cM)$ in the case where $\cM=\rmod{B}(\lmod{H})$ is a category of right modules over a bialgebra (or Hopf algebra) $B\in \lmod{H}$, where $\lmod{H}$ is braided monoidal, using techniques from reconstruction theory. In this case, we have a fiber functor to $\Vect$, and $P=\triv$ is the functor mapping a vector space to the trivial $B$-module. We can proof the structural results about the Drinfeld and Hopf center at the level of generality of this section which is often easier. For instance, we show that the relative Drinfeld center has a natural monoidal structure\footnote{One can give the category $\cH^G_{\cB}(\cM)$ a monoidal structure which is not compatible with the forgetful functor, see e.g. \cite{Bes} in the case $\cM=\rmod{B}(\cB)$ by taking relative tensor products $\otimes_B$.} and that $\cZ_{\cB}(\cM)$ is braided.

\begin{proposition}\label{monoidalcenter}
For any monoidal functor $G\colon \cM \to \cV$, we can give $\cZ^G_{\cB}(\cM)$ a monoidal structure by setting $(V,c^V)\otimes (W,c^W):=(V\otimes W,c^{V\otimes W})$, where
\begin{equation}
c^{V\otimes W}:=\alpha_{G(X),V,W}(c^V\times \ide_W)\alpha^{-1}_{V,G(X),W}(\ide_V\times c^W)\alpha_{V,W,G(X)}.
\end{equation}
By construction, there is a strict monoidal functor $\cZ^G_{\cB}(\cM)\to \cM$.
\end{proposition}
\begin{proof}
This is shown by straightforward but lengthy checking of the axioms and compatibilities. The associativity isomorphism in $\cZ^G_{\cB}(\cM)$ is just the associativity isomorphism of $\cM$, which can be checked to be compatible with the centralizing isomorphisms of threefold tensor products.

Next, we have to check that $C^{V\otimes W}$ is monoidal in the sense of \ref{monadicityc}. The proof of this requires repeated application of naturality of $c^V$, $c^W$ to $\Delta$, the hexagonal axioms for $\alpha$, and naturality of $\alpha$ with respect to $c^V$ and $c^V$. \end{proof}

The main advantage of the Drinfeld center is that it is braided.

\begin{proposition}\label{braidedcenter}
The category $\cZ_{\cB}(\cM)$ has a braiding defined by
\[
\Psi=\Psi_{(V,c^V),(W,c^W)}:=c^V_W\colon (V\otimes W,c^{V\otimes W})\to (W\otimes V,c^{W\otimes V}),
\]
for objects $(V, c^V)$ and $(W, c^W)$ in $\cZ_{\cB}(\cM)$.
\end{proposition}
\begin{proof}
This follows by applying naturality of $c^V$ to the morphisms $c^W_X$ for any object $X$ of $\cM$. Indeed, this gives that $c^V_{W\otimes X}(\ide_V\times c_X^W)=(c^W_X\times \ide_V)c^V_{W\otimes X}$. Now we use the monadicity of $c$ as in (\ref{monadicityc}) twice giving the required commutative square. It is clear that the tensor product of two $(F,P)$-admissible objects is again $(F,P)$-admissible.
\end{proof}

We define the category $\BiMod^\cB_\cM(\cM^{\op{reg}},\leftexp{G_1}{\cV}^{G_2})$ as the full subcategory of bimodules $\cM^{\op{reg}}\to\leftexp{G_1}{\cV}^{G_2}$ that corresponds to $\Isom^\otimes(\cV\otimes G_2,G_1\otimes \cV)$ under the equivalence of Proposition \ref{centerdata}.

\begin{theorem}\label{fundamentalequivalence}
There is an equivalence of monoidal categories
\[
\BiMod^\cB_\cM(\cM^{\op{reg}},\cV^G)\cong \cZ^G_\cB(\cM),
\]
where $\BiMod^\cB_\cM(\cM^{\op{reg}},\cV^G)$ is a monoidal category via composition of functors. If $\cV^G=\cM^{\op{reg}}$, this is an equivalence of braided monoidal categories.
\end{theorem}
\begin{proof}
Recall that by Proposition \ref{centerdata} there is an equivalence of categories for the larger categories without the $(F,P)$-admissibility requirement. The left hand side is by definition the subcategory corresponding to $(F,P)$-admissible objects under this equivalence. To show that the monoidal structure defined in \ref{monoidalcenter} corresponds to the monoidal structure of composition of functors on the left hand side, observe that for $\phi, \psi \colon \cV^{G}\to \cV^{G}$, the commutativity isomorphism of the composition $\phi\psi$ is given by
\begin{align}
c^{\phi\psi}_X&=(\lambda_{\phi})_{X,\psi(I)}\psi((\lambda_{\psi})_{X,I})\psi((\rho^{-1}_{\psi})_{I,X})(\rho^{-1}_\phi)_{\psi(I),X}.
\end{align}
In \ref{centerdata} we saw that for $\phi$ corresponding to $(V,c)$ and $\psi$ corresponding to $(W,d)$, we have $\phi(\psi(I))=\phi(W)=V\otimes W$, and $\rho^{-1}$ is given by $\alpha$ working with the regular action for simplicity. Using these equalities, we find
\begin{align}
c^{\phi\psi}_X&=(\lambda_{\phi})_{X,\psi(I)}(\phi(\ide_X\otimes c^{\psi}_X))(\rho^{-1}_\phi)_{\psi(I),X}=c^{V\otimes W}_X,
\end{align}
comparing with the definition of the monoidal structure in $\cZ^G_\cB(\cM)$ from \ref{monoidalcenter}. The fact that the braidings are related in the special case $\cV^G=\cM^{\op{reg}}$ can be checked by a similar calculation.
\end{proof}

Note that the monoidal category $\BiMod^\cB_{\cM}(\cM^{\op{reg}}, \cM^{\op{reg}})$ is strict as composition of functors is strictly associative, and the additional datum of compositions of the transformations $\rho$ and $\lambda$ is strictly associative too. However, the reinterpretation of the data of bimodule morphisms as pairs $(V,c)$ yields a non-strict monoidal category if $\cM$ is not strict.

Theorem \ref{fundamentalequivalence} gives an easy way to find module categories over $\cZ_\cB(\cM)$ from bimodule categories over $\cM$. The resulting categorical actions is the main topic of this section and will be reinterpreted in various reformulations of the braided Drinfeld and Heisenberg double in the course of this paper. The general statement is:

\begin{corollary}\label{bimoduleaction}
Let $\cV$ be an $M$-bimodule. Then there exists a natural action by composition of functors
\[
\triangleright\colon \cZ_\cB(\cM)\times \BiMod_\cM(\cM^{\op{reg}},\cV)\longrightarrow \BiMod_\cM(\cM^{\op{reg}},\cV).
\]
In particular, for $\cV$, $G_i$ as in \ref{FPsetting}, this restricts to an action
\[
\triangleright\colon \cZ_\cB(\cM)\times \leftexp{G_1}{\cZ}^{G_2}_\cB(\cM)\longrightarrow \leftexp{G_1}{\cZ}^{G_2}_\cB(\cM).
\]
\end{corollary}
\begin{proof}
The first statement follows by using the action of $\BiMod^{\cB}_{\cM}(\cM^{\op{reg}}, \cM^{\op{reg}})$ on $\BiMod_\cM(\cM^{\op{reg}},\cV)$ by composition and relating it to the Drinfeld center by \ref{fundamentalequivalence}. Note that there is no dependence on $\cB$ in this statement. For the second part, it is easy to check that the action obtained as a special case of the first statement, restricts to the $(F,P)$-admissible subcategory $\BiMod^{\cB}_{\cM}(\cM^{\op{op}}, \leftexp{G_1}{\cV}^{G_2})$.
\end{proof}

Given that $\cV$ has (left) duals (that, is the category $\cV$ is \emph{rigid}), we can show that the relative Drinfeld centers inherit the same structure.

\begin{proposition}\label{generalduals}
Let $\cV$ be a rigid category. Then the center $\cZ^{G}_{\cB}(\cM)$ is rigid with the dual of an object $(V,c)$ given by $V^*$ with commutativity isomorphism
\begin{equation}
c^*_{X}:=(\ev_V\times \ide)(\alpha^{-1}\times \ide_{V^*})(\ide\times (c^{V}_X)^{-1}\times\ide)(\alpha\times \ide)\alpha^{-1}(\ide_{V^*\otimes G(X)}\times \coev_V).
\end{equation}
The coevaluation and evaluation morphisms are given by the respective maps in $\cM$.
\end{proposition}
\begin{proof}
Note first that a monoidal structure is necessary to talk about rigidity of a category. Hence the restriction $G_1=G_2$. Now $c^*$ is a well defined natural transformation $V^*\otimes G\to G\otimes V^*$. It is easy to see that $c^*$ is $\otimes$-compatible and thus $(V^*,c^*)$ gives an element of $\cZ^{G}(\cM)$. Moreover $\ev_V$ and $\coev_V$ commute with $c\otimes c^*$ and $c^*\otimes c$. We will sketch in more detail how commutativity with $\ev_V$ can be proved. This will be straightforward after proving that
\begin{equation}\label{keystepduals}
(\ev_V\otimes \ide_X)\alpha^{-1}_{V^*,V,X}(\ide_{V^*}\otimes c_X^{-1})\alpha_{V^*,X,V}=(\ide_X\otimes \ev_V)\alpha_{X,V^*,V}(c^*_X\otimes \ide_V).
\end{equation}
Starting from the right hand side of the equation, using the definition of $c^*$, we extract the expression $(\ide_V\otimes \ev_V)\alpha^{-1}(\coev_V\otimes \ide_V)$. This is done using naturality of $\alpha^{-1}$ in $\coev$, $\ev$ or $c^{-1}_X$ and the hexagonal axioms. This expression equals $\ide_V$ by rigidity of $\cM$.

It remains to check that the dual (according to Proposition~\ref{generalduals}) of an admissible object is admissible again. For this, note that for any $X$ of $\cB$ and $(F,P)$-admissible $(V,c)$ in the center, we have that $F(c^{-1}_{P(X)})$ can be expressed in terms of the inverse braiding and structural isomorphisms. From this we can conclude admissibility of $(V^*,c^*)$.
\end{proof}

The next lemma shows how to extract the inverse of the centralizing isomorphism from the definition of the dual.

\begin{lemma}\label{centerinvertible}
If $\cM$ has duals, then the inverse can be described using the definition of the dual as
\begin{equation}
c^{-1}_X=(\ide\otimes (\ide\otimes \ev))(\ide_V\otimes\alpha)(\ide\otimes c^*_X\otimes \ide)(\ide_V\otimes \alpha^{-1})\alpha_{V,V^*,X\otimes V}(\coev \otimes \ide_{X\otimes V}).
\end{equation}
\end{lemma}
\begin{proof}
Key in the proof is to use (\ref{keystepduals}). We first extract the right hand side of this equation in the right hand side of the claim. After applying (\ref{keystepduals}), we use naturality of $\alpha^{-1}$ in $c_X^{-1}$ and $\coev$ as well as the hexagonal axioms to transform the resulting composition of maps into
\[
((\ide_V\otimes \ev_V)\otimes \ide_X)(\alpha\otimes \ide_X)((\coev\otimes \ide_V)\otimes \ide_X)c_X^{-1}=c_X^{-1}.\qedhere
\]
\end{proof}

In fact, one can show that if $\cM$ has right duals, then these can be used to show that every natural transformation $V\otimes \ide_\cM\to \ide_\cM$ is automatically invertible. This fact will be used in Section \ref{ydsection}. For this, recall that the right dual $\leftexp{*}{V}$ for an object $V$ of $\cM$ is an object together with morphisms $\ev_V'\colon V\otimes \leftexp{*}{V}\to I$ and $\coev'_V\colon I \to \leftexp{*}{V}\otimes V$, such the axioms
\begin{align}\label{rightdualaxioms}
(\ev_V'\otimes \ide_V)\alpha^{-1}_{V,\leftexp{*}{V},V}(\ide_V\otimes\coev'_V)&=\ide_V,\\
(\ide_{\leftexp{*}{V}}\otimes \ev'_V)\alpha_{\leftexp{*}{V},V,\leftexp{*}{V}}(\coev'_V\otimes \ide_{\leftexp{*}{V}})&=\ide_{\leftexp{*}{V}},
\end{align}
are satisfied.

\begin{lemma}\label{cinversebyduals}
Let $\cM$ have right duals. Then for $(V,c)\in \cZ_\cB(\cM)$,
\begin{equation}
c_X^{-1}=(\ev'\otimes \ide_{V\otimes X})\alpha^{-1}(\ide\otimes\alpha)(\ide_X\otimes c_{\leftexp{*}{X}}\otimes \ide_{\leftexp{*}{X}})(\ide\otimes \alpha^{-1})(\ide_{X}\otimes (V\otimes \coev'_X)).
\end{equation}
\end{lemma}
\begin{proof}
Applying the hexagonal axiom and naturality of $\alpha^{-1}$ in $\coev'$ and $c_X$ we obtain
\begin{align*}
c_X&(\ev'\otimes \ide_{V\otimes X})(\alpha^{-1})(\ide\otimes\alpha)(\ide_X\otimes c_{\leftexp{*}{X}}\otimes \ide_{\leftexp{*}{X}})(\ide\otimes \alpha^{-1})(\ide_{X}\otimes (V\otimes \coev'_X))\\
=&(\ev'\otimes \ide)(\alpha^{-1}\otimes \ide)\alpha^{-1}(\ide\otimes \alpha^{-1})(\ide_{X\otimes \leftexp{*}{X}}\otimes c_X)\\&(\ide\otimes\alpha)(\ide_X\otimes c_{\leftexp{*}{X}}\otimes \ide_{\leftexp{*}{X}})(\ide\otimes \alpha^{-1})(\ide_{X}\otimes (V\otimes \coev'_X)).
\end{align*}
Applying first the definition of the monoidal rule (\ref{monadicityc}) and naturality of $c_{\leftexp{*}{X}}$ in $\coev'$, followed by the right dual axioms this expression becomes
\begin{align*}
(\ev'\otimes \ide)&(\alpha^{-1}\otimes \ide)\alpha^{-1}(\ide_X\otimes c_{\leftexp{*}{X}\otimes X})(\ide_{X\otimes V}\otimes \coev')\\
=&((\ev_X'\otimes \ide_X)\otimes \ide_V)(\alpha^{-1}\otimes \ide_V)((\ide_X\otimes \coev'_X)\otimes \ide_V)=\ide_{X\otimes V}.\qedhere
\end{align*}
\end{proof}

At this general level, we can show that the relative Drinfeld center acts on the relative Hopf center. We will later use this result to obtain twisting results on the level of algebras (cf. \ref{cocycletwists}).

\begin{theorem}\label{cataction}
There is a left action of the monoidal category $\cZ^{G}_{\cB}(\cM)$ on $\cH^{G}_{\cB}(\cM)$ defined by $(V,c)\triangleright (W,d)=(V\otimes W, d^{V\triangleright W})$, where
\begin{align}
c^{V\triangleright W}:&=\alpha_{G(X),V,W}(c_X\times \ide_W)\alpha^{-1}_{V,G(X),W}(\ide_V\times d_X),
\end{align}
for $(V,c)\in \cZ^{G}(\cM)$ and $(W,d)\in \cH^{G}_\cV(\cM)$.
\end{theorem}
\begin{proof}
This is a special case of Corollary \ref{bimoduleaction}, using the translation under Proposition \ref{centerdata} and the monoidal equivalence Theorem \ref{fundamentalequivalence}.

The result can also easily be seen directly using the monoidal structure introduced in (\ref{monoidalcenter}). The left action isomorphism $\chi$ will simply be the associativity isomorphism $\alpha$ in $\cV$. It is clear that with this action, the resulting object will again be $(F,P)$-admissible.
\end{proof}

In particular, there is a natural action $\triangleright \colon \cZ_{\cB}(\cM)\times \cH_{\cB}(\cM)\to \cH_{\cB}(\cM)$.

\begin{remark}
If $\cM$, $\cV$, $\cB$ and all functors carry additional structure such as being additive, abelian (with left or right exact functors), or $k$-linear, then these structures are inherited by the mixed relative Drinfeld centers. The constructions discussing in this section however work in the generality of braided monoidal category without such structures. In Section~\ref{doubles} a $k$-linear setting is used.
\end{remark}


\subsection{Braided Reconstruction Theory}\label{reconstruction}

We now recall a categorical reconstruction theorem which generalizes Tannaka-Krein duality -- which is classically stated in the setting of a monoidal category $\cC$ over $\Vect$ -- to the setting of a monoidal functor $F\colon \cC\to \cB$ where $\cB$ is any braided monoidal category which we tread as strictly monoidal (identifying all ways of setting the brackets), but keep track of the associativity isomorphisms in $\cC$. We assume that the functor $F$ is strictly monoidal\footnote{Dropping this assumption leads to weak quasi-Hopf algebras. A reconstruction theorem for such objects can be found in \cite{Hae}.}. More details about this can be found in \cite[Section~9.4.2]{Maj1} (or \cite[3.2]{Maj6} for a comodule version) for the braided Hopf algebra case, and  for the quasi-Hopf algebra case (but not braided) see \cite[Section~2]{Maj5} for a comodule version. 

In order to state the reconstruction theorem, we need to assume certain representability conditions on the functor $F\colon \cC\to \cB$. The basic assumption is that the functor $\Nat_\cC(-\otimes F,F)\colon \cB^{\op{op}}\to \Set$ is representable. This means there exists an object $B\in \cB$ such that
\begin{equation}\label{representability}
\Nat_\cC(V\otimes F,F)\cong \Hom_\cB(V,B), \quad \forall V\in \cB.
\end{equation}
For example, if $\cB=\Vect$, $\cC=\lmod{B}$ and $F$ the forgetful functor, we have
\[
\Nat_{\lmod{B}}(F,F)\cong\Nat_{\lmod{B}}(k\otimes F,F)\cong \Hom_{\Vect}(k,B)\cong B,
\]
recovering $B$. To recover classical Tannaka-Krein duality, one considers the vector space $\Nat_{\lmod{B}}(F,F)$ for $\cC\to \fVect$ under suitable assumptions (see e.g. \cite{Del}).

\begin{theorem}
Let $F\colon \cC\to \cB$ be a functor satisfying (\ref{representability}) with respect to some object $B$ in $\cB$. Then $B$ is an algebra object in $\cB$ and $F$ factors as $\cC\to B\text{-}\Mod(\cB)\stackrel{F}{\to} \cB$ where $F$ is the forgetful functor.

The object $B$ is universal in the following sense: If $B$ is another such algebra object in $\cB$, then there exists a unique morphism $B'\to B$ such that the induced pullback functor $\lmod{B}(\cB)\to \lmod{B'}(\cB)$ makes the diagram
\begin{equation}\label{universalrec}
\begin{diagram}
&&\cC &&\\
&\ldTo &&\rdTo &\\
\lmod{B}(\cB)&\rTo&&&\lmod{B'}(\cB)\\
&\rdTo &&\ldTo &\\
&&\cB &&
\end{diagram}
\end{equation}
commute.
\end{theorem}
\begin{proof}[Proof (Sketch)]
Indeed, the natural transformation $\sigma\colon B\otimes F\to F$ corresponding to $\ide_B$ under (\ref{representability}) gives an action of $H$ on each object $F(X)$ for $X\in \cC$. The product map $B\otimes B\to B$ corresponds to the natural transformation given by
\[
\sigma_X(\ide_B\otimes \sigma_X) \colon B\otimes B\otimes F(X)\to F(X),
\]
for $X\in \cC$. The unit is given by the natural transformation $I\otimes F\to F$ of the unit of the monoidal structure of $\cB$.
The morphism $\sigma_X$ gives any object $X$ of $\cC$ a $B$-modules structure, which we will sometimes denote by $\triangleright$.
\end{proof}

In order to obtain more structure on $B$, we need to assume more structure on $\cC$ and that this structure is preserved by the functor $F$. In addition, we will need \emph{higher representabilities}. Given a morphism $\beta\colon V\to B^{\otimes n}$ we can define a natural transformation $\theta_V^n(\beta)$ as the composition
\begin{equation}
\begin{diagram}
V\otimes F(X_1)\otimes\ldots\otimes F(X_n)&\rTo^{\beta\otimes \ide}&B^{\otimes n}\otimes F(X_1)\otimes\ldots\otimes F(X_n)\\\dTo^{\theta_V^n(\beta)}&&\dTo\\F(X_1)\otimes\ldots\otimes F(X_n)&\lTo^{\sigma_{X_1}\otimes\ldots\otimes \sigma_{X_n}}&B\otimes F(X_1)\otimes \ldots\otimes B\otimes F(X_n),
\end{diagram}
\end{equation}
where the right vertical arrow is obtained by the braiding in $\cB$.

\begin{definition}\label{higherrep}
We say that a functor $F\colon \cC\to \cB$ is \emph{higher representable} if the functors $\Nat(V\otimes F^{\otimes n},F^{\otimes n})$ are representable for $n\geq 0$ by the object $B^{\otimes n}$ such that a morphism $\beta$ corresponds to $\theta_V^n(\beta)$.
\end{definition}

This condition says that the representing object $B$ for $\Nat(V\otimes F,F)$ induces representability of $\Nat(V\otimes F^{\otimes n},F^{\otimes n})$ by $B^{\otimes n}$. This condition is not automatic in general. In a classical Tannaka-Krein duality setting, working with $\cB=\fVect$, it will be automatically satisfied.

\begin{theorem}\label{bialgebrareconstruction}
Let $F\colon \cM\to \cB$ be a monoidal functor (not necessarily strict) satisfying the higher representability conditions for some object $B\in \cB$. Then $B$ is a universal quasi-bialgebra object in $\cB$ such that $F$ factors as $\cM\to B\text{-}\Mod(\cB)\stackrel{F}{\to} \cB$.
\end{theorem}
\begin{proof}
The coproduct $\Delta$ is the morphism $B\to B\otimes B$ corresponding to the natural transformation $\delta \colon B\otimes F^2\to F^2$ defined by
$\delta_{X,Y}= \sigma_{X\otimes Y}.$
For the counit, we define $F^{\otimes 0}$ to be the constant functor $I$ with image the unit object $I$ in $\cB$. The data of a natural transformation $V\otimes F^{\otimes 0}\to F^{\otimes 0}$ consists of only one morphism $V\to I$. The counit is defined to be the morphism $a_I\colon B\to I$. 
The 3-cycle $\phi\colon I \to B^{\otimes 3}$ corresponds to the natural transformation
$F(\alpha_{X,Y,Z})$,
coming from the associativity isomorphism in $\cM$. The quasi-coassociativity of $\Delta$ now follows (under translation with use of the higher representability condition) from the commutativity of the square
\begin{equation}
\begin{diagram}
B\otimes F((X\otimes Y)\otimes Z)&\rTo^{\ide_B\otimes F(\alpha_{X,Y,Z})}&B\otimes F(X\otimes (Y\otimes Z))\\
\dTo^{\sigma_{(X\otimes Y)\otimes Z}}&&\dTo^{\sigma_{X\otimes (Y\otimes Z)}}\\
F((X\otimes Y)\otimes Z)&\rTo^{F(\alpha_{X,Y,Z})}&F(X\otimes (Y\otimes Z)),
\end{diagram}
\end{equation}
for object $X,Y,Z$ of $\cM$, which uses naturality of $\sigma$.
Moreover, the hexagonal axiom translates to the 3-cycle condition. 
The proof that $\Delta$ is an algebra homomorphism in $\cB$ uses naturality of the braiding (see \cite[Figure~9.16(b)]{Maj1}). Note that also for quasi-bialgebras, $\Delta$ is an algebra homomorphism, i.e. the bialgebra condition holds strictly (not up to isomorphism).
\end{proof}

If we assume even more structure on the category $\cM$ and $\cB$, we obtain more structure on the representing bialgebra $B$. We will need the following preliminary observation regarding the interplay of left and right duals in $\cB$:

\begin{lemma}\label{rightdualslemma}
If $\cB$ is a braided monoidal category (with associativity isomorphism $\alpha$) which is rigid (i.e. left duals exist). Then the left dual of an object $V$ is also a right dual.
\end{lemma}
\begin{proof}
Recall that a right dual $\leftexp{*}{V}$ for an object $V$ of $\cB$ is an object together with morphisms $\ev_V'\colon V\otimes \leftexp{*}{V}\to I$ and $\coev'_V\colon I \to \leftexp{*}{V}\otimes V$, such that (\ref{rightdualaxioms}) is satisfied.
Using the braiding on $\cB$, we define $\ev_V':=\ev_V\Psi^{-1}$, and $\coev'_V:=\Psi\coev_V$. This gives $V^*$ the structure of a right dual.
\end{proof}

In order to recover the antipode of $B$, we need to assume the existence of a natural \emph{duality isomorphism} $d_X\colon F(X)^*\to F(X^*)$. If $\cM$ is strict ($\alpha=\ide$), then $d_X=(\ev_{F(X)}\otimes \ide_{F(X^*)})(\ide_{F(X)^*}\otimes F(\coev_{X}))$. 
We require that the compatibility condition $d_{X\otimes Y}=d_Y\otimes d_X$ of the monoidal structure with the duality holds.

We say the functor $F$ is rigid if $d$ exists, and for the evaluation and coevaluation morphisms in $\cM$, the conditions
\begin{equation}
\ev_{\cM}=\ev_{F(X)}(d_X\otimes \ide), \qquad \coev_{\cM}=(\ide\otimes d_X^{-1})\coev_{F(X)},
\end{equation}
are satisfied for any object $X$ of $\cM$.

\begin{theorem}\label{majidreconstruction}$~$\nopagebreak
\begin{enumerate}
\item[(a)]Let $\cB$ be rigid. Then $\cM$ and the functor $F$ are rigid if and only if the representing object $B$ is a quasi-Hopf algebra object in $\cB$.\label{hopfreconstruction} 
\item[(b)] \label{rightdualreconstruction} Let $\cB$ be a braided monoidal and rigid. Then $\cM$ has left and right duals on the same object, and $F$ is rigid, if and only if $S$ has an invertible antipode.
\item[(c)]\label{braidedreconstruction} Let $\cB$ be braided monoidal. Then $\cM$ is braided monoidal if and only if we can define a second coproduct $\Delta^{\op{cop}}$ (see Definition~\ref{oppositecoproduct}) and a universal $R$-matrix turning $B$ into a quasitriangular quasi-bialgebra (respectively quasi-Hopf algebra if we are in case (a)) in $\cB$.
\end{enumerate}
\end{theorem}

In the following, we will explore these structures more concretely and sketch the proofs. For part (a), we first define the map $a$ as the map $I\to B$ corresponding to the natural isomorphism
\[
(\ide\otimes F(\ev_X))(\ide\otimes d_X\otimes \ide)(\coev_{F(X)}\otimes \ide),
\]
while $b$ is defined using
\[
(\ide\otimes \ev_{F(X)})(\ide\otimes d_X^{-1}\otimes \ide)(F(\coev_X)\otimes \ide).
\]
This implies that
\begin{align}
F(\ev_{F(X)})&=\ev_{F(X)}(\ide\otimes \sigma_{X})(\ide\otimes a \otimes \ide),\\
F(\coev_{F(X)})&=\sigma_X\otimes \ide(b\otimes \ide_{F(X)\otimes F(X)^*})\coev_{F(X)}.
\end{align}
The dual action is defined as $\sigma^*_X:=d_X^{-1}\sigma_{X^*}(\ide\otimes d_X)$.
We then define the antipode as the morphism $S\colon B\to B$ corresponding to the natural transformation $B\otimes F\to F$ defined for an object $X$ of $\cM$ as
\begin{equation}\label{antipotetrans}
(\ide\otimes \ev_{F(X)})(\ide\otimes d_X^{-1})(\ide_{F(X)}\otimes \sigma_{X^*}\otimes \ide)(\Psi\otimes d_X\otimes \ide)(\ide_B\otimes\coev_{F(X)}\otimes \ide_{F(X)}).
\end{equation}
That is, translating the action $\sigma_{X^*}$ on the dual to an action on $F(X)$ using conjugation by $d_X$. We directly derive a formula for translating between the action on $F(X)$ and the dual action:
\begin{equation}\label{dualactionev}
\ev_{F(X)}(\triangleright^*\otimes \ide_{F(X)})=\ev_{F(X)}(\ide\otimes \triangleright)(\Psi\otimes \ide_{F(X)})(S\otimes \ide_{F(X^*)\otimes F(X)}).
\end{equation}

We can now easily proof the antipode axioms (\ref{quasiantipode}).
We also check directly that the duality conditions in $\lmod{H}(\cB)$ are equivalent to the conditions (\ref{antipodecyclecond}). This completes the proof that having an antipode for $B$ is equivalent to the existence of left duals in $\cM$ via reconstruction.

For part (b), we need the following Lemma regarding the antialgebra and coalgebra morphism properties of the antipode.

\begin{lemma}\label{antipodelemma}
The antipode $S$ satisfies
\begin{equation}
S m = m\Psi(S\otimes S), \qquad \Delta S= (S\otimes S)\Psi\Delta.
\end{equation}
If an inverse $S^{-1}$ exists, then it satisfies
\begin{equation}
S^{-1}m = m\Psi^{-1}(S^{-1}\otimes S^{-1}), \qquad \Delta S^{-1}= (S^{-1}\otimes S^{-1})\Psi^{-1}\Delta.
\end{equation}
\end{lemma}
\begin{proof}
It is an exercise to adapt the proof of \cite[Figure 9.14]{Maj1} to quasi-Hopf algebras. This proof uses (\ref{antipodecyclecond}) for $\phi^{-1}$.
\end{proof}

Using this Lemma, we can further observe conditions on $S^{-1}$ which are equivalent to the antipode axioms (\ref{quasiantipode}):
\begin{align}
m^2(\ide \otimes a\otimes S^{-1})\Psi^{-1}\Delta&=a\varepsilon,
&m^2(S^{-1}\otimes b\otimes \ide)\Psi^{-1}\Delta&=b\varepsilon.
\end{align}

We now turn to the proof of (b): In Lemma~\ref{rightdualslemma}, we saw that $\cB$ has right duals. Given that the antipode $S$ is invertible, we can use the left duals in $\cM$ to define right duals. The right dual $\leftexp{*}{F(X)}$ is defined to be $F(X)^*$ with (co)evaluation maps
\begin{align}
\ev'&=\ev(\ide\otimes \sigma)(\ide\otimes S^{-1}a \otimes \ide)\Psi^{-1},\\
\coev'&= \Psi(\sigma\otimes \ide)(S^{-1}b\otimes \ide)\coev.
\end{align}
It is not hard to check that the maps $\ev'_{F(X)}$, $\coev'_{F(X)}$ give the left dual $F(X)^*$ a right dual structure. The proofs use Lemma \ref{antipodelemma} and Lemma \ref{rightdualslemma}.

Conversely, given right duals $\leftexp{*}{X}$ in $\cM$ on the same objects, we view the natural transformation $d_X$ as $d_X\colon \leftexp{*}{F(X)}\to F(\leftexp{*}{X})$. We then apply reconstruction to the natural transformation
\[
(\ide\otimes \ev_{F(X)})(\ide\otimes d_X^{-1}\otimes \ide)(\ide\otimes \sigma_{\leftexp{*}{X}})(\Psi^{-1}\otimes d_X\otimes \ide_{F(X)})(\ide_{B}\otimes \coev_{F(X)}\otimes \ide_{F(X)}),
\]
to give a map $S'\colon B \to B$. If we define the \emph{right} dual action as
\begin{equation}
\leftexp{*}{\triangleright}:=d_X^{-1}\sigma_{\leftexp{*}{X}}(\ide_B\otimes d_X).
\end{equation}
Then we can derive the following property:
\begin{align}
\ev_{F(X)}(\leftexp{*}{\triangleright }\otimes \ide_{F(X)})&=\ev_{F(X)}(\ide_{F(X)^*}\otimes \triangleright)(\ide_{F(X)^*}\otimes S'\otimes \ide_{F(X)})(\Psi\otimes \ide_{F(X)}).
\end{align}
Further, the morphisms
\begin{align}
\coev^r_\cM&:=\Psi(\ide\otimes \leftexp{*}{\triangleright})(\ide\otimes b\otimes \ide)\coev_{F(X)},\\
\ev^r_\cM&:=\ev_{F(X)}(\leftexp{*}{\triangleright}\otimes \ide )(a\otimes \ide_{F(X)\otimes F(X)^*})\Psi^{-1}
\end{align}
are morphisms of left $B$-modules in $\cB$. 
We can now show that the map $S'$ is inverse to the antipode $S$ recovered from the left module structure in $\cM$. One checks that the right duality axioms correspond to the conditions (\ref{antipodecyclecond}) after application of $S$.

For the readers convenience, we include the definitions of the antipode (and its inverse) via reconstruction theory using graphical calculus in Figure~\ref{antipode}. This generalizes \cite[Figure  9.15]{Maj1}, where the definition of $m$, $\Delta$, $1$, $\varepsilon$ are as given there. Further, the duality structure on $\cM=\lmod{B}(\cB)$ is given in Figure \ref{leftmodduals}

\begin{figure}
\[
\begin{array}{cc}
\vcenter{\hbox{\small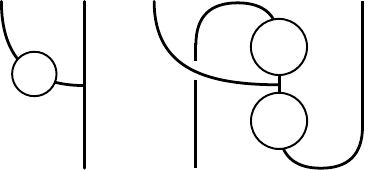}},\qquad
&
\vcenter{\hbox{\small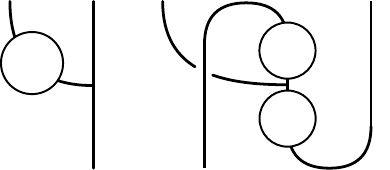}}
\end{array}
\]
\caption{Antipode reconstruction}
\label{antipode}
\end{figure}

\begin{figure}
\[
\begin{array}{cccc}
\ev_\cM^l=\vcenter{\hbox{\small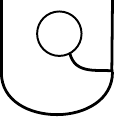}}~,
&
\coev_\cM^l=\vcenter{\hbox{\small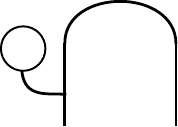}}~,
&
\ev_\cM^r=\vcenter{\hbox{\small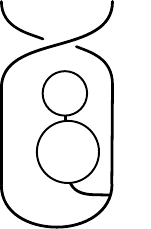}},
&
\coev_\cM^r=\vcenter{\hbox{\small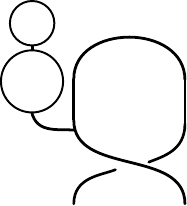}}~.
\end{array}
\]
\caption{Duals of left modules}
\label{leftmodduals}
\end{figure}

To describe the structure on $B$ obtained from the braiding, and thus proof (c), we will need a general definition of an universal $R$-matrix in a braided monoidal category. For this, we have to consider an \emph{opposite} coproduct $\Delta^{\op{cop}}$.

\begin{definition}\label{oppositecoproduct}
The opposite coproduct $\Delta^{\op{cop}}$ is defined to be the morphism $B\to B\otimes B$ in $\cB$ corresponding to the natural transformation $\tau$ given by
\begin{equation}
\tau_{X,Y}=\Psi^{-1}_{F(X),F(Y)}\sigma_{Y\otimes X}(\ide\otimes \Psi_{F(X),F(Y)}).
\end{equation}
That is, the action of the opposite coproduct satisfies the identity
\begin{equation}
\begin{split}&(\triangleright_V\otimes \triangleright_W)(\ide_B \otimes \Psi\otimes\ide_{V\otimes W})(\Delta^{\op{cop}}\otimes \ide_{V\otimes W})\\&=(\triangleright_V\otimes \triangleright_W)(\ide_B \otimes \Psi^{-1}\otimes \ide_{V\otimes W})(\Psi^{-1}\Delta\otimes \ide_{V\otimes W}).\end{split}
\end{equation}
Dually, we define the \emph{opposite product} $m^{\op{op}}$.
\end{definition}

We express this diagram using graphical calculus in Figure~\ref{oppositecoproductpic} (the opposite coproduct $\Delta^{\op{cop}}$ is labeled with $\text{cop}$ in the diagram).
\begin{figure}
\begin{center}
\small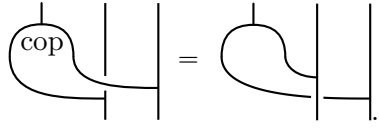.
\end{center}
\caption{The opposite coproduct}
\label{oppositecoproductpic}
\end{figure}

\begin{remark}In the symmetric monoidal case, $\Delta^{\op{cop}}$ is the usual opposite coproduct $\Psi\Delta$, but this does not hold in a general braided monoidal category $\cB$. The two following lemmas, which will be needed in Section \ref{ydsection}, explain the relationship in the general case.
\end{remark}

\begin{lemma}Let $B$ be a quasi-bialgebra in $\cB$ with coassociator $\phi$. 
\begin{enumerate}
\item[(a)]Denote by $B^{\coop}$ the object $B$ equipped with the same product but opposite coproduct $\Delta^{\op{cop}}$. Then $B^{\coop}$ is a quasi-bialgebra object in $\cB$ (with coassociator $\phi^{-1}$). By definition,
\[
\lmod{B^{\coop}}(\cB)\cong \lmod{\leftexp{\op{cop}}{B}}(\overline{\cB}),
\]
is an isomorphism of monoidal categories. If $B$ is a quasi-Hopf algebra, then so is $B^{\coop}$ with the same antipode $S$ (and the same elements $a$, $b$ as $B$).
\item[(b)] Denote by $B^{\oop}$ the object $B$ with the same coproduct but opposite product $m^{\op{op}}$. Then $B^{\oop}$ is a quasi-bialgebra object in $\cB$ (with coassociator $\phi$). By definition,
\[
\lcomod{B^{\oop}}(\cB)\cong \lcomod{\leftexp{\op{op}}{B}}(\overline{\cB}).
\]
If $B$ is a Hopf algebra, then so is $B^{\oop}$ with the same antipode $S$ (and the same $a$, $b$ as $B$).
\end{enumerate}
\end{lemma}
\begin{proof}
This is an exercise in graphical calculus.
\end{proof}

We need to introduce further structure via reconstruction to state the axioms for a universal $R$-matrix in this general setting. Similar to the definition of the opposite coproduct, we introduce \emph{twisted coassociators}. For any element $\sigma \in S_3$, we define $\phi^{\ov{\sigma}}$ by reconstruction corresponding to the transformation
$\Psi_\sigma^{-1} F(\alpha)\Psi_\sigma,$
where $\Psi_\sigma$ is the composition of $\Psi$s acting on two of the three tensor components according to a minimal expression of $\sigma$ as a product of transpositions of adjacent indices. For example, if $\sigma=(13)=(12)(23)(12)$, then $\Psi_\sigma=(\Psi\otimes \ide)(\ide\otimes \Psi)(\Psi\otimes \ide)$. Such an expression of $\sigma$ is not unique, but the resulting $\Psi_\sigma$ is independent of choice.

After these preliminary observations and notations, we turn back to part (c) of the proof of \ref{majidreconstruction}. The universal $R$-matrix $R$ corresponds to the natural transformation $\varrho$ given by
\[
\varrho_{X,Y}=\Psi^{-1}_{F(X),F(Y)}F(\Psi^{\cM}_{X,Y})\colon F^{\otimes 2}\to F^{\otimes 2}.
\]
It satisfies the following axioms, obtained from the braiding axioms (in $\cM$):
\begin{align}
m_{B\otimes B}(R\otimes \Delta)&=m_{B\otimes B}(\Delta^{\op{cop}}\otimes R),\label{braidingaxiom1}\\
(\ide\otimes \Delta)R&=m_{B^{\otimes 3}}^5((\phi^{-1})^{\ov{(123)}}\otimes (\ide\otimes 1\otimes \ide)R\otimes \phi^{\ov{(12)}}\otimes (R\otimes 1)\otimes \phi^{-1}),\label{braidingaxiom2}\\
(\Delta^{cop}\otimes \ide)R&=m_{B^{\otimes 3}}^5(\phi^{\ov{(13)}}\otimes (1\otimes R)\otimes (\phi^{-1})^{\ov{(123)}}\otimes ((\ide\otimes 1\otimes \ide)R)\otimes \phi^{\ov{(12)}}).\label{braidingaxiom3}
\end{align}

\begin{remark}
For the purpose of Sections~\ref{doublessection} and \ref{twistedchapter}, we either have $B=\Vect$ which is symmetric monoidal, or $\phi=1\otimes 1\otimes1$, so these axioms simplify. We include this generality for completeness.
\end{remark}

We are particularly interested in cases for which the functor $\cM\to \lmod{B}(\cB)$ is an equivalence. For this reason, we mention a version of the reconstruction theorem similar to Tannaka-Krein duality \cite[Theorem~2.1]{Maj7} (where comodules are used).

\begin{theorem} \label{finrec}
Let $\cM$ be a rigid monoidal category which is $k$-linear over $\fVect$ and equivalent to as small one. Then there is a finite-dimensional quasi-Hopf algebra $H$ over $k$ with invertible antipode such that $\cM\cong H\text{-} \Mod(\fVect)$.
\end{theorem}

This theorem can be viewed as a special case of Theorems~\ref{bialgebrareconstruction} and \ref{majidreconstruction}. As before, if $\cM$ is braided, then $H$ is quasitriangular. The following result may be useful when addressing the question whether $\cM \cong \lmod{B}(\cB)$:

\begin{lemma}If the functor $F\colon \cM \to \cB$ is faithful, then the functor $\cM\to \lmod{B}(\cB)$ is fully faithful.
\end{lemma}


\subsection{Yetter-Drinfeld and Hopf Modules}\label{ydsection}

For the reminder of this section, let $\cM=\rmod{B}(\cB)$ where $B$ is a quasi-bialgebra object in $\cB$.
We want to reinterpret the categories $\cZ(\cM)$ and $\cH(\cM)$ in more familiar terms leading to the definition of Yetter-Drinfeld and Hopf modules. 

Note that the forgetful functor $F\colon \cM\to \overline{\cB}$ always has a section $\triv$ given by the functor mapping an object $X$ in $\cB$ to the trivial module $X^{\triv}$ on it (with action given by the counit of $\cB$). This functor is the identity on morphisms. In the following, we will often omit writing the functor $F$.

\begin{remark}
We are considering right $B$-modules in this section to make a description of the center in terms of right $B$-modules and left $C$-modules more convenient (for $C$ dually paired with $B$) in Section \ref{pairedydsect}.

A priori, it seems that we can choose whether to consider the forgetful functor $F$ as mapping to $\cB$ or $\overline{\cB}$. It turns out that it is necessary to use $\overline{\cB}$ for the description of the center in terms of YD -modules even though this choice may seem less natural.
\end{remark}

\begin{warning}\label{rightmodwarning}
The reconstruction theory in Section \ref{reconstruction} for quasi-Hopf algebras is \emph{not} symmetric with respect to switching to right modules. Right action by the coassociator $\phi$ gives the \emph{inverse} associativity isomorphism $\alpha^{-1}$ rather than $\alpha$. This happens because for right modules we read the action of a product of elements from left the right, while for left modules from right to left, and the 3-cycle condition (\ref{3cocycle}) is \emph{not} left-right-symmetric. Note that this problem does not occur in braided Hopf algebra reconstruction (the case $\phi=1$) as then all conditions are symmetric.

This has the following consequences for reconstruction of duals in $\rmod{B}(\cB)$:
\begin{align}
\ev_{\cM}^l&=\ev_{F(X)}(\ide_{F(X)^*} \otimes \triangleleft)(\ide_{F(X)^*\otimes F(X)} \otimes S^{-1} b),\\\coev_{\cM}^l&=(\triangleleft\otimes \ide_{F(X^*)})(\ide_{F(X)}\otimes S^{-1}a\otimes \ide_{F(X)^*})\coev_{F(X)},\\
\ev_\cM^r&=\ev_{F(X)}(\ide_{F(X)^*}\otimes \triangleleft)(\ide_{F(X)^*\otimes F(X)}\otimes b)\Psi^{-1},\\
\coev_\cM^r&=\Psi(\triangleleft \otimes \ide_{F(X)^*})(\ide_{F(X)}\otimes a\otimes \ide_{F(X)^*})\coev_{F(X)}.
\end{align}
\end{warning}
These left and right (co)evaluations in $\rmod{B}(\cB)$ are depictured in Figure \ref{rightmodduals}
\begin{figure}[htb]
\[
\begin{array}{cccc}
\ev_\cM^l=\text{\begin{minipage}{0.13\linewidth}
\small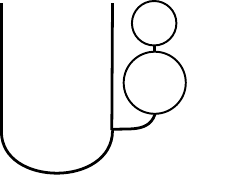
\end{minipage}},
&
\coev_\cM^l=\text{\begin{minipage}{0.08\linewidth}
\text{\small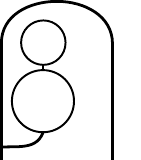}
\end{minipage}},
&
\ev_\cM^r=\text{\begin{minipage}{0.125\linewidth}
\small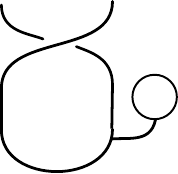
\end{minipage}},
&
\coev_\cM^r=\text{\begin{minipage}{0.08\linewidth}
\text{\small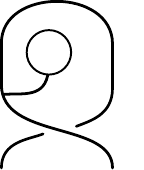}
\end{minipage}}.
\end{array}
\]
\caption{Right module duals}
\label{rightmodduals}
\end{figure}

\begin{definition}\label{quasiyddef}Let $B,C$ be quasi-bialgebras in a braided monoidal category $\cB$ with a morphism $G\colon B\to C$ of quasi-bialgebras. 
\begin{enumerate}
\item[(a)]
Define the category of \emph{Yetter-Drinfeld modules} over $(B,C)$ in $\cB$, denoted $\leftsub{C^{\oop}}{\mathcal{YD}}^B(\cB)$, as having objects $V$ of $\cB$ with a right action $\triangleright$ of $B$, together with a map $\delta\colon V\to C\otimes V$ (called \emph{quasi}-coaction) which satisfies the rules given using graphical calculus in Figure~\ref{quasiydmods}.
\begin{figure}[htb]
\[
\begin{array}{cccc}
\text{\begin{minipage}{0.43\linewidth}
\scalebox{1.20}[1.20]{\footnotesize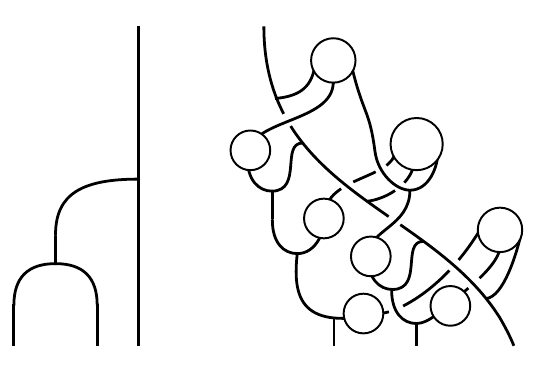}
\end{minipage}}
&,&
\text{\begin{minipage}{0.45\linewidth}
\scalebox{1.20}[1.20]{\footnotesize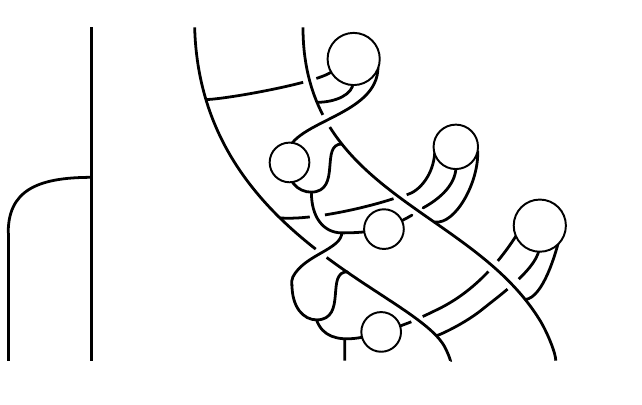}
\end{minipage}}&
\end{array}
\]
\caption{Quasi-comodule conditions}
\label{quasiydmods}
\end{figure}

In the case where $B$ and $C$ are \emph{strict} bialgebras (i.e. with trivial coassociator $\phi$), these rules give that $\delta$ is a left $C^{\oop}$-coaction in $\cB$. Further, the two structures $\delta$ and $\triangleleft$ satisfy the \emph{Yetter-Drinfeld condition}
\begin{equation}\label{ydcondition}
(m \otimes \ide)(G\otimes \delta)\Psi^{-1}(\triangleleft \otimes \ide)(\ide\otimes\Delta)=(m \otimes \ide)(\ide\otimes G\otimes\triangleleft)(\ide\otimes \Psi \otimes \ide)(\delta\otimes \Delta).
\end{equation}
Morphisms in $\leftsub{C^{\oop}}{\mathcal{YD}}^B(\cB)$ are those commuting with the $B$-module and $C^{\oop}$-quasi-comodule\footnote{This definition is \emph{not} the same as the quasi-modules in \cite{BB}.} structure.
\item[(b)]Further, define the category of \emph{Hopf modules} over $(B,C)$ in $\cB$, denoted by $\leftsub{C^{\oop}}{\mathcal{H}}^B(\cB)$, which consists of objects of $\cB$ with a right action and a  map $\delta$ as above (satisfying the first rule in Figure \ref{quasiydmods}, such that the \emph{Hopf condition}
\begin{equation}\label{hopfcondition}
\delta\triangleleft=(m \otimes \ide)(\ide\otimes G\otimes\triangleleft)(\ide\otimes \Psi \otimes \ide)(\delta\otimes \Delta)
\end{equation}
is satisfied.
Again, morphisms commute with the action and quasi-coaction.
\end{enumerate}
\end{definition}

Note that in the case where $B$ and $C$ are strict bialgebras, the Hopf module condition can be reformulated by saying that $\delta$ is a morphism of right $B$-modules, where $C$ is a $B$-module via the action induced by $G\colon B \to C$. It is helpful to use graphical calculus to visualize the different conditions in Figure \ref{conditionspic}. The first picture displays the compatibility condition for YD-modules, the second one for Hopf modules. The reason why $C^{\oop}$ appears instead of $C$ is to make the category of YD-modules into a monoidal category (see e.g \cite[Lemma~3.3.2]{Bes} for a direct proof). This fails when using $C$. Note that the coassociator $\phi$ only appears in the comodule and monoidal rule for YD-modules but not in the compatibility condition.
\begin{figure}
\begin{center}
\small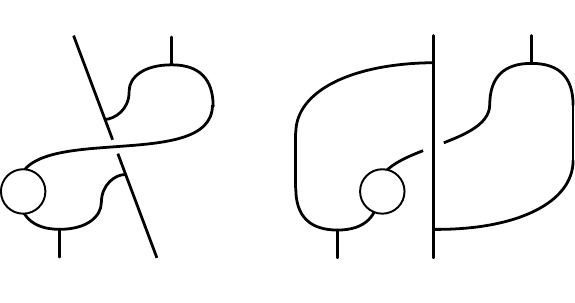,\hspace{1cm} \small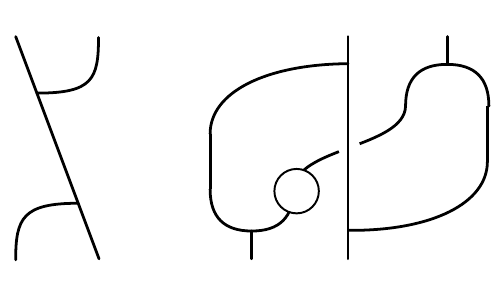.
\end{center}
\caption{Right module YD- and Hopf-compatibility conditions}
\label{conditionspic}
\end{figure} 

Considering left instead of right $B$-modules we obtain the category $\leftexpsub{B}{C}{\mathcal{YD}}(\cB)$ with compatibility conditions from Figure \ref{leftconditionspic}.
\begin{figure}
\begin{center}
\small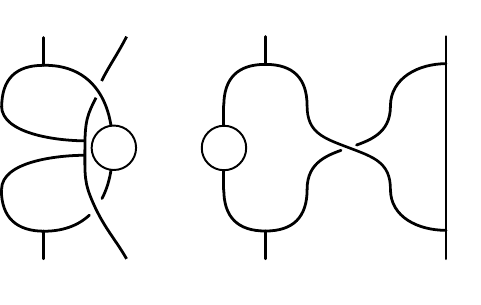,\hspace{1cm} \small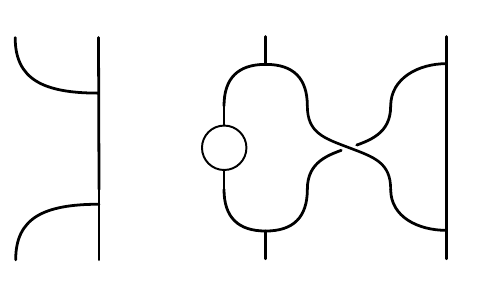
\end{center}
\caption{Left module YD- and Hopf-compatibility conditions}
\label{leftconditionspic}
\end{figure}

\begin{proposition}\label{ydbraidings}Let $B$ be a quasi-bialgebras in $\cB$.
\begin{enumerate}
\item[(a)]
The category $\leftsub{B^{\oop}}{\mathcal{YD}}^B(\cB)$ is monoidal (combining the monoidal structures of $\rmod{B}(\cB)$ and the second rule in Figure \ref{quasiydmods}) and \emph{pre}-braided\footnote{That is, the braiding is not necessarily a natural isomorphism.} with pre-braiding given by 
\begin{equation}
\leftsub{B^{\oop}}\Psi^B:=(\triangleleft \otimes \ide_V)(\ide_W\otimes \delta)\Psi^{-1}.
\end{equation}
The forgetful functor $\leftsub{B^{\oop}}{\mathcal{YD}}^B(\cB)\to \overline{\cB}$ is a functor of pre-braided monoidal categories.
\item[(b)]
The category $\leftexpsub{B}{B}{\mathcal{YD}}(\cB)$ is monoidal and pre-braided with pre-braiding
\begin{equation}
\leftexpsub{B}{B}{\Psi}:=(\triangleright \otimes \ide)(\ide\otimes \Psi)(\delta\otimes \ide).
\end{equation}
The forgetful functor $\leftexpsub{B}{B}{\mathcal{YD}}(\cB)\to \cB$ is one of pre-braided monoidal categories.
\end{enumerate}
\end{proposition}
\begin{proof}
These statements will all be proved in \ref{equivalence}.
Using graphical calculus (and \ref{mainidentity}), the braidings are given in Figure \ref{braidingsfigure}.\end{proof}
\begin{figure}
\[
\begin{array}{ccccccc}
\leftsub{B^{\oop}}{\Psi}^B&=&\vcenter{\hbox{\small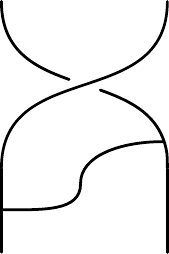}}~, &\qquad&\leftexpsub{B}{B}{\Psi}&=&\vcenter{\hbox{\small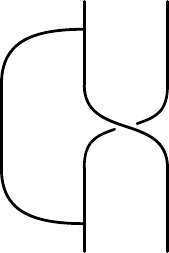}}~.
\end{array}
\]
\caption{The braidings of YD-modules}
\label{braidingsfigure}
\end{figure}

\begin{corollary}$~$
\begin{enumerate}
\item[(a)] The forgetful functor $\leftsub{B^{\oop}}{\mathcal{YD}}^B(\cB)\longrightarrow \rmod{B}(\cB)$ is a strict monoidal functor of pre-braided monoidal categories.
\item[(b)] The forgetful functor $\leftexpsub{B}{B}{\mathcal{YD}}(\cB)\longrightarrow\lmod{B}(\cB)$ is a strict monoidal functor of pre-braided monoidal categories.
\end{enumerate}
\end{corollary}

\begin{lemma}\label{centerinvertiblebetter}
Let $B,C$ be quasi-Hopf algebras in $\cB$ with a morphism $G\colon B\to C$ preserving the structure. For any YD-module $V$ over $(B,C)$ with quasi-coaction $\delta$, and any $C$-module $X$, the map
\[
c(\delta)_X\colon  V\otimes X\to X\otimes V, \quad c(\delta)_X:=(\triangleleft\otimes \ide)(\ide\otimes \delta)\Psi^{-1}_{X,V}
\]
has an inverse $c(\delta)_X^{-1}\colon  X\otimes V\to V\otimes X$ which equals $(\ide_V\otimes \triangleleft)(\Psi\otimes \ide_C)(\ide_X\otimes \delta')$,
for a \emph{right} $C$-quasi-coaction $\delta'$. We give the formula for $\delta'$ in Figure \ref{rightquasicoaction}.
\end{lemma}
\begin{figure}
\[
\vcenter{\hbox{\scalebox{1.20}[1.20]{\footnotesize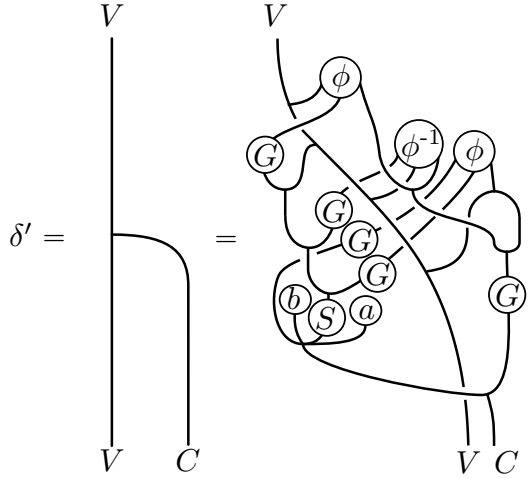}}}
\]
\caption{The right quasi-coaction $\delta'$}
\label{rightquasicoaction}
\end{figure}

\begin{proof}
We first show that $c(\delta)^{-1}$ is a right inverse to $c(\delta)$. For this, it suffices to show that $c(\delta)_C \delta'=1\otimes \ide_V$. Insert $1\otimes 1\otimes 1=(\ide\otimes S \otimes \ide)((\Delta G)\otimes \ide_{B}\otimes G)\phi$ into the right hand side of this equation $\delta'$ so that the right hand side of (\ref{3cocycle}) appears. Next, apply the 3-cycle condition and note that, using Figure \ref{quasiydmods}, $m_{B^{\otimes 3}}(S\otimes a \otimes \ide)(\Delta\otimes \ide_V)\delta=a$ appears. The expression then simplifies to $1 \otimes \ide_V$ using (\ref{antipodecyclecond}).

To show that $c(\delta)^{-1}$ is a left inverse, it suffices to show that $c(\delta)^{-1}_C\delta=\ide_C\otimes 1$. This is done in a similar way to the first part (using $3$-cycle manipulations and the antipode axioms).
\end{proof}

\begin{proposition}\label{equivalence}
Let $\cB$ be a braided monoidal category, $B$ a quasi-Hopf algebra in $\cB$. Further assume that
\[
G \colon \rmod{C}(\cB)\to \rmod{B}(\cB)=\cM
\]
be a monoidal functor factoring through the forgetful functor $F$. Then the category $\cZ^{G}_\cB(\cM)$ is isomorphic as a monoidal category to the category of YD-modules $\leftsub{C^{\oop}}{\mathcal{YD}}^B(\cB)$, and the category $\cH^{G}_\cB(\cM)$ is isomorphic to the category of Hopf modules $\leftsub{C^{\oop}}{\cH}^B(\cB)$.
If $G=\ide_\cM$, then the isomorphism is one of braided monoidal categories.
\end{proposition}

\begin{proof}
As $G$ factors through $F$, $G$ is strict monoidal. Further observe that by the universal property of reconstruction (\ref{universalrec}) the functor $G$ is equivalent to a morphism of quasi-bialgebras $B\to C$ which we also denote by $G$.

First, consider the case of $\cZ^{G}_\cB(\cM)$. We recall that all objects in $\cZ^{G}_\cB(\cM)$ are $(F,\triv)$-admissible. To fix notation, we write $\Nat^\otimes(V\otimes G,G\otimes V)$ for monoidal transformations satisfying (\ref{monadicityc}) which are not necessarily invertible. Using $\Nat$ instead of $\Isom$ in the notation of (\ref{badnotation}), we can define a functor
\[
\delta(-)\colon \Nat_\cB^\otimes (\cV\otimes G, G\otimes \cV) \to \leftsub{C^{\oop}}{\mathcal{YD}}^B(\cB)
\]
by mapping $(V,c)$ to $V$ with the quasi-coaction $\delta(V,c) \colon V\to C\otimes V$ defined as $c_B(\ide_V\otimes 1)$. On morphisms, $\delta$ is just the identity. Conversely, define $c(\delta)$ for a coaction $\delta$ on $V$ to be the natural transformation $V\otimes G \to G\otimes V$ defined by $c(\delta)_V:= (\triangleleft \otimes \ide_V)(\ide_{G(X)}\otimes G\otimes \ide_V)(\ide_{G(X)}\otimes \delta)\Psi^{-1}_{V,G(X)} \colon V\otimes G(X)\to G(X)\otimes V$ for any object $X$ of $\cM$, using the inverse braiding $\Psi^{-1}$ in the base category $\cB$. 

We have to show that the functors $\delta$ are well-defined (i.e. map to the categories claimed). For this, we first verify that under the functor $\delta(-)$, the requirement that $c$ is $\otimes$-compatible gives that $\delta(c)$ is a $C$-quasi-comodule and vice versa. Further we have to show that the Yetter-Drinfeld condition (\ref{ydcondition}) for $\delta(V,c)$ and the $B$-action on $V$ corresponds precisely the condition that $c_B(\ide\otimes 1)$ is a morphism of $B$-comodules. The key observation to prove this is the identity
\begin{equation}\label{mainidentity}
c_B=(m_B\otimes \ide_V)(G\otimes \ide_B\otimes \ide_V)(\ide \otimes \delta(c))\Psi^{-1}_{V,B}.
\end{equation}
Note that this identity crucially depends on $(F,P)$-admissibility. 
It is computed by observing that $\ide_V=\triangleleft (\ide_V\otimes 1)$, where we view $\triangleleft$ as a morphism of $B$-modules $V^{\triv}\otimes B\to V$. The equality arises if we apply naturality of $c$ to the regular action $\triangleleft\colon B^{\triv}\otimes B\to B$. Note that all the associativity isomorphisms (which enter the picture as right action by $\phi$) vanish because they act on $B^{\triv}$. In order to prove the observation above, we use \ref{monoidalcenter} and rewrite it using (\ref{mainidentity}).

The mappings $\delta$ and $c$ are mutually inverse: It is clear that $\delta((V,c(\delta)))=\delta$ as $\Psi(\ide\otimes 1)=1\otimes \ide$. To verify that $c(\delta(V,c))=c$ we use that the action $\triangleleft$ is a morphism of right $B$-modules $V^{\triv}\otimes B\to V$, where $B$ has the regular action. Applying naturality of $c$, monadicity of $c$, and the assumption that $(V,c)$ is admissible over $\overline{\cB}$ implies that $c(\delta(V,c))=c$. This establishes that $\delta(-)$, with inverse $c(-)$, form an isomorphism of categories. 

For $c(\delta)$ to give an object of the center, it needs to be invertible. This is not true for general quasi-bialgebras, but can be assured using the assumption that the antipode $S$ exists. For this, we use Lemma~\ref{centerinvertiblebetter}. Hence $\cZ_\cB^G\cong \leftsub{C^{\oop}}{\mathcal{YD}}^B(\cB)$ is an isomorphism of categories via the mutually inverse functors $\delta(-)$ and $c(-)$.

Looking at the Hopf-center $\cH^G_{\cB}(\cM)$, we note that the one can run a analogous argument to establish an isomorphism with the category of Hopf modules $\leftsub{C^{\oop}}{\cB}$. In this case, the compatibility condition obtained is (\ref{hopfcondition}).

Next, restricting to the case where $G=\ide_{\cM}$, we note that the monoidal structure and braiding of $\leftsub{C^{\oop}}{\mathcal{YD}}^B(\cB)$ are precisely the ones induced by the monoidal structure of $\cZ^{G}_\cB(\cM)$ thus making the isomorphism $\delta$ an isomorphism of braided monoidal categories. Monadicity of the equivalence can also be show for general $G$.

If $B$ is only a quasi-bialgebra, the equivalence does not hold as stated, as $\otimes$-compatible natural transformations are not necessarily invertible. We however still obtain a pre-braiding on the category of YD-modules. This establishes Proposition \ref{ydbraidings} of which we had postponed the proof. The key observation for translating properties and constructions in $\cZ^G_{\cB}(\cM)$ to the description in terms of Yetter-Drinfeld modules is (\ref{mainidentity}). Using this, the monoidal structure in the Drinfeld center (cf. \ref{monoidalcenter}) translates precisely to the claimed monoidal structure on the category of Yetter-Drinfeld modules (which comes from the coaction on $B$ and the second axiom in Figure \ref{quasiydmods}). This completes the proof of the isomorphism of braided monoidal categories $\cZ^G_{\cB}(\cM)\cong \leftsub{C^{\oop}}{\mathcal{YD}}^B(\cB)$.
\end{proof}

The case we will be most interested in is $G=\ide_\cM$. In this case the category $\cZ_{\cB}(\cM)$ is equivalent to the category of YD-modules over $B$ in $\cB$, and the category $\cH_{\cB}(\cM)$ is equivalent to the category of Hopf modules over $B$ in $\cB$.

Similarly, for left $B$-modules, the categories $\cZ_\cB(\lmod{B}(\cB))$ and $\leftexpsub{B}{B}{\mathcal{YD}}(\cB)$, and the corresponding Hopf-versions, are isomorphic as braided monoidal categories. Consider the isomorphism of categories
\[
\rmod{B}(\cB)\to \lmod{B^{\coop}}(\cB), \quad (V,\triangleleft) \mapsto (V,\triangleright:=\triangleleft\Psi^{-1}(S^{-1}\otimes \ide_V)).
\]
As the center constructions are stable under isomorphism, we find that
\begin{equation}
\cZ_\cB(\lmod{B}(\cB))\cong \cZ_\cB(\rmod{B^{\coop}}(\cB)),
\end{equation}
or the corresponding equivalence for the Hopf centers. 

\begin{remark}
It is important to note for later applications that Proposition~\ref{equivalence} does not rely on $\cB$ being rigid. This will enable us to work with the category of countably infinite-dimensional vector spaces later.
\end{remark}

A fundamental theorem, proved in this general form in \cite{BD}, states that provided that in the category $\cB$ equalizers split, there is an equivalence of categories $
\cH_\cB(\rmod{B}(\cB)) \cong \cB$,
for $B$ a Hopf algebra object in $\cB$. We can recover part of this statement in the quasi-case:

\begin{theorem}
Assume that $\cB$ has split antipodes and let $B$ be a quasi-Hopf algebra in $\cB$. Then there exists a functor
\begin{align*}
\Res\colon \cH_\cB(\rmod{B})\stackrel{\sim}{\longrightarrow}\cB,\qquad 
V\mapsto V^B,
\end{align*}
with right inverse given by the fully faithful functor
\begin{align*}
\Ind\colon &\cB\rightarrow\cH_\cB(\rmod{B}), &V\mapsto &(B\otimes V, (\triangleleft \otimes \ide_V)(\ide_B\otimes \Psi),\delta \otimes \ide_V),&g\mapsto \ide_B\otimes g.
\end{align*}
\end{theorem}
\begin{proof}
First check that the functor $\Ind$ gives Hopf modules as stated, which is an easy exercise in graphical calculus. Next, it is possible to see that the functor $\Ind$ is fully faithful. To prove fullness, use that each morphism $f\colon \Ind(V)\to \Ind(W)$ is of the form $\ide_B\otimes f'$, where $f'=(\varepsilon \otimes \ide_V)f(1\otimes \ide_V)$ (using an analogous computation as in \cite{BD}).

Next, given a Hopf module $(V, \triangleleft, \delta)$, we consider the morphism
\begin{equation}
e_V:=\triangleleft(\triangleleft \otimes a^{-1})(\ide_V\otimes S^{-1})\Psi^{-1}(\ide_B\otimes\triangleleft)(\delta\otimes a).
\end{equation}
Using the antipode axiom and the condition (\ref{hopfcondition}), we can show that $e_V$ is an idempotent in $\cB$. By assumption on $\cB$, it splits as $e_V=\iota_V\pi_V$, where $\iota_V\colon V^B\to V$ and $\pi_V \colon V\to V^B$ for an object $V^B$ of $\cB$, s.t. $\pi_V\iota_V=\ide_{V^B}$. We can now define the functor $\Res$ using a choice of such a splitting. On morphisms, we map $f\colon V\to W$ to $\pi_Wf\iota_V$. To show this gives a functor, we use the identity that $fe_V=e_Wf$ for any morphism of Hopf modules. This follows directly from $f$ commuting with $\delta$ and $\triangleright$.

It is easy to check that $\Res\Ind\cong \ide_{\cB}$ directly. Observe that in this case, $\iota=(1\otimes \ide_V)$, and $\pi=(\varepsilon \otimes \ide_V)$ for the object $\Ind(V)=B\otimes V$.
\end{proof}

Note that if both $\cM$ and its strictification (from \ref{strictification}) are higher representable, then we can conclude that $\cH_\cB(\cM)\cong \cB$ is an equivalence of categories using the proof of \cite{BD}.

\begin{remark}The construction of $\delta'$ in Lemma~\ref{centerinvertiblebetter} gives a functor
\[
\Theta \colon \leftsub{C^{\oop}}{\mathcal{YD}}^B(\cB)\to \mathcal{YD}^B_C(\cB),
\]
to the category of right YD-modules, which is the identity on morphisms. This functor is part of an equivalence of categories. This symmetry is not valid for Hopf modules. In fact, for $(V,c)\in \cH^G_\cB(\rmod{B}(\cB))$, the pair $(V,c)$ is \emph{not} a right Hopf module over $(B,C)$. 
\end{remark}

Recall that, given $\cB$ is rigid, then it has simultaneous left and right duals (cf. Proposition~\ref{rightdualslemma}). Further, $\cM$ has left left and right duals if and only if $B$ is a quasi-Hopf algebra with invertible antipode (cf. \ref{rightmodwarning}). We reinterpret the results of Proposition~\ref{generalduals} under the equivalence of Proposition~\ref{equivalence} showing that in this case the categories of YD-module have duals.

\begin{corollary}\label{quasiydduals}
Let $B$, $C$ be quasi-Hopf algebras, $G\colon B\to C$. Then the category $\leftsub{C^{\oop}}{\mathcal{YD}}^B(\cB)$ is rigid. The left dual action is given by
\begin{align}
\triangleleft^*&=(\ev_V\otimes\ide_{V^*})(\ide_{V^*}\otimes \triangleleft\otimes \ide_{V^*})(\ide_{V^*\otimes V}\otimes \Psi_{V^*,B})(\ide_{V^*}\otimes \coev_V\otimes S^{-1}),
\end{align}
and the dual quasi-coaction $\delta^*$ is defined, using graphical calculus where the right quasi-coaction $\delta'$ is denoted by \begin{minipage}{0.5cm}\scalebox{0.5}[0.5]{\small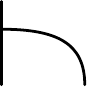}\end{minipage}, as depictured in Figure~\ref{yddualseqn}
\end{corollary}

\begin{figure}[htb]
\[
\vcenter{\hbox{\scalebox{1.2}[1.2]{\footnotesize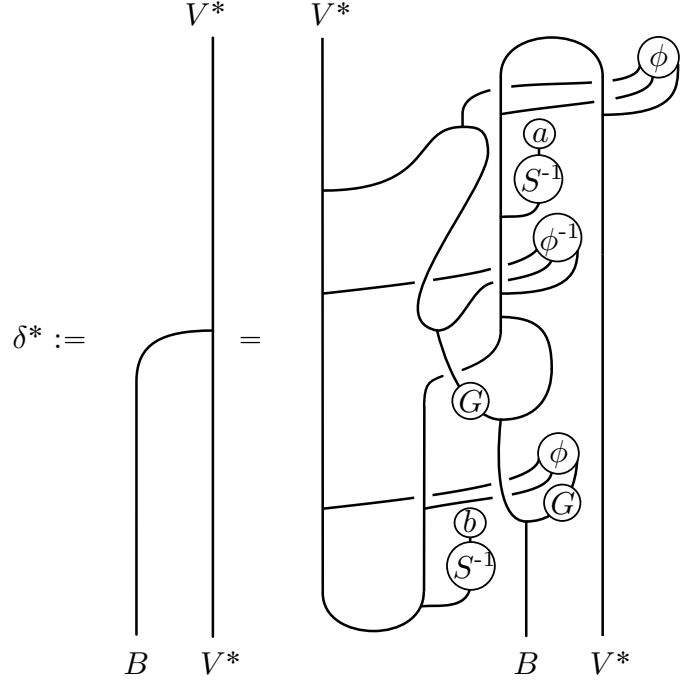}}}.
\]
\caption{Dual quasi-coaction}
\label{yddualseqn}
\end{figure}

\begin{proof}
The formula for $\delta^*$ follows by translating Proposition \ref{generalduals} under (\ref{mainidentity}). We use $\delta'$ to express the inverse commutativity isomorphism in terms of YD-modules. One can then use the functor $\Theta$ to compute this from the data of the quasi-coaction $\delta$. The dual right $B$-action on $V^*$ is the usual dual action in $\cM=\rmod{B}(\cB)$.
\end{proof}

Finally, it is now a direct corollary that the category $\leftsub{B^{\oop}}{\mathcal{YD}}^B(\cB)$ acts on the category $\leftsub{B^{\oop}}{\mathcal{H}}^B(\cB)$ on the left, using Proposition \ref{equivalence}
 to translate Theorem \ref{cataction} to the module-(quasi-)comodule description of this section.


\subsection{Quasitriangularity}\label{quasitriangularity}

Note that for any monoidal category $\cM$ over $\cB$, there always exists a forgetful functor $F\colon \cZ_\cB(\cM)\to \cM$. In this section, we want to study the situation when this functor has a right inverse, i.e. a fully faithful monoidal functor $R\colon \cM\to \cZ_\cB(\cM)$ such that there exists an natural isomorphism $\gamma\colon FR\stackrel{\sim}{\Longrightarrow}\ide_{\cM}$. This will induce a braiding on the category $\cM$ coming from the braiding of $\cZ_\cB(\cM)$. Note that the natural isomorphism $\gamma$ is required to be compatible with the monadicity transformation $\mu^F$, $\mu^R$. That means that the diagram
\begin{equation}\label{gammamonoidal}
\begin{diagram}
FR(X\otimes Y)&\rTo^{F(\mu^R_{X,Y})}&F(R(X)\otimes R(Y))&\rTo^{\mu^F_{R(X),R(Y)}}&FR(X)\otimes FR(Y)\\
&\rdTo^{\gamma_{X\otimes Y}}&&\ldTo^{\gamma_X\otimes \gamma_Y}&\\
&&X\otimes Y&&
\end{diagram}
\end{equation}
commutes. Moreover, the functor $RF\colon \cB\to \cB$ is required to be a functor of monoidal categories. This means that the diagram
\begin{equation}\label{gammabraided}
\begin{diagram}
RF(A\otimes B)&\rTo^{R(\mu^F_{A,B})}&R(F(A)\otimes F(B))&\rTo^{\mu^R_{F(A),F(B)}}&RF(A)\otimes RF(B)\\
\dTo^{RF(\Psi^\cB_{A,B})}&&&&\dTo^{\Psi^\cB_{RF(A),RF(B)}}\\
RF(B\otimes A)&\rTo^{R(\mu^F_{B,A})}&R(F(B)\otimes F(A))&\rTo^{\mu^R_{F(B),F(A)}}&RF(B)\otimes RF(A)
\end{diagram}
\end{equation}
is required to commute. We say that $RF$ \emph{preserves} the braiding $\Psi^\cB$ in this case.

\begin{definition}\label{quasitriangular}
Let $\cM$ be a monoidal category and $\cB$ braided monoidal. We call a monoidal functor $R\colon \cM\to \cB$ which is a right inverse to $F$, with the data of $\gamma$ satisfying the compatibilities (\ref{gammamonoidal}) and (\ref{gammabraided}) above a \emph{quasitriangular structure} on $\cM$. This implies that $R$ is fully faithful.
\end{definition}

\begin{lemma}\label{quasitriangularbraiding}
A quasitriangular structure $R\colon \cM\to \cB$ gives a braiding on the category $\cB\stackrel{F}{\to}\cM$, such that the functors $R$ and $F$ preserve the braidings.
\end{lemma}
\begin{proof}
For any two objects $X$ and $Y$ of $\cM$, define the braiding by the composition
\begin{equation}
\Psi^{\cM}_{X,Y}:=\gamma_{Y\otimes X}F((\mu^R_{Y,X})^{-1})F(\Psi^{\cB}_{R(X),R(Y)})F(\mu^R_{X,Y})\gamma^{-1}_{X\otimes Y}.
\end{equation}
The functor $F$ preserves this braiding. To see this, combine three commutative squares: The one defining $\Psi^\cM$ applied to $F(A)$ , $F(B)$, $F$ applied to (\ref{gammabraided}), and naturality of $\gamma$ in $F(\Psi^\cB_{A,B})$. This gives
\begin{align*}
&\gamma_{F(B\otimes A)}FR(\mu^F_{B,A})F(\mu^R_{F(B),F(A)})^{-1}F(\mu^R_{F(B),F(A)})\gamma^{-1}_{F(B)\otimes F(A)}\Psi_{F(A),F(B)}^\cM\\
&=F(\Psi_{A,B}^\cB)\gamma_{F(A\otimes B)}FR(\mu^F_{A,B})F(\mu^R_{F(A),F(B)})^{-1}F(\mu^R_{F(A),F(B)})\gamma^{-1}_{F(A)\otimes F(B)}.
\end{align*}
We can simplify to
\begin{align*}
&\gamma_{F(A\otimes B)}FR(\mu^F_{A,B})\gamma^{-1}_{F(A)\otimes F(B)}\\
&=\mu^F_{A,B}\gamma_{F(A)\otimes F(B)}F(\mu^R_{F(A),F(B)})^{-1}F(\mu^R_{F(A),F(B)})\gamma^{-1}_{F(A)\otimes F(B)}\\
&=\mu^F_{A,B}(\gamma_{F(A)}\otimes \gamma_{F(B)})(\gamma^{-1}_{F(A)}\otimes \gamma^{-1}_{F(B)})=\mu^F_{A,B},
\end{align*}
by first applying naturality of $\gamma$ in $\mu_{A,B}^F$, and then applying (\ref{gammamonoidal}) twice. Hence $F(\Psi_{A,B}^\cB)\mu_{A,B}^F=\mu_{B,A}^F\Psi_{F(A),F(B)}^\cM$
as required.

The proof that $R$ preserves the braiding starts by composing the diagrams of $R$ applied to the definition of $\Psi^\cM_{X,Y}$ with (\ref{gammabraided}), and naturality of $\Psi^\cB$ applied to $R(\gamma_X)$, $R(\gamma_Y)$. This gives the equality
\begin{align*}
&\Psi^B_{R(X),R(Y)}(R(\gamma_{X})\otimes R(\gamma_Y))\mu^R_{FR(X),FR(Y)}R(\mu^F_{R(X),R(Y)})RF(\mu_{X,Y}^R)R(\gamma_{X\otimes Y}^{-1})\\
&=(R(\gamma_{Y})\otimes R(\gamma_X))\mu^R_{FR(Y),FR(X)}R(\mu^F_{R(Y),R(X)})RF(\mu_{Y,X}^R)R(\gamma_{Y\otimes X}^{-1})R(\Psi_{X,Y}^\cM).
\end{align*}
Applying first naturality of $\mu^R$, and then (\ref{gammamonoidal}) we can simplify to
\begin{align*}
&(R(\gamma_{X})\otimes R(\gamma_Y))\mu^R_{FR(X),FR(Y)}R(\mu^F_{R(X),R(Y)})RF(\mu_{X,Y}^R)R(\gamma_{X\otimes Y}^{-1})\\
&=\mu^R_{X,Y}R(\gamma_{X}\otimes \gamma_Y)R(\mu^F_{R(X),R(Y)})RF(\mu_{X,Y}^R)R(\gamma_{X\otimes Y}^{-1})\\
&=\mu^R_{X,Y}R(\gamma_{X\otimes Y})R(\gamma_{X\otimes Y})^{-1}=\mu^R_{X,Y}.
\end{align*}
This proves $\mu^R_{Y,X}R(\Psi_{X,Y}^\cM)=\Psi^B_{R(X),R(Y)}\mu_{X,Y}$ as claimed.
\end{proof}

\begin{theorem}
Let $\cM$ be a monoidal category over a braided monoidal category $\cB$ (with functors $F,P$ as in \ref{FPsetting}). Then $\cM$ is braided monoidal (and $F$, $P$ preserve the braidings) if and only if $\cM$ has a quasitriangular structure over $\cZ_\cB(\cM)$.
\end{theorem}

\begin{proof}
Applying Lemma~\ref{quasitriangularbraiding} to $\cZ_\cB(\cM)$ shows that if such a quasitriangular structure exists, then $\cM$ is braided. It is clear that the functors $F,P$ preserve the braidings. For the converse, we use the functor $R\colon \cM\to \cZ_{\cB}(\cM)$ defined by mapping an object $V$ of $\cM$ to the pair $(V, \Psi^\cM_{V,-})$, where $\Psi^\cM$ is the given braiding in $\cM$, and the identity on morphisms.
\end{proof}

We can recover the usual definitions of quasitriangularity (including \emph{co}-quasi\-tri\-an\-gu\-larity, and \emph{weak} quasitriangularity (due to Majid \cite{Maj2}), see Section~\ref{BCquasitriangular}) from this more general definition using reconstruction theory. These results can be given in the general setting over a braided monoidal category $\cB$. To do this, we assume that the functors $F$ and $R$ are \emph{strictly} monoidal and inverse to each other (i.e. $\gamma=\ide_{\cM}$, $\mu^F=\ide_{\cM}$ and $\mu^R=\ide_{\cB}$).

\begin{proposition}
A quasi-bialgebra $B$ in $\cB$, is quasitriangular, i.e. an $R$-matrix satisfying (\ref{braidingaxiom1})-(\ref{braidingaxiom3}) exists, if and only if the category $\cM=\lmod{B}(\cB)$ possesses a quasitriangular structure $R\colon \cM\to \cZ_\cB(\cM)$. 

If we identify $\cZ_\cB(\cM)$ with the category of YD-modules over $B$, the functor $R$ can be identified with the functor mapping a $B$-module $(V, \triangleright)$ to the YD-module on $V$ with action $\triangleright$ and quasi-coaction $\Psi(\triangleright \otimes \ide_B)(\ide_B\otimes \Psi)(R\otimes \ide)$.
\end{proposition}

\begin{proof}
Recall that by Lemma \ref{quasitriangularbraiding} the existence of the functor $R$ implies that $\lmod{B}(\cB)$ is braided, with braiding induced by $R$. By higher representability \ref{higherrep} this gives the existence of a universal $R$-matrix $R\colon I \to B\otimes B$ satisfying the axioms (\ref{braidingaxiom1})-(\ref{braidingaxiom3}). Moreover, as the functor $R$ preserves the braiding by construction, this gives the equality
\begin{equation}
\Psi(\triangleright_V\otimes \triangleright_W)(\ide\otimes \Psi\otimes \ide)(R\otimes \ide_{V\otimes W})=(\triangleright_W\otimes \ide)(\ide\otimes\Psi)(\delta_V\otimes \ide),
\end{equation}
as we assume $\mu^R=\ide$. From this, we conclude that the quasi-coaction needs to be as stated by apply naturality of the braiding to $1\to B$, where $B$ has the regular action. The YD-condition corresponds to the $R$-matrix axiom that conjugation by $R$ transforms the coproduct into the opposite coproduct $\Delta^{\op{cop}}$ (see \ref{braidingaxiom1}). The other axioms correspond to the quasi-coaction condition and the fact that the functor $R$ is monoidal. Note that we need left module versions of these axioms here. These can however easily obtained from \ref{monoidalcenter} and (\ref{monadicityc}) using braided left module reconstruction as in \ref{majidreconstruction}. Conversely, given a universal $R$-matrix, one can run the argument backwards and check that the functors are obtained from the axioms (\ref{braidingaxiom1})-(\ref{braidingaxiom3}).
\end{proof}


\section{Braided Drinfeld and Heisenberg Doubles}\label{doubles}

In this section, we use reconstruction theory to obtain general definitions of the braided Drinfeld and Heisenberg doubles. In the case of the braided Drinfeld double, this resembles the double bosonization in \cite{Maj2}. For this, we will now leave the generality of quasi-Hopf algebras restricting to strict Hopf algebras. The main reason for this is that the picture becomes more symmetric, as the dual of a Hopf algebra is a Hopf algebra (which is \emph{not} true for quasi-Hopf algebras). Recall the definition of dually paired Hopf algebras from \ref{duallypaired}. We note that, the restriction to Hopf algebras leads to several simplifications. For instance, as mentioned before, the quasi-coaction $\delta\colon B\to B\otimes V$ of the definition of Yetter-Drinfeld modules now becomes a $B^{\oop}$-coaction. We will summarize the simpler formulae in \ref{strictsection}. Next, we use the categorical definition of quasitriangularity from Section~\ref{quasitriangularity} to discuss the notion of weak quasitriangularity in the setting of dually paired bialgebras in $\cB$. 

Once this preliminary work has been done, we will take $\cB$ to be $\lmod{H}$ or $\lcomod{H}$ where $H$ is an ordinary Hopf algebra in $\Vect$ which is either quasitriangular or weak quasitriangular (but possibly infinite-dimensional). We then define braided Drinfeld and Heisenberg doubles via reconstruction theory and give explicit presentations using modified Sweedler's notation. We also interpret the  analogues of the BGG-category $\cO$ in this general context. 

Finally, we explain how the categorical action from \ref{cataction} gives that the braided Heisenberg double is a 2-cocycle twist of the braided Drinfeld double. This generalizes an earlier result of \cite{Lu} and is the main result of this section.


\subsection{Simplifications for Strict Hopf Algebras}\label{strictsection}

As mentioned above, for $B$ a Hopf algebra in $\cB$, the category $\leftsub{B^{\oop}}{\cal{YD}}^B(\cB)$ consists of simultaneous $B$-modules and $B^{\oop}$-comodules in $\cB$ which satisfy (\ref{ydcondition}). Further, the inverse braiding can now be given in Figure \ref{braidingsimpler}. Furthermore, we can introduce an equivalent form of the YD-conditions, given that the antipode is invertible:
\begin{figure}
\[
\begin{array}{ccccccc}
\leftsub{B^{\oop}}{\Psi}^B&=&\vcenter{\hbox{\small\import{Graphics/}{ydbraiding3b.pdf_tex}}}, &\qquad&(\leftsub{B^{\oop}}{\Psi}^B)^{-1}&=&\vcenter{\hbox{\small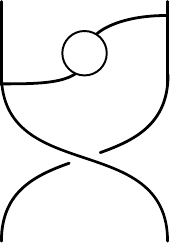}}.
\end{array}
\]
\caption{Braiding and inverse braiding for strict Hopf algebras}
\label{braidingsimpler}
\end{figure}

\begin{lemma}\label{altconds}
The YD-condition for $\leftsub{B^{\oop}}{\mathcal{YD}}^B(\cB)$ is equivalent to either of the conditions of Figure \ref{altydconds}.
\end{lemma}
\begin{figure}
\[
\begin{array}{ccc}
\vcenter{\hbox{\small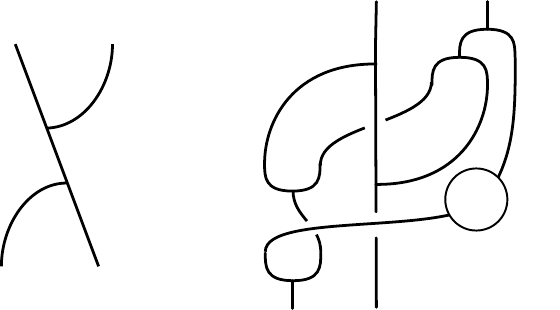}},&\text{~and~}&\vcenter{\hbox{\small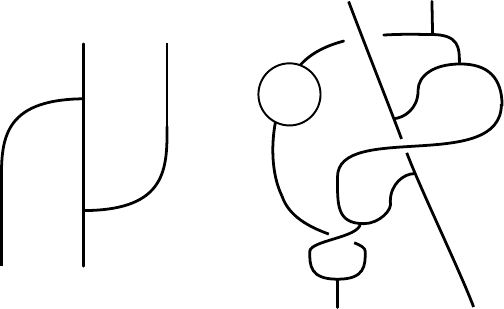}}.
\end{array}
\]
\caption{Equivalent YD-conditions for strict Hopf algebras}
\label{altydconds}
\end{figure}

\begin{proof}
This is an exercise in graphical calculus
using that $m(S^{-1}\otimes \ide)\Psi^{-1}\Delta=1\varepsilon$.
\end{proof}

We interpret this relation as enabling us to permute the action past the coaction and vice versa (similar formulas can be given for quasi-Hopf algebras, but obstructions caused by the 3-cycle occur). For strict Hopf algebras, we can also give simpler formulae relating the (left) dual YD-modules $V^*$ to the given structures on $V$.

\begin{corollary}
Let $B$ be a Hopf algebra with invertible antipode in $\cB$, then the category $\leftsub{B^{\oop}}{\mathcal{YD}}^B(\cB)$ is rigid. Here, the dual action and coaction are given by
\begin{align}
\delta^*&=(\ev_V\otimes S\otimes \ide_{V^*})(\ide_{V^*}\otimes \Psi_{B,V}\otimes \ide_{V^*})(\ide_{V^*}\otimes \delta \otimes \ide_{V^*})(\ide_{V^*}\otimes \coev_V)\\
\triangleleft^*&=(\ev_V\otimes\ide_{V^*})(\ide_{V^*}\otimes \triangleleft\otimes \ide_{V^*})(\ide_{V^*\otimes V}\otimes \Psi_{V^*,B})(\ide_{V^*}\otimes \coev_V\otimes S^{-1}).
\end{align}
\end{corollary}
\begin{proof}
This is an immediate consequence of Corollary~\ref{quasiydduals}, setting $\phi=1\otimes 1\otimes 1$ and $\alpha=\beta=1$. 
\end{proof}

In the setting of strict Hopf algebras in $\cB$, the R-matrix axioms simplify to the ones in Figure~\ref{quasitriangularaxioms}. See \cite[Figure~9.18]{Maj1} for a proof using graphical calculus. Conversely, given such a universal $R$-matrix for $B$, the category $\lmod{B}(\cB)$ (or $\rmod{B}(\cB)$) is braided monoidal (this is a special case of \ref{majidreconstruction}(c)).

\begin{figure}
\[
\text{\begin{minipage}{0.9\linewidth}\small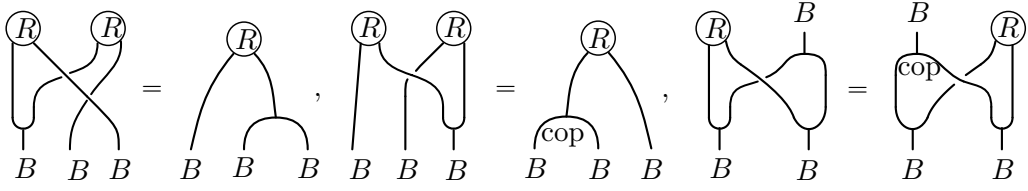\end{minipage}}
\]
\caption{R-matrix axioms for braided Hopf algebras}
\label{quasitriangularaxioms}
\end{figure}

Note that, unless we are in the symmetric monoidal case, the convolution inverse $R^{-1}$ does not give an $R$-matrix up to application of the braiding. It is however possible to define $R^{\op{op}}$ by applying reconstruction theory to $\overline{\cM}$ (that is, $\cM$ with opposite braiding). The axioms for a dual $R$-matrix (giving a braiding on $\lcomod{B}(\cB)$ can be obtained by rotating this picture in the horizontal axis.

\subsection{Yetter-Drinfeld Modules over Dually Paired Hopf Algebras}\label{pairedydsect}

Given dually paired Hopf algebras $C$, $B$ (see Section \ref{duallypaired}), we can embed the category $\leftsub{B^{\oop}}{\mathcal{YD}}^B(\cB)=\cZ_\cB(\rmod{B}(\cB))$ into a larger category of Yetter-Drinfeld modules over $(C,B)$ by applying the functor
\[
\leftsub{B^{\oop}}{\Phi}\colon \lcomod{B^{\oop}}(\cB)\to \lmod{B}(\cB),\quad (V,\delta)\mapsto (V,\triangleright_\delta),
\]
where $\triangleright_\delta=(\ev\otimes \ide_V)(\ide_B\otimes \delta)$, to the left comodule structure of the YD-module. The resulting category is denoted by $\leftexp{C}{\mathcal{YD}}^B(\cB)$ and consists of objects $V$ of $\cB$ with a left $C$-action $\triangleright$ and a right $B$-action $\triangleleft$ compatible via the YD-condition
\begin{equation}
\begin{split}&(\ev\otimes \triangleleft)(\ide\otimes \Psi_{V,B}\otimes \ide)(\ide\otimes \triangleright \otimes \ide)(\Delta_C\otimes \ide\otimes \Delta_B)\\&=(\triangleright \otimes \ev)(\ide\otimes \Psi_{C,V}\otimes \ide)(\ide\otimes \triangleleft \otimes \ide)(\Delta_C\otimes \ide\otimes \Delta_B)\end{split}.
\end{equation}
Morphisms are required to commute with both the left $C$- and the right $B$-action. Note that the monoidal structure is given by the usual monoidal structures for left $C$- and right $B$-modules in $\cB$. Note that the larger category is not braided in general any more (for this, we require that a coevaluation map $I\to B\otimes C$ exists in $\cB$).

We can also embed the category of Hopf modules $\leftsub{B^{\oop}}{\mathcal{H}}^B(\cB)=\cH(\rmod{B}(\cB))$ into a larger category $\leftexp{C}{\mathcal{H}}^B(\cB)$ by again applying the functor $\leftsub{B^{\oop}}{\Phi}$ to the left comodule structure. The resulting category consists of objects in $\cB$ with a left $C$-module and a right $B$-module structure such that
\begin{equation}
\triangleright(\ide_{B}\otimes \triangleleft)=(\ev\otimes \triangleleft)(\ide_C\otimes \Psi_{V,B}\otimes \ide_B)(\ide_C\otimes \triangleright \otimes \ide_{B\otimes B})(\Delta_C\otimes \ide_V \otimes \Delta_B).
\end{equation}
Again, morphisms are required to commute with both the left and right action. We add the compatibility conditions in the language of graphical calculus in Figure~\ref{duallypairedconds}. Note that the YD-condition given by the first diagram is precisely the one used in \cite[A.2]{Maj2}.

\begin{figure}
\[
\begin{array}{ccc}
\vcenter{\hbox{\small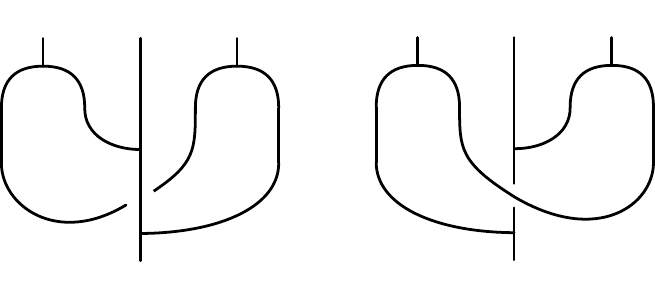}} & , &\vcenter{\hbox{\small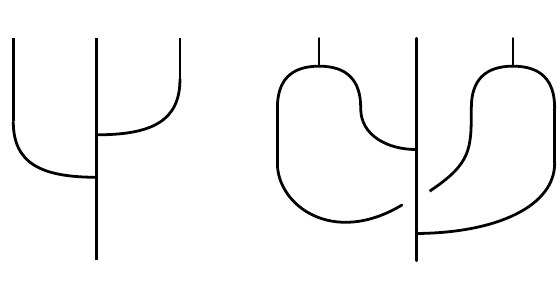}}
\end{array}
\]
\caption{YD- and Hopf-compatibility conditions for dually paired Hopf algebras}
\label{duallypairedconds}
\end{figure}

\begin{proposition}\label{duallypairedaction}
The monoidal category $\leftexp{C}{\mathcal{YD}}^B(\cB)$ acts on the category $\leftexp{C}{\mathcal{H}}^B(\cB)$, where $V\triangleright W:=V\otimes W$ with left $C$-action on $V\otimes W$ given by $\Delta_C$ and right $B$-action on $V\otimes W$ by $\Delta_B$.
\end{proposition}
\begin{proof}
It is easy to check using graphical calculus that $V\triangleright W$ is a Hopf module over $(C,B)$. It is also clear that a pair of morphisms $f\colon V\to V'$, $g\colon W\to W'$ induces a morphism $f\otimes g\colon V\triangleright W\to V'\triangleright W'$ as both $f$, $g$ commute with the respective $C$, $B$ actions, so their tensor product will commute with the tensor product actions.
\end{proof}

Finally, we can also compute the center of monoidal categories of comodules in terms of Yetter-Drinfeld modules.

\begin{proposition}\label{comodulecenter}
For a Hopf algebra object $B$ in $\cB$, there are isomorphisms of categories
\begin{align*}
\cZ_\cB(\lcomod{B}(\cB))&\cong \leftsub{B}{\cal{YD}}^{B^{\coop}}(\cB),\\
\cH_\cB(\lcomod{B}(\cB))&\cong \leftsub{B}{\cal{H}}^{B^{\coop}}(\cB).
\end{align*}
\end{proposition}
\begin{proof}
For $(V,c)\in \cZ_\cB(\lcomod{B}(\cB))$, consider the map $\triangleright :=(\varepsilon\otimes \ide_V)c_{B}$. Using monadicity of $c$ and the monoidal structure on the center, we find that $\triangleright$ is a right $B^{\coop}$-module. Dually to the proof of $\ref{equivalence}$ (with the simplification that $\phi=1\otimes 1\otimes 1$, we find that the datum of $\triangleright$ allows to recover $c$, and that conversely any right $B^{\coop}$-module that satisfies the YD-condition (\ref{ydcondition}) gives an element of the center. The proof for the Hopf center is again analogous.
\end{proof}

\begin{corollary}\label{centerequivalence}
There exists an isomorphism of braided monoidal categories 
\[
\cZ_\cB(\lcomod{B}(\cB))\cong \overline{\cZ_\cB(\lmod{B}(\cB))}.
\]
\end{corollary}
\begin{proof}
This can be proved by showing that the monoidal categories $\leftsub{B}{\cal{YD}}^{B^{\coop}}(\cB)$ and $\leftexpsub{B}{B}{\cal{YD}}(\cB)$ are isomorphic. To do this, we recall that $\rmod{B^{\coop}}(\cB)\cong \rmod{\leftexp{\op{cop}}{B}}(\cB)$ by definition of the opposite coproduct. The latter category is equivalent to $\lmod{B}(\cB)$ by the functor
\[
(V,\triangleleft)\mapsto(V,\triangleright:=\triangleleft\Psi(S\otimes \ide_V).
\]
We use this functor to translate the datum of the right $B^{\coop}$-module of an element of $\leftsub{B}{\cal{YD}}^{B^{\coop}}(\cB)$ to a left $B$-module structure. It is an exercise to check, using the alternative YD-conditions of \ref{altconds}, that the resulting left module is YD-compatible with the comodule structure on $V$. Note that the braiding corresponds to the opposite braiding. As all translations used are isomorphisms (which are the identity on morphisms), this gives an equivalence of the centers as stated.
\end{proof}


\subsection{Weak Quasitriangularity}\label{BCquasitriangular}

We now apply the general categorical viewpoint on quasitriangularity from Section~\ref{quasitriangularity} to a setting of dually paired Hopf algebras $B,C$ in $\cB$. The aim is to reinterpret Majid's concept of weak quasitriangularity in this context as it is a crucial feature used in defining braided Drinfeld doubles of braided Hopf algebras in comodule categories as algebras. In fact, it is needed in order to include quantum groups as an example of braided Drinfeld doubles (as done in \cite{Maj2}).

The aim is to rewrite the datum of a dual R-matrix via applying evaluation to the datum of morphisms $R\colon B^{\oop}\to C$. In Figure \ref{weakqtaxioms2} we define two ways of doing this. Here, $R^{\oop}$ is the dual $R$-matrix obtain by reconstruction such that the condition of Figure \ref{opprmatrix} holds. We will in general need to remember the datum of both morphisms $R$ and $\ov{R}$.
It is furthermore required that $\ev$ is a pairing satisfying
\begin{align}\label{oppairing}
\begin{split}\ev(m_B\otimes \ide_{B})&=\ev(\ide_C\otimes \ev\otimes \ide_B)(\ide_{C\otimes C}\otimes \Delta_C),\\
\ev(\ide_C\otimes m^{\op{op}}_B)&=(\ev\otimes \ev)(\ide_C\otimes \Psi_{C,B}\otimes \ide_B)(\Delta_C\otimes \ide_{B\otimes B})\end{split}
\end{align}
The reason for this is that we want to be able to write the braiding of $B^{\oop}$-comodules in terms of the induced $H$-module structure. In Figure \ref{weakqtrequirements}, there are two ways of doing this, using the $H$-module structure on either tensorand of the braiding.
\begin{figure}
\[
\vcenter{\hbox{\small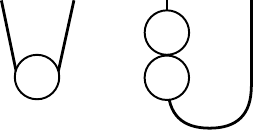}}~~,\qquad \vcenter{\hbox{\small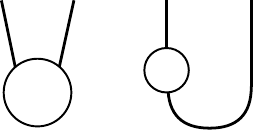}}~~.
\]
\caption{A weak quasitriangular structure for dually paired Hopf algebras}
\label{weakqtaxioms2}
\end{figure}
\begin{figure}
\[
\Psi=\vcenter{\hbox{\small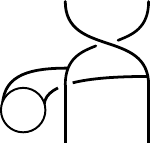}}=\vcenter{\hbox{\small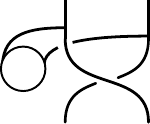}}~~.
\]
\caption{The opposite R-matrix}
\label{opprmatrix}
\end{figure}

\begin{figure}
\begin{align*}
\Psi=\vcenter{\hbox{\small\import{Graphics/}{dualrmatrixbraiding.pdf_tex}}}&=\vcenter{\hbox{\small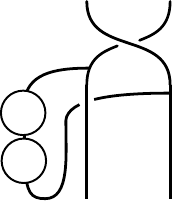}}=\vcenter{\hbox{\small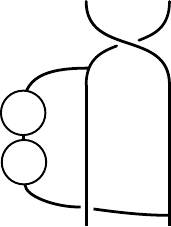}}~~,
&\Psi=\vcenter{\hbox{\small\import{Graphics/}{dualrmatrixopbraiding.pdf_tex}}}&=\vcenter{\hbox{\small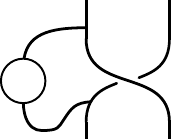}}=\vcenter{\hbox{\small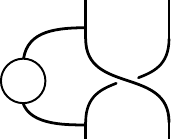}}~~.
\end{align*}
\caption{Braidings for weak quasitriangular structures}
\label{weakqtrequirements}
\end{figure}

Being able to rewrite the dual R-matrix in this way is essential if we want to apply module reconstruction to braided categories of \emph{co}modules, for the reason that we need to express all formulas in terms of the induced action (from a coaction) via $\ev$. This is used in the construction of the quantum groups $\Ug$ as braided Drinfeld doubles (cf. \ref{quantumgroups}).

Our definition of weak quasitriangularity resembles the idea of the definition in \cite{Maj2}. It also gives an intermediate notion between the stronger requirement of $C$ having a universal R-matrix and the weaker assumption of $B^{\oop}$ having a dual R-matrix.

\begin{lemma}\label{Rfunctorlemma}
Given a pairing $\ev$ as in (\ref{oppairing}) and let $R\colon B^{\oop}\to C$ be a morphism of bialgebras in $\cB$. Then for any given left $B^{\oop}$-comodule $(V,\delta)$ in $\cB$,
\begin{enumerate}
\item[(i)] $\triangleright :=(\ev\otimes \ide_V)(R\otimes \delta)$ is a left $B^{\oop}$-module in $\cB$, and
\item[(ii)] $\triangleleft :=(\ev\otimes \ide_V)(R\otimes \Psi)(\delta\otimes \ide_B)$ is a right $B^{\oop\coop}$-module in $\cB$.
\end{enumerate}
\end{lemma}
\begin{proof}
This is not hard to check, using (\ref{oppairing}) and the definition of the (co)opposite product. Note that $B^{\oop\coop}=(B^{\oop})^{\coop}$ and hence is a bialgebra in $\cB$.
\end{proof}

\begin{proposition}\label{weakqtstructuresprop}
Let $R\colon B^{\oop}\to C$ be a morphism of Hopf algebras with an evaluation as in (\ref{oppairing}) such that the condition 
\begin{equation}\label{weakqtaxiom}
\begin{split}
&(m^{op}_C\otimes \ev(R\otimes \ide_B))(\ide_C\otimes \Psi_{C,B}\otimes\ide_B)(\Delta_C\otimes \Delta_B)\\&=(\ev(R\otimes \ide_B)\otimes m_B)(\ide_C\otimes \Psi_{C,B}\otimes\ide_B)(\Delta_C\otimes \Delta_B)
\end{split}
\end{equation}
holds. Then there exists two quasitriangular structures on the category of comodules $\lcomod{B^{\oop}}(\cB)$ in $\cZ_\cB(\lmod{B^{\oop}}(\cB))$, both with respect to the forgetful functor.
\begin{enumerate}
\item[(i)]  The functor mapping $(V,\delta)$ to $(V, \delta, \triangleright:=(\ev\otimes \ide_V)(R\otimes \delta))$ as in \ref{Rfunctorlemma}(i).
\item[(ii)] The functor mapping $(V, \delta)$ to $(V,\delta, \triangleright:=(\ev\otimes \ide_V)(R\otimes \Psi)(\delta\otimes\ide_B)\Psi(S\otimes \ide_V))$, which is the composition of the functor in \ref{Rfunctorlemma}(ii) with the equivalence \ref{centerequivalence}.
\end{enumerate}
\end{proposition}

\begin{proof}
In Lemma~\ref{Rfunctorlemma}, we checked that the stated compositions of maps indeed give left $B^{\oop}$-modules (respectively right $B^{\oop \coop}$-modules). The condition (\ref{weakqtaxiom}) is precisely what is required for the resulting map $\Psi$ as in Figure \ref{weakqtrequirements} to be morphisms of comodules. Hence, $\lcomod{B^{\oop}}(\cB)$ is braided. This gives the quasitriangular structure of (i) as described and a quasitriangular structure of $\lcomod{B^{\oop}}(\cB)$ in $\ov{\cZ_\cB(\lcomod{B^{\oop}}(\cB))}$. By the equivalence \ref{centerequivalence}, this corresponds to the quasitriangular structure (ii) as stated.
\end{proof}

\begin{definition}
A \emph{weak quasitriangular pair} for dually paired Hopf algebras $C,B$ is the datum of two morphisms of Hopf algebras $R,\ov{R}\colon B^{\oop}\to C$ in $\cB$ which satisfy (\ref{weakqtaxiom}), s.t. with respect to a pairing $\ev$ of $B^{\oop}$ and $C$ as in (\ref{oppairing}), the axioms from Figure \ref{weakqtaxioms2} hold.
\end{definition}

Note that such maps $R, \ov{R}$ are automatically convolution invertible with convolution inverses given by $R^{-1}:=RS$ and $\ov{R}^{-1}:=\ov{R}S$. The existence of quasitriangular structures as in \ref{weakqtstructuresprop} does \emph{not} imply the existence of the maps $R$, $\ov{R}$. In fact, it only implies the existence of dual universal R-matrices. In the following, we will describe how one can obtain the maps $R$, $\ov{R}$ via reconstruction theory under certain representability conditions.

\begin{lemma}
If the category $\lcomod{B^{\oop}}(\cB)$ can be (higher) represented in the sense of Section \ref{reconstruction} as modules over a Hopf algebra $C$ in $\cB$, then there exists a weak quasitriangular pair $R$, $\ov{R}$ for the dually paired Hopf algebras $C,B$.
\end{lemma}

\begin{proof}
Consider the image of $(B^{\oop})^{\op{coreg}}$, $B^{\oop}$ with coregular coaction given by the coproduct, under $E\colon \lcomod{B^{\oop}}(\cB)\to \lmod{C}(\cB)$. This gives a map $\gamma\colon C\otimes B \to B$. Note that the functor $E$ factors through the forgetful functor to $\cB$ by assumption. We define
\[
\ev:=\varepsilon\gamma\colon C\otimes B\to I.
\]
Now let $\delta$ be any $B^{\oop}$-comodule structure on an object $V$ of $\cB$. Then $\delta$ is a morphism of $B^{\oop}$-comodules $V\to (B^{\oop})^{\op{coreg}}\otimes V^{\triv}$. This follows simply from the comodule condition. Hence, $\delta$ is a morphism of $C$-modules w.r.t. the image under $E$ and we derive that for the $C$-action (denoted by $\triangleright$) on $V$ under $E$ we have
\begin{equation}
\triangleright=(\varepsilon\otimes \ide_V)\delta \gamma =(\varepsilon\otimes \ide_V)(\triangleright \otimes \ide_V)(\ide_C\otimes \delta)=(\ev\otimes \ide_V)(\ide_C\otimes\delta).
\end{equation}
Hence, the functor $E$ is induced by the map $\ev$. From this we derive directly that $\ev$ satisfies the axioms (\ref{oppairing}).

To find $\ov{R}$, we apply reconstruction (of $C$-modules) to the natural transformation
\[
\ov{R}\colon B^{\oop}\otimes F\to F, \quad \ov{R}_{(V,\delta)}=(R^{\oop}\otimes \ide_V)(\ide_B\otimes \delta).
\]
The morphism $R$ can be obtained by applying reconstruction to
\[
R\colon B^{\oop}\otimes F\to F, \quad R_{(V,\delta)}=(R\otimes \ide_V)(S^{-1}\otimes \delta).
\]
hence $\ev(RS\otimes \ide)=R$ as required in Figure \ref{weakqtaxioms2}.
\end{proof}

The following proposition relates the different concepts for quasitriangularity.

\begin{proposition} Let $B,C$ be dually paired Hopf algebras in $\cB$.
\begin{itemize}
\item[(a)]
A universal $R$-matrix for $C$ induces a weak quasitriangular structure for $B,C$, which induces a dual $R$-matrix on $B$.
\item[(b)]
Let $C$ be the left dual of $B$ in the sense that a coevaluation map $\coev\colon I \to B\otimes C$ exists, satisfying that the usual duality relations hold. Then the three concepts in (a) coincide.
\end{itemize}
\end{proposition}

\begin{proof}
This is due to \cite{Maj2}.
To prove part (a), given a universal $R$-matrix for $C$, define
\begin{align}
\ov{R} & := (\ide_C\otimes \ev)(R^{\oop}\otimes \ide_B).
\end{align}
To obtain $R$, we need to find a morphism $\widetilde{R}\colon I\to B\otimes B$ such that
\begin{equation}
(\triangleright_V\otimes \triangleright_W)(\ide\otimes \Psi_{C,B}\otimes \ide)(\widetilde{R}\otimes \ide_{V\otimes W})=\Psi_{F(W),F(V)}F(\Psi^{\cM}_{V,W})\Psi^{-1}_{F(W),F(V)}\Psi^{-1}_{F(V),F(W)}.
\end{equation}
This can be thought of as $R^{\oop\oop}$ and exists by higher representability. Now define
\begin{align}
R^{-1} & := (\ide_C\otimes \ev)(\widetilde{R}\otimes \ide_B).
\end{align}
It was already observed in Proposition \ref{weakqtstructuresprop} that a weak quasitriangular structure induced a dual R-matrix.

Part (b) follows from the observation that given a coevaluation map, the categories $\lcomod{B^{\oop}}(\cB)$ and $\lmod{B}(\cB)$ are canonically equivalent.
\end{proof}

Note that in a symmetric monoidal category this theory simplifies as $R^{\oop}$, $\widetilde{R}$ can be obtained from $R$ using the symmetry.

We will later observe that in the case of the quantum group, a weak quasitriangular structure exists for the Hopf algebra $\mC \mZ^n$. 
This is essential in the interpretation of $\Ug$ as a \emph{double bosonization} in \cite{Maj2}.

\begin{remark}
If $\cB=\Vect$, then the functor $E\colon \lcomod{B^{\oop}}\to \lmod{C}$ is fully faithful if the pairing $\ev$ is perfect. Clearly, $E$ is faithful as it is the identity on morphisms. Assume the pairing is perfect and consider a linear map $f\colon V\to W$. The expression\footnote{Using modified Sweedler's notation, cf. \ref{presentations}.}
\begin{align*}
\ev(g\otimes v^{(-1)})\otimes f(v^{(0)})&-\ev(g\otimes (fv)^{(-1)})\otimes (fv)^{(0)}, &\forall v\in g\in C.
\end{align*}
is zero for all if $E(f)$ is a morphism of $C$-modules. But if that means that the difference $v^{(-1)}\otimes f(v^{(0)})-(fv)^{(-1)}\otimes (fv)^{(0)}$ lies in the right radical of $\ev$. If $\ev$ is perfect these terms have to be zero and hence $f$ is a morphism of $B^{\oop}$-comodules. This shows $E$ is full.
\end{remark}


\subsection{Definition of Braided Drinfeld and Heisenberg Doubles}\label{doublessection}

Assume that $B$ and $C$ are perfectly dually paired Hopf algebras in $\cB$. Using the description of YD-modules over ($B$, $C$) from Section~\ref{pairedydsect}, we have obtained a fully faithful functor
\[
\cZ_\cB(\rmod{B}(\cB))\hookrightarrow \leftexp{C}{\mathcal{YD}}^B(\cB).
\]
In the case where $\cB=\lmod{H}$ for a quasitriangular Hopf algebra $H$ (or $\cB=\lcomod{A}$ for $A$ and $H$ perfectly dually paired Hopf algebras in $\Vect$ with a weak quasitriangular structure), we can now define braided Drinfeld and Heisenberg doubles via reconstruction theory (see Section~\ref{reconstruction}). The advantage of this approach is that we can extract all formulae from braided diagrams and the structure maps of the (Hopf) algebras thus defined will satisfy all the axioms by general theory and no explicit checks have to be carried out. For the braided Drinfeld double this repeats Majid's construction of the \emph{double bosonization} in \cite{Maj2} using left $H$-modules instead of right ones. In \cite{Maj2} the computations for the Hopf algebra structure are also carried out explicitly.

\begin{remark}
In the following, we will work over $\Vect$, which is the symmetric monoidal category of countably infinite-dimensional vector spaces. We want to work with an infinite-dimensional Hopf algebras $B$. Now $B$ does not necessarily have a perfectly dually paired Hopf algebra $C$ in the sense of \ref{duallypaired}. In fact, the maximal subalgebra of the vector space dual $B^*=\Hom_k(B,k)$ which is a Hopf algebra with the dual structure from $B$ is
\[
B^\circ=\lbrace f\in B^*\mid \exists I\triangleleft B, \dim B/I<\infty\rbrace,
\]
where $I$ are Hopf ideals. However, $C=B^\circ$ is not necessarily perfectly paired with $B$. Using finite-dimensional representations, we can describe $B^\circ$ as the Hopf algebra of \emph{matrix coefficients} (see e.g. \cite{BG}). From this description we obtain the condition that $B^\circ$ and $B$ are perfectly paired (with respect to the natural pairing) if and only if no element of $B$ acts by zero on all finite-dimensional modules of $B$. As we will often work with positively graded Hopf algebras (for example, studying Nichols algebras), we will include the following Lemma:
\end{remark}

\begin{lemma}
If $B$ is a positively graded Hopf algebra with finite-dimensional graded pieces, then $B^\circ$ and $B$ are perfectly paired.
\end{lemma}
\begin{proof}
Write $B=\oplus_{n\geq 0}B_n$ and consider the finite-dimensional module $B_{<k} :=B/{\oplus_{n\geq k}B_n}$. For each $b\in B$, we find that $b\in \oplus_{n\leq k}B_n$ for some $k$. Then $0\neq b\cdot 1\in B_{<k+1}$, so we have found a finite-dimensional module on which $b$ acts non-trivially.
\end{proof}

\begin{lemma}\label{representabilitylemma}
For $C$ and $B$ dually (not necessarily perfectly) paired bialgebras in $\cB$, the forgetful functor $\leftexp{C}{\mathcal{YD}}^B(\cB)\to \Vect$ is higher representable on the vector space $C\otimes H\otimes B$. The same holds true for the functor $\leftexp{C}{\cH}^B(\cB)\to \Vect$.
\end{lemma}

\begin{proof}
The structure of an object of the categories of YD- or Hopf-modules over $B,C$ can be encoded by the action maps of $B$, $C$ and $H$. The compatibility conditions and all other axioms are expressed by conditions involving the action of these vector spaces on modules. This basic observation can be used to show (higher) representability.
\end{proof}

\begin{definition}\label{doubledefinition}$~$\nopagebreak
\begin{enumerate}
\item[(i)] The \emph{braided Drinfeld double} $\Drin_H(C,B)$ is the Hopf algebra obtained from $\leftexp{C}{\mathcal{YD}}^B(\lmod{H})$ by reconstruction on $C\otimes H\otimes B$. If $A,H$ have a weak quasitriangular structure, denote the algebra obtained by reconstruction on $C\otimes H\otimes B$ from $\leftexp{C}{\mathcal{YD}}^B(\lcomod{A^{\oop}})$ by $\Drin_A(C,B)$.
\item[(ii)] The \emph{braided Heisenberg double} $\Heis_H(C,B)$ is the algebra obtained by reconstruction on $C\otimes H\otimes B$ from $\leftexp{C}{\mathcal{H}}^B(\cB)$. In the weak quasitriangular case, denote the resulting algebra by $\Heis_A(C,B)$.
\end{enumerate}
\end{definition}

Note that a PBW-decomposition is given by construction for these algebras. Moreover, we have that
\begin{align*}
\lmod{\Drin_H(C,B)}&\cong \leftexp{C}{\mathcal{YD}}^B(\lmod{H}),\\
\lmod{\Heis_H(C,B)}&\cong \leftexp{C}{\mathcal{H}}^B(\lmod{H}),
\end{align*}
where the first equivalence is one of monoidal categories. Note that in the case $\cB=\lcomod{A}$, such equivalences do \emph{not} hold if $A$, $H$ are infinite-dimensional. The reason is that in $\leftexp{C}{\mathcal{YD}}^B(\lcomod{A^{\oop}})$ the $H$-actions are induced by $A$-coactions. If the pairing of $A$ and $H$ is perfect, then they correspond to all locally finite (integrable) modules. But modules in $\Drin_A(C,B)$ can be more general (compare this to the requirement of studying weight modules of the quantum group).

\begin{remark}
We choose the notation $\Drin_H(C,B)$ indicating both $B$ and $C$. This is because for given $B$ we can consider different dually paired Hopf algebras (which may not be perfectly paired). This gives a more flexible definition allowing the treatment of algebras which have no perfectly paired dual Hopf algebra.
\end{remark}

Before providing examples, we will write out explicit presentations for the abstractly defined doubles.


\subsection{Explicit Presentations}\label{presentations}

In the following, we will use Sweedler's notation to write down presentations for the algebras just defined. For this, we denote the coproducts by $\Delta(x)=x_{(1)}\otimes x_{(2)}$ (for $x$ an element of $H$, $B$ or $C$). Note that in the case of $B$ and $C$ these are coproducts in the braided monoidal category $\lmod{H}$ which are often referred to as \emph{braided} coproducts. We write $x^{(-1)}\otimes x^{(0)}$ for left coactions. In this notation, summation over tensors is omitted. We maintain to use the notation $\triangleright$ and $\triangleleft$ for actions. When clarification is needed, we denote the products in $H$, $B$ or $C$ by $\cdot$ with a lower index indicating the algebra. We denote the $R$-matrix of $H$ by $R=R^{(1)}\otimes R^{(2)}$ and its convolution inverse by $R^{-1}=R^{-(1)}\otimes R^{-(2)}$.

\begin{proposition}\label{presentationprop1}
The algebra $\Drin_H(C,B)$ is generated by the subalgebras $H$, $B^{op}$ (meaning that $bb'=b'\cdot_B b$) and $C$ subject to the following relations:
\begin{align}
hc&=(h_{(1)}\triangleright c)h_{(2)}, &(\Leftrightarrow  ch&=h_{(2)}(S^{-1}h_{(1)}\triangleright c))\label{boso1}\\
hb&=(h_{(2)}\triangleright b)h_{(1)}, &(\Leftrightarrow bh&=h_{(1)}(Sh_{(2)}\triangleright b))\label{boso2}
\end{align}
\begin{align}
b_{(2)}R^{(1)}c_{(2)}\ev(c_{(1)}\otimes (R^{(2)}\triangleright b_{(1)}))&= c_{(1)}R^{(2)}b_{(1)}\ev((R^{(1)}\triangleright c_{(2)})\otimes b_{(2)}).\label{drincross}
\end{align}
The coproducts are given by
\begin{align}
\Delta(h)&=h_{(1)}\otimes h_{(2)},
&\Delta(b)&=(R^{(2)}\triangleright b_{(1)})\otimes b_{(2)}R^{(1)},
&\Delta(c)&=c_{(1)}R^{(2)}\otimes R^{(1)}\triangleright c_{(2)}.
\end{align}
The counit is simply given by $\varepsilon(chb)=\varepsilon(c)\varepsilon(h)\varepsilon(b)$. The antipode and inverse antipode are the anti-algebra morphisms given by:
\begin{align}
S(h)&=S(h),&S^{-1}(h)&=S^{-1}(h),\\
S(b)&=R^{-(2)}(R^{-(1)}\triangleright Sb),&S^{-1}(b)&=(R^{-(1)}\triangleright S^{-1}b)R^{-(2)},\\
S(c)&=R^{-(1)}(R^{-(2)}\triangleright Sc),&S^{-1}(c)&=(R^{-(2)}\triangleright S^{-1}c)R^{-(1)}.
\end{align}
The algebra $\Heis_H(C,B)$ has the same algebra bosonization relations (\ref{boso1})-(\ref{boso2}) as $\Drin_H(C,B)$, but relation (\ref{drincross}) (referred to as the cross relation) is replaced by
\begin{equation}
cb=b_{(2)}R^{(1)}c_{(2)}\ev(c_{(1)}\otimes (R^{(2)}\triangleright b_{(1)})).\label{heiscross}
\end{equation}
\end{proposition}
\begin{proof}
These formulas are obtained using reconstruction in $\Vect$. For $h\in H$, $b\in B$ and $c\in C$, we define the action by
\begin{equation}
(c\otimes h\otimes b)\triangleright v:=b\triangleright(h\triangleright(v\triangleleft b)).
\end{equation}
The formulae for the antipode are required by defining e.g. $Sb$ to be the element of $\Drin_H(C,B)$ which satisfies $\ev(b\triangleright f\otimes v)=(f\otimes S(b)\triangleright v)$ for all $f\in V^\circ$ and $v\in V$ where $V^\circ$ is the finite dual of the space $V$.
\end{proof}

\begin{lemma}
The cross relation (\ref{drincross}) in $\Drin_H(C,B)$ is equivalent to each of the following relations:
\begin{align}
\begin{split}cb=&R^{-(2)}_1b_{(2)}R^{(2)}(R_2^{-(2)}\triangleright c_{(2)})\ev(R_2^{-(1)}R_3^{-(1)}\triangleright c_{(3)}\otimes R^{-(1)}_1\triangleright S^{-1}(b_{(3)}))\\&\ev(R^{-(2)}_3\triangleright c_{(1)}\otimes R^{(1)}\triangleright b_{(1)}),\end{split}\\
\begin{split}bc=&(R_1^{-(1)}\triangleright c_{(2)})R_1^{(2)}b_{(2)}R_2^{(1)}\ev(R_2^{-(1)}R_1^{(1)}\triangleright c_{(3)}\otimes b_{(3)})\\&\ev(R_1^{-(2)}R_2^{-(2)}\triangleright S^{-1}(c_{(1)})\otimes R^{(2)}_2\triangleright b_{(1)}).\end{split}
\end{align}
\end{lemma}

\begin{proof}
Apply reconstruction to Lemma~\ref{altconds} after reinterpreting the $B^{\oop}$-coaction as a $C$-action using the functor $\leftsub{B^{\oop}}{\Phi}$.
\end{proof}

Using this Lemma, we can write down general product formulas in PBW form in $\Drin_H(C,B)$:
\begin{align}
\begin{split}
chb\cdot c'h'b'=&c(h_{(1)}R_1^{-(1)}\triangleright c'_{(2)})h_{(2)}R_1^{(2)}R_2^{(1)}h'_{(1)}b'(S(h'_{(2)})R_3^{-(1)}\triangleright b_{(2)})
\\&\ev(R_2^{-(1)}R_1^{(1)}\triangleright c'_{(3)}\otimes b_{(3)})\ev(R_1^{-(2)}R_2^{-(2)}\triangleright S^{-1}(c'_{(1)})\otimes R^{-(2)}_3R^{(2)}_2\triangleright b_{(1)})
\end{split}
\end{align}

The general product formula  for $\Heis_H(C,B)$ in PBW form is:
\begin{equation}
\begin{split}
chb\cdot c'h'b'=&c(h_{(1)}R_1^{-(1)}\triangleright c'_{(2)})h_{(2)}R^{(1)}h'_{(1)}b'(S(h'_{(2)})R_2^{-(1)}\triangleright b_{(2)})\\
&\ev(R_1^{-(2)}\triangleright S^{-1}(c'_{(1)})\otimes R^{(2)}R_2^{-(2)}\triangleright b_{(1)}).
\end{split}
\end{equation}

\begin{example}\label{weylalgebras}
Let $X=\mA^n$ and $B=\mC[x_1,\ldots,x_n]$ its coordinate ring. We denote its restricted dual (as a Hopf algebra) by $C=\Theta_{X}=\mC[\partial_1,\ldots,\partial_n]$, where $\ev(\partial_i,x_j)=\delta_{i,j}$. Both $B$ and $C$ are primitively generated Hopf algebras over $k$ and perfectly paired via $\ev$. One easily sees that $\Heis(\cO_X)=\cD_X=A_n$ is the ring of differential operators on $X$, the \emph{nth Weyl algebra}. The Drinfeld double is simply $\Drin(\cO_X)=\mC[x_1,\dots,x_n,\partial_1,\dots,\partial_n]$ which can be identified with $\cO_{T^*X}$, the ring of functions on the tangent space. Hence there is an action of $\lmod{\cO_{T^*X}}$ on $\lmod{\cD_X}$.
\end{example}

We have an analogue of Proposition~\ref{presentationprop1} in the case where $\cB=\lcomod{A}$. Recall that in this case, we have convolution invertible morphisms of Hopf algebras $R, \overline{R}\colon A^{\oop}\to H$ such that
\begin{align}
R(a\otimes a')&=\ev(R^{-1}(a')\otimes a)=\ev(\overline{R}(a)\otimes a')\\
R^{-1}(a\otimes a')&=\ev(\ov{R}^{-1}(a)\otimes a')=\ev(R(a')\otimes a).
\end{align}

\begin{proposition}\label{presentationprop2}
The Hopf algebra $\Drin_A(C,B)$ is generated by the subalgebras $H$, $B^{\op{op}}$ (opposite product in $\Vect$) and $C$ with the same bosonization relations (\ref{boso1})-(\ref{boso2}) as $\Drin_H(C,B)$ and cross relation
\begin{equation}
b_{(2)}\overline{R}({b_{(1)}}^{(-1)})c_{(2)}\ev(c_{(1)}\otimes {b_{(1)}}^{(0)})=c_{(1)}R^{-1}({c_{(2)}}^{(-1)})b_{(1)}\ev({c_{(2)}}^{(0)}\otimes b_{(2)}).
\end{equation}
The coproducts are given by 
\begin{align}
\Delta(h)&=h_{(1)}\otimes h_{(2)},
&\Delta(b)&={b_{(1)}}^{(0)} \otimes b_{(2)}\overline{R}({b_{(1)}}^{(-1)}),&\Delta(c)&=c_{(1)}R^{-1}({c_{(2)}}^{(-1)})\otimes {c_{(2)}}^{(0)}.
\end{align}
The unit and counit are as before and the formulas for the antipode and inverse antipode are given by
\begin{align}
S(h)&=S(h),&S^{-1}(h)&=S^{-1}(h),\\
S(b)&=R(b^{(-1)})S(b^{(0)}),&S^{-1}(b)&=S^{-1}(b^{(0)})R(b^{(-1)}),\\
S(c)&=\ov{R}^{-1}(c^{(-1)})S(c^{(0)}),&S^{-1}(c)&=S^{-1}(c^{(0)})\ov{R}^{(-1)}(c^{(-1)})
\end{align}
The cross-relation of $\Heis_A(C,B)$ is given by
\begin{equation}
cb=b_{(2)}\ov{R}({b_{(1)}}^{(-1)})c_{(2)}\ev({c_{(1)}}\otimes {b_{(1)}}^{(0)}).
\end{equation}
\end{proposition}
\begin{proof}
All expressions for $\Drin_H(C,B)$ (or $\Heis_H(C,B)$) can be translated into expressions for $\Drin_A(C,B)$ (or $\Heis_A(C,B)$) by using the following rules which are derived from the weak quasitriangular structure on $A$, $H$:
\begin{align}
R^{(1)}\otimes R^{(2)}\triangleright v&=\ov{R}(v^{(-1)})\otimes v^{(0)},\\
(R^{(1)}\triangleright v)\otimes R^{(2)}&=v^{(0)}\otimes R^{-1}(v^{(-1)}),\\
R^{-(1)}\otimes R^{-(2)}\triangleright v&=\ov{R}^{-1}(v^{(-1)})\otimes v^{(0)},\\
(R^{-(1)}\triangleright v)\otimes R^{-(2)}&=v^{(0)}\otimes R(v^{(-1)}).\qedhere
\end{align}
\end{proof}

\begin{example}\label{classicalexample}
Let us consider the case where $H=k$ and $B$ is a finite-dimensional Hopf algebra over $k$ to discover the classical notions. In this case, $\Drin(H):=\Drin_k(H^*,H)$ is given on $H^*\otimes H^{op}$ with product determined by the relation
\begin{equation}
cb=b_{(2)}c_{(2)}\ev(c_{(1)}\otimes b_{(1)})\ev(c_{(3)}\otimes S^{-1}(b_{(3)})).
\end{equation}
The coproduct, counit and antipode are simply the corresponding tensor product structures on $H^*\otimes H$. The $R$-matrix on $\Drin(H)$ is given by the coevaluation map of $H$.
Note that there are two conventional differences of this definition to the definition of $\Drin(H)$ usually found in the literature (cf. e.g. \cite[7.1]{Maj1}). At first, $H$ and $H^*$ are \emph{categorically} dually paired. That is $\ev(aa'\otimes bb')=\ev(a\otimes b')\ev(a'\otimes b)$. Second, the left modules are left-right Yetter-Drinfeld modules (that is, YD-compatible \emph{right}-$B$ action with a \emph{left} $B$-coaction). 

For this reason, we will often consider the Drinfeld double
$\Drin(H^{\op{op}})$ which is by definition $\Drin_k(H^{*\op{cop}},H^{\op{op}})$. 
This Hopf algebra has the property that
\[
\lmod{\Drin(H^{\op{op}})}=\leftexpsub{H}{H}{\cal{YD}}.
\]
It has $H$, $H^*$ as subalgebras such that
\begin{equation}
cb=b_{(2)}c_{(2)}\ev(c_{(3)}\otimes b_{(1)})\ev(c_{(3)}\otimes S(b_{(1)})).
\end{equation}
The coproducts are given by $\Delta(b)=b_{(1)}\otimes b_{(2)}$ and $\Delta(c)=c_{(2)}\otimes c_{(1)}$. The universal R-matrix for $\Drin(H)$ is $\tau\coev_H$, where $\tau$ is the symmetric braiding in $\Vect$. We will write $\coev_H=e_\alpha\otimes f_\alpha\in H\otimes H^*$. The algebra $\Drin(H^{\op{op}})$ recovers the classical Drinfeld double of $H$ (as found in the literature, cf. e.g. \cite{Kas}) and has the same categorical interpretation.

For example, consider the Drinfeld double $\Drin(G^{\op{op}}):=\Drin(k[G]^{\op{cop}},kG^{\op{op}})$, for $G$ a finite group. It is generated by $kG$ and the algebra of $k$-valued functions $k[G]$ (basis $\delta_h(g)=\delta_{h,g}$). Note that the coproduct is
$\Delta(\delta_h)=\sum_{ab=g}{\delta_a\otimes \delta_b}$.
The relations in this algebra are $g\delta_h=\delta_{ghg^{-1}}g$.
\end{example}

\begin{example}\label{primitivelygen}
A vast class of examples of braided Hopf algebras is given by \emph{Nichols algebras}. These play an important part in the classification of \emph{Pointed Hopf algebras} (see e.g. \cite{AS} for a survey). An important feature is that they are primitively generated. That is, $B$ is generated by elements $b\in B$ such that $\Delta(b)=b\otimes 1+1\otimes b$, and the same is true for $C$. For such braided Hopf algebras, we obtain simpler formulae as the cross relations will turn out to be commutator relations. Denote the space of primitive elements of $B$ by $P(B)$ and similarly the space of primitive elements in $C$ by $P(C)$. Then $\Drin_H(C,B)$ is generated by $H$, $P(B)$ and $P(C)$ with respect to the bosonization relations (\ref{boso1}) and (\ref{boso2}) and the cross relation
\begin{equation}\label{primitivedrincross}
[b,c]=R^{(2)}\ev(R^{(1)}\triangleright c\otimes b)-R^{-(1)}\ev(R^{-(2)}\triangleright c\otimes b),
\end{equation}
for $b\in P(B)$, $c\in P(C)$.
The coproducts are given on the generators by
\begin{align}
\Delta(h)&=h_{(1)}\otimes h_{(2)},\\
\Delta(b)&=1\otimes b+(R^{(2)}\triangleright b)\otimes R^{(1)},\\
\Delta(c)&=c\otimes 1+R^{(2)}\otimes R^{(1)}\triangleright c.
\end{align}
The condition (\ref{heiscross}) is equivalent to the commutator relation
\begin{equation}
[c,b]=R^{(1)}\ev(c\otimes R^{(2)}\triangleright b),
\end{equation}
for $c\in P(C)$ and $b\in P(B)$.
Working over $\Drin(H)$ for an ordinary Hopf algebra $H$, we can view $B\in \Hopf(\lmod{\Drin(H)})$ as a YD-module over $H$ (an object of $\cZ(\rmod{H})$. Then the relation (\ref{heiscross}) is equivalent to
\begin{equation}\label{restrictedprimitive}
[c,b]=b^{(-1)}\ev(c\otimes b^{(0)}).
\end{equation}
To compare the definition of the braided Heisenberg double from \cite[Section~5]{BB}, we have to apply algebra reconstruction on $C\otimes H\otimes B$ (rather than $C\otimes \Drin(H)\otimes B$).
That is, we consider the subalgebra generated by $H$, $B$ and $C$ of $C\otimes \Drin(H) \otimes B$. We include this \emph{restricted} version of the braided Heisenberg double.
\end{example}

\begin{lemma}\label{restrictedheis}
Let $B\in \Hopf(\cZ(\rmod{H}))$, then the \emph{restricted} braided Heisenberg double $\overline{\Heis}_H(C,B)$ of $H$ over $B$ is the algebra generated by $H$, $B^{op}$ and $C$ subject to the relations
\begin{align}
hb&=(h_{(2)}\triangleright b)h_{(1)},\\
hc&=(h_{(1)}\triangleright c)h_{(2)},\\
cb&=b_{(2)}{b_{(1)}}^{(-1)}c_{(2)}\ev(c_{(1)}\otimes {b_{(1)}}^{(0)}).
\end{align}
It is the subalgebra of $\Heis_{\Drin(H)}(C,B)$ generated by $H$, $B^{op}$ and $C$. 
\end{lemma}
Note that $H$ is assumed to be finite-dimensional for $\Drin(H)$ to be quasitriangular. We observe that there is a restriction functor $\lmod{\Heis_{\Drin(H)}(C,B)}\to \lmod{\overline{\Heis}_H(C,B)}$. 

Now the braided Heisenberg double defined in \cite[Section~5]{BB} is isomorphic to $\ov{\Heis}_H(C^{\op{cop}}, B^{\op{op}})$ where the (co)opposites are taking in $\Vect$, and $B,C$ are Nichols algebras. The cross relation for this algebra is (\ref{restrictedprimitive}).

\begin{corollary}\label{restrictedaction}
The coproduct on $\Drin_{\Drin(H)}(C,B)$ induces a left action of the monoidal category $\lmod{\Drin_{\Drin(H)}(C,B)}$ on $\lmod{\overline{\Heis}_H(C,B)}$.
\end{corollary}
\begin{proof}
For this to hold, we need to ensure that the restriction of the coproduct $\Delta$ viewed as an algebra map $\ov{\Heis}_{H}(C,B)\to \Drin_{\Drin(H)}(C,B)\otimes \Heis_{\Drin(H)}(C,B)$ maps to $\Drin_{\Drin(H)}(C,B)\otimes \overline{\Heis}_H(C,B)$. This can be checked on generators. Clearly $\Delta(h)\in H\otimes H$ satisfies this. Further, $\Delta(b)={b_{(1)}}^{(0)}\otimes b_{(2)}{b_{(1)}}^{(-1)}$ and $\Delta(c)=c_{(1)}f_\alpha\otimes  e_\alpha\triangleright c_{(2)}$ both lie in the subspace $(C\otimes H^*\otimes H\otimes B)\otimes (C\otimes H\otimes B)$. 
\end{proof}

\begin{example}\label{demonstrationexpl}
In \ref{weylalgebras}, the Drinfeld double of $\cO_X$ collapses to simply be the tensor product Hopf algebra of $\cO_X$ and its dual. We will now consider the symmetric algebra as a Hopf algebra object in $\leftexpsub{C_2}{C_2}{\cal{YD}}$ for the group $C_2$ with two elements. We can also include the alternating algebra in the same setting as an example.

Let $V$ be a finite-dimensional vector space. Consider the braidings $\tau\colon V\otimes V\to V\otimes V$, where $\tau(v\otimes w)=w\otimes v$, and $-\tau$. These can naturally be realized as braidings coming from $C_2$-YD module structures on $V$, corresponding to $\delta(v)=s\otimes v$ where $C_2=\lbrace 1,s\rbrace$ with the trivial action for $\tau$, respectively the sign representation $sv=-v$ for $-\tau$. Using this, we have that
\begin{align*}
S(V)&=\bigslant{T(V)}{(v\otimes w-w\otimes v)}=\bigslant{T(V)}{\ker(\ide+\tau)},\\
\Lambda(V)&=\bigslant{T(V)}{(v\otimes w+w\otimes v)}=\bigslant{T(V)}{\ker(\ide-\tau)}
\end{align*}
are braided Hopf algebras in $\leftexpsub{C_2}{C_2}{\cal{YD}}=\lmod{\Drin(C_2)}$. As they are Nichols algebras, they are perfectly paired with $S(V^*)$ (respectively $\Lambda(V^*)$).

The restricted Heisenberg double $\ov{\Heis}_{kC_2}(S(V^*),S(V))$ is now $A_n\otimes kC_2$, but $\Drin_{\Drin(C_2)}(S(V^*),S(V))$ has the non-trivial commutator relation
\begin{equation}
[f,v]=(s-\delta_1-\delta_s)\ev(f\otimes v).
\end{equation}
Hence viewing $S(V)$ over $\Drin(C_2)$ causes the resulting Drinfeld double to be non-commutative.

For the exterior algebra, we have that $\Heis_{kC_2}(\Lambda(V^*),\Lambda(V))$ is generated by $v\in V$, $f\in V^*$ and $1,s$ with additional relations
\begin{equation}
sv=-vs, \quad sf=-fs, \quad [f,v]=\ev(f\otimes v).
\end{equation}
In $\Drin_{\Drin(C_2)}(\Lambda(V^*),\Lambda(V))$ the commutator relation is
\begin{equation}
[f,v]=(s-\delta_1+\delta_s)\ev(f\otimes v).
\end{equation}
\end{example}

\begin{example}[The quantum groups]\label{quantumgroups}
Majid constructs Lusztig's version of $\Ug$ (for generic $q$) as braided Drinfeld doubles in \cite{Maj2}. We repeat the construction with the conventions of this paper. For this, a (symmetric) Cartan datum $\cdot$ gives a paring of two lattice group algebras $A=k\mZ[I]=k[g_i^{\pm 1}]$ and $H=k[K_i^{\pm 1}]$ via $\langle K_i,g_j\rangle=q^{2\tfrac{i\cdot j}{i\cdot i}}$ (where $k=\mC(q)$). The dual R-matrix $R(g_i,g_j)=q^{i\cdot j}$ is given by a weak quasitriangular structure with $R(g_i)=K^{-\tfrac{i\cdot i}{2}}$ (and $\ov{R}=R^{-1}$). We can then define an $A$-comodule $F=k\langle F_1,\ldots,F_n\rangle$ by $F_i\mapsto g_i^{-1}\otimes F_i$ and denote its dual by $E=k\langle E_1,\ldots,E_n\rangle$. The tensor algebras (in $\lcomod{A}$) $B=T(F)$, $C=T(E)$ are dually paired via
\begin{equation}
\ev(E_i,F_i)=\frac{\delta_{i,j}}{(q_i^{-1}-q_i)},
\end{equation}
where $q_i:=q^{\tfrac{i\cdot i}{2}}$. This pairing extends uniquely to a Hopf algebra pairing of the tensor algebras and it is a result by Lusztig (see \cite{Lus}) that then left and right radical of this pairing are precisely the quantum Serre relations. Hence the quotients are $U_q(\fr{n}_+)$ (of $C$) and $U_q(\fr{n}_-)$ (of $B$), which are Nichols algebras. Now translating (\ref{primitivedrincross}) under \ref{presentationprop2} we obtain the correct relation
\begin{equation}
[E_i,F_j]=\frac{K_i^{\tfrac{i\cdot i}{2}}-K_i^{-\tfrac{i\cdot i}{2}}}{(q_i-q_i^{-1})}\delta_{i,j}.
\end{equation}
Further, the bosonization relations (\ref{boso1}) and (\ref{boso2}) give
\begin{equation}
K_iE_j=q^{2\tfrac{i\cdot j}{i\cdot i}}E_jK_i, \quad K_iF_j=q^{-2\tfrac{i\cdot j}{i\cdot i}}F_jK_i.
\end{equation}
The coproducts are given by
\begin{align}
\Delta(F_i)&=F_i\otimes K^{-\tfrac{i\cdot i}{2}}+1\otimes F_i,
&\Delta(F_i)&=E_i\otimes 1+K^{\tfrac{i\cdot i}{2}}\otimes E_i.
\end{align}
Hence the resulting Hopf algebra is $\Ug$.

The cross relation for the Heisenberg double $D_q(\g):=\Heis_A(U_q(\fr{n}_+),U_q(\fr{n}_-))$ is given by
\begin{equation}
[E_i,F_j]=\frac{K_i^{-\tfrac{i\cdot i}{2}}}{q_i^{-1}-q_i}\delta_{i,j}.
\end{equation}
We will look at the action of $\lmod{\Uq}$ on $\lmod{D_q(\fr{sl}_2)}$ in Section \ref{quantumgroupaction}.
\end{example}


\subsection{R-Matrices for the Drinfeld Double}\label{rmatricessect}

While the Hopf algebra structure of $\Drin_H(C,B)$ can be defined for general infinite-dimensional Hopf algebras with a choice of a dually paired Hopf algebra, viewing the braided structure of $\lmod{\Drin_H(C,B)}$, and hence the $R$-matrix, at this level of generality is problematic as they require the existence of a coevaluation map $\coev\colon I\to B\otimes C$ in $\cB=\lmod{H}$. In the infinite-dimensional case, given a perfect pairing, this map may still exist as a formal power series, given that there is an orthonormal bases of the dual Hopf algebra $C$ to a basis of $B$. 

Recall the braidings from Figure \ref{braidingsimpler}. Given the existence of a coevaluation map $\coev\colon I\to B\otimes C$, the category $\leftexp{C}{\mathcal{YD}}^B(\cB)$ is braided via the braiding
\begin{align}
\leftexp{C}{\Psi}^B&:=(\triangleleft \otimes \triangleright)(\ide\otimes \coev\otimes \ide)\Psi^{-1},\\
(\leftexp{C}{\Psi}^B)^{-1}&:=\Psi(\triangleleft \otimes \triangleright)(\ide\otimes \ide_B\otimes S\otimes \ide)(\ide\otimes \coev\otimes \ide).
\end{align}

Given that $B$ and $C$ are not necessarily finite-dimensional Hopf algebras with a perfect duality pairing $\ev\colon C\otimes B\to k$ such that there is a basis $\lbrace b_i\rbrace_i$ of $B$ with an orthogonal basis $\lbrace c_i \rbrace_i$ of $C$, we can consider the formal power series
$\coev_B:=\sum_i{b_i\otimes c_i}$.
This is \emph{not} a morphism $k\to B\otimes C$, but satisfies the duality axioms of the coevaluation map. Using this formal sum, we can write an $R$-matrix for $\Drin_H(C,B)$ as the power series (omitting summation)
\begin{equation}
R_{\Drin_H(C,B)}:=c_iR^{-(2)}\otimes b_iR^{-(1)}, \quad R_{\Drin_H(C,B)}^{-1}:=R^{(2)}c_i\otimes R^{(1)}S(b_i).
\end{equation}
If $\cB=\lcomod{A}$, the situation is even more problematic. Expressing the braiding
\begin{equation}
\Psi_{V,W}(v,w)=\ev(R^{-1}(v^{(-1)})\otimes w^{(-1)})w^{(0)}\otimes v^{(0)}
\end{equation}
as an $R$-matrix, which is necessary in algebra reconstruction, involves a coevaluation map which will be an infinite sum unless $H$, $A$ are finite-dimensional. We denote it by $\coev_H=h_j\otimes a_j$. Then
\begin{align}
R=a_j\otimes R^{-1}(h_j)=\ov{R}(h_j)\otimes a_j,\qquad R^{-1}=\ov{R}^{-1}(h_j)\otimes a_j=a_j\otimes R(h_j).
\end{align}
Further, the $R$-matrix on $\Drin_A(C,B)$ is given by
\begin{equation}
R_{\Drin_A(C,B)}:=c_iR(h_j)\otimes b_ia_j , \quad R_{\Drin_A(C,B)}^{-1}:=a_jb_i \otimes R^{-1}(h_j)S(c_i).
\end{equation}
To overcome the problematic of the braiding involving formal infinite sums, we will describe the properly braided subcategory of modules over the braided Drinfeld doubles corresponding to the center $\cZ_\cB(\rmod{B}(\cB))$ in Section \ref{Osection}.

\begin{example}
Consider $\Uq$. We can compute that
$\ev(f^n\otimes e^n)=\tfrac{[n]_q!}{(q^{-1}-q)^n}$,
where $[n]_q=1+q^{-2}+\ldots+q^{2(n-1)}$ and $[n]_q!=[1]_q\cdot\ldots\cdot[n]_q$. Hence
\begin{equation}
\coev_B=\sum_{n\geq 0}\frac{(q^{-1}-q)^n}{[n]_q!}F^n\otimes E^n.
\end{equation}
However, writing the required coevaluation map for $I\to H\otimes A$ is problematic. A solution is to introduce the element $q^{h\otimes h}\in (A\otimes A)^*$. This element satisfies
\begin{equation}
\langle q^{h\otimes h}\otimes (g^n\otimes g^m) \rangle=q^{2nm}.
\end{equation}
This can be generalized to any Cartan datum. See \cite{Maj2} for more details.
\end{example}

\begin{example}
The small quantum groups $u_q(\g)$ can also be realized as examples of braided Drinfeld doubles. For this, we assume that $q$ is a primitive $r$th root of unity, where $r$ is odd and $\op{char}k$ does not divide $r$. As underlying Hopf algebras, we take the group algebras $A=H=k(C_r)^n$. 
\end{example}


\subsection{The Category \texorpdfstring{$\cO$}{\emph{O}}}\label{Osection}

The $R$-matrices $R_{\Drin_H(C,B)}$ and $R_{\Drin_A(C,B)}$ from Section~\ref{rmatricessect} will induce the braiding on the subcategories $\cZ_{\lmod{H}}(\rmod{B}(\lmod{H}))$ respectively its subcategory $\cZ_{\lcomod{A}}(\rmod{B}(\lcomod{A}))$. For perfectly paired $B$ and $C$, these Drinfeld centers are \emph{almost} analogues of the BGG-category $\cO$ for the quantum groups. To support this interpretation, recall that for infinite-dimensional dually paired bialgebras $B$ and $C$ in $\lmod{H}$, the essential image of the fully faithful functor
\[
\leftsub{B^{\oop}}{\Psi}\colon\lcomod{B^{\oop}}(\lmod{H})\to \lmod{C}(\lmod{H}) 
\]
consists of those $C$-modules $V$ such that for each $v\in V$, $C/\Ann(v)$ is finite-di\-men\-sional. Such modules are called \emph{locally finite} or \emph{admissible}. Hence, the full subcategory $\cZ_{\lmod{H}}(\rmod{B}(\lmod{H}))\subset \lmod{\Drin_H(C,B)}$ consists of those YD-modules which are locally finite for the $C$-action, but not necessarily for the $B$-action. For example, one can define Verma modules in $\cZ_{\lmod{H}}(\rmod{B}(\lmod{H}))$.

When working in $\cB=\lcomod{A}$, one also has to restrict to $\Drin_A(C,B)$-modules for which the $H$-action is induced by an $A$-coaction, i.e. the $H$-action has to be locally finite too. Using these observations, we conclude the following lemma:

\begin{lemma}\label{catOlemma}
For the quantum group $\Ug$ at $q$ not a root of unity, the subcategory of
\[
\cZ_{\lcomod{k\mZ^n}}(\rmod{U_q(\mathfrak{n}^+)}(\lcomod{k\mZ^n}))
\]
on finitely generated modules is the BGG-category $\cO$ (denoted by $\cO_q(\g)$) for quantum groups as defined in \cite[3.1]{AM}.
\end{lemma}
\begin{proof}
An object $V$ in $\cO_q(\g)$ satisfy three properties:
\begin{enumerate}
\item[(I)]  $V$ is finitely generated over $\Ug$ 
\item[(II)]  $V$ is a weight module, i.e. there exists a direct sum decomposition
\[
V=\oplus_{\lambda\in \mZ^n}{V_\lambda},\qquad V_\lambda=\lbrace v\in V\mid K_i\triangleright v=q^{2\lambda_i} \rbrace.
\]
\item[(III)] $V$ is locally finite as a $U_q(\fr{n}_+)$-module.
\end{enumerate}
If $V$ comes from an object in $\cZ_{\lcomod{k\mZ^n}}(\rmod{U_q(\mathfrak{n}^+)}(\lcomod{k\mZ^n}))$, then (II) is automatically fulfilled because the $H$-action is induced by an $A=k\mZ^n$-coaction. Further, (III) is fulfilled as the $C$-action is induced by an $B^{\oop}$-coaction and hence locally finite. Further, any such module can be obtained from an object of the center as the pairing of $B$ and $C$ is perfect. Thus, if we restrict to finitely generated modules, we recover $\cO_q(\g)$ as a subcategory of the center.
\end{proof}

Hence we can define the category $\cO$ for a braided Drinfeld (or Heisenberg) double as the subcategory of finitely generated modules in $\cZ_{\cB}(\rmod{B}(\cB))$ (respectively $\cH_{\cB}(\rmod{B}(\cB))$). We denote this category by $\cZ^{\op{fg}}_{\cB}(\rmod{B}(\cB))$ (respectively $\cH^{\op{fg}}_{\cB}(\rmod{B}(\cB))$).

\begin{definition}\label{categoryOdef}$~$\nopagebreak
\begin{enumerate}
\item[(a)]
For $\cB=\lmod{H}$, we define the category $\cO^{\Drin}_H(C,B)$ (or $\cO^{\Heis}_H(C,B)$) as the full subcategory of $\lmod{\Drin_H(C,B)}$ (respectively, $\lmod{\Heis_H(C,B)}$) of objects which are
\begin{enumerate}
\item[(I)]   finitely generated as $\Drin_H(C,B)$-modules (resp. $\Heis_H(C,B)$-modules),
\item[(II)]  semisimple as $H$-modules,
\item[(III)] locally finite as $C$-modules.
\end{enumerate}
\item[(b)]
Working with $\cB=\lcomod{A^{\oop}}$, we define the category $\cO^{\Drin}_A(C,B)$ (respectively, $\cO^{\Heis}_A(C,B)$) as the full subcategory of $\lmod{\Drin_A(C,B)}$ (respectively, $\lmod{\Heis_A(C,B)}$) on objects satisfying (I), (III) and
\begin{enumerate}
\item[(II)'] The $H$-module structure is induced by an $A^{\oop}$-comodule structure.
\end{enumerate}
\end{enumerate}
\end{definition}

In the special setting of Section~\ref{rmatricessect} where we have the existence of orthonormal bases $\lbrace e_\alpha \rbrace$ of $B$ and $\lbrace f_\alpha \rbrace$ of $C$, we were able to define a formal power series $\coev=e_\alpha\otimes f_\alpha$ serving as formal coevaluation map. Given this structure, we can link the categories $\cO^{\Drin}_H(C,B)$ (and the other versions) with $\cZ_\cB^{\op{fg}}(\rmod{B}(\cB))$, generalizing Lemma~\ref{catOlemma}.

\begin{theorem}\label{catOthm}
Let $C$, $B$ be perfectly paired Hopf algebras in $\lmod{H}$ such that $\coev$ exists as a formal power series. Then there exist isomorphisms of categories
\begin{align*}
\cO^{\Drin}_H(C,B)&\cong \cZ_{\lmod{H}}^{\op{fg},H-\op{ss}}(\rmod{B}(\lmod{H})),\\
\cO^{\Heis}_H(C,B)&\cong \cH_{\lmod{H}}^{\op{fg},H-\op{ss}}(\rmod{B}(\lmod{H})).
\end{align*}
If $C$, $B$ are such objects in $\lcomod{A^{\oop}}$ where $A,H$ have a weak quasitriangular structure, then
\begin{align*}
\cO^{\Drin}_A(C,B)&\cong \cZ_{\lcomod{A^{\oop}}}^{\op{fg},A-\op{ss}}(\rmod{B}(\lcomod{A^{\oop}})),\\
\cO^{\Heis}_A(C,B)&\cong \cH_{\lcomod{A^{\oop}}}^{\op{fg},A-\op{ss}}(\rmod{B}(\lcomod{A^{\oop}})).
\end{align*}
\end{theorem}
\begin{proof}
The proof works as the proof of Lemma~\ref{catOlemma} in this more general setting. Semisimplicity is not automatic any more, so it needs to be imposed in general (hence the superscripts $H-\op{ss}$ or $A-\op{ss}$). Note that $\coev=e_\alpha\otimes f_\alpha$ can be used to define a $B^{\oop}$-comodule, given a locally finite $C$-module $V$ via $\delta(v) := e_\alpha\otimes f_\alpha\triangleright v$, for all $v\in V$. This expression is always well-defined if $V$ is locally finite (i.e. only finite sums of tensors occur).
\end{proof}

We would like to obtain an action of the category $\cO$ for the Braided Drinfeld double on the category $\cO$ for the braided Heisenberg double. The problem with this is that the action does not necessarily preserve finite-generation of the module. In generality, we therefore need to restrict to finite-dimensional modules.

\begin{corollary}\label{fdcatoaction}
Under the assumptions on $C$, $B$ as in \ref{catOthm}, there are actions
\begin{align*}
\triangleright: \lmod{\Drin_H(C,B)}^{\op{fd}}\otimes \cO^{\Heis}_{H}(C,B)&\longrightarrow \cO^{\Heis}_{H}(C,B),\\
\triangleright: \lmod{\Drin_A(C,B)}^{\op{fd}}\otimes \cO^{\Heis}_{A}(C,B)&\longrightarrow \cO^{\Heis}_{A}(C,B).
\end{align*}
\end{corollary}

Similar to the study of the category $\cO$ in other contexts, we can introduce \emph{standard} (or \emph{Verma}) modules in the category $\cO$ for braided Drinfeld and Heisenberg doubles. For this, a general theory as in \cite{GGOR} could be adapted. We will not introduce such a theory in general here, but provide a definition of a Verma module.

\begin{definition}
Let $B,C$ be braided Hopf algebras in $\cB$ and let  $S$ be a simple $H$-module (respectively $A^{\oop}$-comodule). Then we can define a $C^{\op{op}}\rtimes H$-module as $S$ with trivial $C$ action via the antipode of $C$ and the given $H$-action (if $S$ is an $A$-comodule, we view $S$ as an $H$-module via the pairing). The \emph{Verma module} of $M(S)$ is defined as
$M(S):=B\otimes_{C^{\op{op}}\rtimes H}S$. 
This can either be done as a module over $\Drin_{\cB}(C,B)$ or $\Heis_{\cB}(C,B)$ (for $\cB=\lmod{H}$ or $\lcomod{A^{\oop}}$).
\end{definition}

It is not guaranteed that all $M(S)$ are in the respective category $\cO$. However, under additional assumptions, such that the $B$ and $C$ are positively graded, by a inner grading as in \cite{GGOR}, the theory developed therein ensure that they are. In general, we can say that $M(S)$ are in the finitely-generated subcategory of $\cZ_\cB(\rmod{B}(\cB))$ (respectively $\cZ_\cB(\rmod{B}(\cB))$ in the Hopf case). One of the main uses is that all simple objects in $\cO$ appear as quotients of such Verma modules.


\subsection{Cocycle Twists}\label{cocycletwists}

In this section, we observe that the left action of the category $\cZ_\cB(\rmod{B}(\cB))$ on $\cH_\cB(\rmod{B}(\cB))$ extends to an action of $\lmod{\Drin_H(C,B)}$ on $\lmod{\Heis_H(C,B)}$. This action implies that $\Heis_H(C,B)$ is a right cocycle twist of the Hopf algebra $\Drin_H(C,B)$ generalizing an earlier result of \cite{Lu} to the braided case. In particular, $\Heis_H(C,B)$ is a left $\Drin_H(C,B)$-comodule coalgebra.

\begin{definition}
Let $B\in \BiAlg(\cB)$, for $\cB$ a braided monoidal category. A \emph{right 2-cocycle of $B$} is a morphism $\sigma \colon B\otimes B\to I$ such that the morphism $m_\sigma:=m^2_{B\otimes B}(\Delta\otimes \Delta\otimes \sigma)\colon B\otimes B\to B$ gives $B$ an algebra structure which is denoted by $B_\sigma$ and called the \emph{right cocycle twist of $B$} by $\sigma$.
\end{definition}

It is easy to check that a morphism $\sigma\colon B\otimes B\to I$ is a right 2-cocycle if and only if it satisfies the condition
\begin{align}\label{2cocycle}
\begin{split}
&(\sigma \otimes \sigma)(m \otimes \Psi\otimes \ide_B)(\ide_B\otimes \Psi\otimes \Psi)(\Delta\otimes \Delta\otimes \ide_B)\\&=(\sigma\otimes \sigma)(\ide_B \otimes m \otimes \ide_{B\otimes B})(\ide_{B\otimes B}\otimes \Psi\otimes \ide_B)(\ide_B\otimes \Delta\otimes \Delta)
\end{split}
\end{align}
on $B\otimes B\otimes B$, as well as the normalization conditions
$\sigma(\ide\otimes 1)=\sigma(1\otimes \ide)=\varepsilon$.

\begin{lemma}\label{coactionlemma}
The coproduct of $B$ can be viewed as a morphism $\Delta\colon B_\sigma\to B\otimes B_\sigma$ making $B_\sigma$ a left $B$-comodule algebra in $\cB$.
\end{lemma}
\begin{proof}
This is an exercise in graphical calculus.
\end{proof}

\begin{corollary}\label{twistaction}
If $\sigma$ is a right 2-cocycle, then $\Delta$ gives a left action of $\lmod{B}(\cB)$ on $\lmod{B_\sigma}(\cB)$ where the $B_\sigma$-action on $V\triangleright W$ (defined on the object $V\otimes W$) for a $V$-module $V$ and a $B_\sigma$-module $W$ is given by
\begin{equation}
\triangleright_{V\otimes W}=(\triangleright_V\otimes \triangleright_W)(\ide_B\otimes \Psi_{B, V}\otimes \ide_W)(\Delta\otimes V\otimes W).
\end{equation}
That is, given by the coproduct of $B$.
\end{corollary}

We are looking for a converse of the statement in Corollary~\ref{twistaction}. Indeed, the following holds:

\begin{proposition}\label{conversetwist}$~$\nopagebreak
\begin{enumerate}
\item[(i)]
Let $D\in \BiAlg(\cB)$ and $H\in \Alg(\cB)$ such that $\cD=\lmod{D}(\cB)$ acts on $\cH=\lmod{H}(\cB)$ on the left. Assume further that the square
\begin{equation}
\begin{diagram}
\lmod{D}(\cB)\otimes \lmod{H}(\cB)&\rTo^{\triangleright}& \lmod{H}(\cB)\\
\dTo^{F_{\cD}\otimes F_{\cH}}&&\dTo^{F_{\cH}}\\
\Vect\otimes \Vect &\rTo^{\otimes}&\Vect
\end{diagram}
\end{equation}
commutes and that $\Nat(V,F_\cB\otimes F_\cH)$ is reconstructible by $D\otimes H$ (cf. \ref{higherrep}). Then the action is given by a morphism $\delta\colon H\to D\otimes H$ which makes $H$ an algebra object in $\lcomod{D}(\cB)$.
\item[(ii)]
If $F(H)=F(D)$ in $\cB$ and the coaction $\delta$ is given by the coproduct $\Delta$ of $D$, then the product on $H$ is a right cocycle twist by the 2-cocycle $\varepsilon_Dm_H$.
\end{enumerate}
\end{proposition}

\begin{proof}
Part (i) is an argument in reconstruction theory. By representability assumptions, the action $\triangleright$ is given by a map $\delta\colon H\to D\otimes H$. The requirement that $\triangleright$ is an action in the category $\cH$ translates to the property that $H$ is an algebra object in $\lcomod{D}(\cB)$.

To prove part (ii), observe that
\begin{align*}
m_H&=(\ide_D\otimes \varepsilon)\Delta_D m_H=(\ide_D\otimes \varepsilon)\delta m_H\\
&=(\ide_D\otimes \varepsilon)(m_D\otimes m_H)(\ide_D\otimes \Psi_{D,D}\otimes \ide_H)(\Delta_D\otimes \delta),\\
&=(\ide_D\otimes \varepsilon)(m_D\otimes m_H)(\ide_D\otimes \Psi_{D,D}\otimes \ide_D)(\Delta_D\otimes \Delta)
\end{align*}
where we use that that $\delta \colon H\to D\otimes H$ is a morphism of algebras and that $\delta=\Delta$ as maps in $\cB$. Hence $H=D_\sigma$ for $\sigma=\varepsilon m_H$.
\end{proof}

Returning to the situation of Theorem~\ref{cataction}, we observed that the categorical action can be extended to an action of $\leftexp{C}{\mathcal{YD}}^B(\cB)$ on $\leftexp{C}{\mathcal{H}}^B(\cB)$ in Corollary \ref{duallypairedaction}. The action is given by the coproducts of $B$ and $C$. For $\cB=\lmod{H}$ this implies that the action of $\lmod{\Drin_H(C,B)}$ on $\lmod{\Heis_H(C,B)}$ is given by the coproduct of $\Drin_H(C,B)$ viewed as a morphism of algebras
\[
\Delta\colon \Heis_H(C,B)\to \Drin_H(C,B)\otimes \Heis_H(C,B).
\]
The same observation applies when working with $\lcomod{A^{\oop}}$ instead.

\begin{corollary}\label{heistwist}
The algebra $\Heis_H(C,B)$ is a right cocycle twist of the Hopf algebra $\Drin_H(C,B)$ via the 2-cocycle given by $\sigma(chb\otimes c'h'b')=\varepsilon(c)\varepsilon(h)\ev(S^{-1}c'\otimes b)\varepsilon(h')\varepsilon(b')$.
\end{corollary}

\begin{remark}
Note that we can also define left 2-cocycles and a left twist (or cycles and cycle twists). A dual $R$-matrix for $B$ gives a left 2-cocycle on $B$. This cocycle was used in \cite{Lu} to twist the Drinfeld double. Note that we use a different cocycle here but they coincide in the case $H=k$.
\end{remark}


\subsection{The Braided Heisenberg Double of \texorpdfstring{$U_q(\fr{n}_+)$}{\emph{Uq(n+)}} for \texorpdfstring{$\fr{sl}_2$}{\emph{sl2}}}\label{quantumgroupaction}

In Example \ref{quantumgroups}, we introduced the braided Heisenberg double $D_q(\fr{g})$ of $U_q(\fr{n}_+)$. We now want to study the categorical action of $\lmod{D_q(\fr{sl}_2)}^{\op{fd}}$ on the category $\cO_q^{\Heis}(\fr{sl}_2):= \cO^{\Heis}_{k[g^{\pm 1}]}(U(\fr{sl}_2)^*,U(\fr{sl}_2))$ as an example. We first observe that $D_q(\fr{sl}_2)$ has no finite-dimensional modules using a standard argument from the theory of highest weight representations.

\begin{lemma}
In $D_q(\fr{sl}_2)$, the commutator relation
\begin{equation}\label{sl2commutator}
[E,F^m]=\frac{[m]_q K^{-1}}{q^{-1}-q}F^{m-1}.
\end{equation}
holds, where $[n]_q=1+q^{-2}+q^{-4}+\ldots + q^{-2(n-1)}$.
\end{lemma}
\begin{proof}
By induction on $m$, using that $[E,F^m]=[E,F^{m-1}]F+F^{m-1}[E,F]$.
\end{proof}

First, we note that in the $\fr{sl}_2$ case, the pairing $\langle g,K\rangle=q$ is perfect for generic $q$. Hence, the functor 
\[
\Phi\colon \lcomod{k[g^{\pm 1}]}\longrightarrow \lmod{k[K^{\pm 1}]}
\]
induced by the pairing is fully faithful. The essential image of $\Phi$ is the semisimple category generated by one-dimensional simples $k_{q_n}=kv_n$, where $K\triangleright v_n=q^n v_n$. Consider the Verma module $M(n):=M(k_{q^n})$ of weight $q^n$. More generally, write $M(\lambda)$ for the Verma associated to the simple module $k_\lambda$, on which $K$ acts by the scalar $\lambda\in k\setminus 0$.

\begin{lemma}\label{vermasactionsl2}
The Verma module $M(\lambda)$ has a $k$ basis given by $F^mv_n=F^m \triangleright v_n$ for $k\in \mN$. The action is given by
\begin{align}
K\triangleright F^m v_n&=\lambda q^{-2m}v_n,&E\triangleright F^m v_n&=\frac{\lambda^{-1}[m]_q }{q^{-1}-q}F^{m-1}v_n.
\end{align}
\end{lemma}

\begin{corollary}
The Verma modules $M(\lambda)$ are simple for all $\lambda\in k\setminus 0$.
\end{corollary}
\begin{proof}
Lemma \ref{vermasactionsl2} shows that $M(\lambda)$ decomposes as the sum of simple $k[K^{\pm 1}]$-modules as
\[
V=\bigoplus_{m\geq 0}k_{\lambda q^{-2m}}.
\]
Let $W\leq V$ be any submodule. Let $l$ be minimal such that there exists $w\in W$ with $w\in k_{\lambda q^{-m}}$. Such an element exists, as we can take any inhomogeneous element and produce a homogeneous element $w'\in W$ by using the biggest $k$ such that $E^k w'\neq 0$ which then must be homogeneous. This follows as the scalar $\frac{\lambda^{-1}[m]_q }{q^{-1}-q}\neq 0$ for all $m, \lambda$. Now observe that $w=\mu F^l v_n$ and hence $v_n\in W$ as it is proportional to $E^lF^lv_n$.
\end{proof}

\begin{corollary}
The algebra $D_q(\fr{sl}_2)$ has no finite-dimensional representations.
\end{corollary}
\begin{proof}
Let $V$ be a finite-dimensional simple representation. In particular, $V=\oplus_{i=1}^nk_{\lambda_i}$ as a $k[K^{\pm 1}]$-module. Let $v$ be any weight vector (i.e. a vector in $k_{\lambda_i}$ for some $i$), then there exists a sufficiently large $j$ such that $E^jv \in k \langle v,E^1v,\ldots, E^{j-1}v \rangle$. But these are all vectors of different weight, so $E^jv=0$. We may assume $w:=E^{j-1}v\neq 0$ and hence $w$ is a highest weight vector (i.e. one annihilated by $E$) which generates $V$. Hence there exists a surjective morphism of $D_q(\fr{sl}_2)$-modules $\pi\colon M(\lambda)\twoheadrightarrow V$. But then $\ker \pi=0$ as $M(\lambda)$ has no non-trivial submodules.
\end{proof}

We now consider the action of $\lmod{\Uq}^{\op{fd}}$ on $\cO_q^{\Heis}(\fr{sl}_2)$ from Corollary \ref{fdcatoaction}. This gives an analogue of the Quantum Clebsch Gordan rule.

\begin{proposition}\label{clebschgordan}
The categorical action of the irreducible $\Uq$-module $L(n)$ of weight $q^n$ on the Verma $M(m)$ of the category $\cO_q^{\Heis}(\fr{sl}_2)$ decomposes as a sum of simples as
\begin{equation}\label{clebschgordanformula}
L(n)\triangleright M(m)=\bigoplus_{k=0}^n M(m+n-2k).
\end{equation}
\end{proposition}
\begin{proof}
Denote the highest weight vector of $L(n)$ by $v$ and the one of $M(m)$ by $w$. Using the formulas for the coproduct on $\Uq$, which serves as the coaction $\delta$ giving the categorical action, one checks that $v\otimes w$ is a highest weight vector of weight $n+m$. Note that the weight space $(L(n)\triangleright M(m))_{n+m-2k}$ has dimension $k+1$, for $k=0,\ldots, n$. Further note that each highest weight vector in $L(n)\triangleright M(m)$ gives the corresponding Verma module as a direct summand. Hence, we can see inductively that weight space of weight $n+m-2k$ contains a highest weight vector for any $k=0,\ldots,n$.
\end{proof}

We observe that the category $\cO^{\Heis}_q(\fr{sl}_2)$ has the same integral weight lattice parametrizing Verma modules as for $\cO_q(\fr{sl})$ and that the action of $\lmod{\Uq}^{\op{fd}}$ resembles the one on $\cO_q(\fr{sl}_2)$ with the crucial difference being that the Heisenberg version of category $\cO$ has no finite-dimensional simples and hence all Vermas are simple. We expect the results of this section to generalize to the quantum groups of other semisimple Lie algebras $\g$.


\section{Twisted Drinfeld Doubles}\label{twistedchapter}

Section \ref{doubles} does not use the full generality of the definition of YD-modules in \ref{quasiyddef} as we are restricting to the case that $\phi=1\otimes 1\otimes 1$ is the trivial 3-cycle. The preparations done in Section \ref{ydsection} however allow us, more generally, to define the (braided) Drinfeld and Heisenberg double of a quasi-Hopf algebra. This provides a Heisenberg analogue to the Drinfeld double of a quasi-Hopf algebra introduced in \cite{Maj3}.

In this Section, we will first give a description in terms of generators and relations of the Drinfeld double of a quasi-Hopf algebra in $\Vect$. Next, we provide definitions of \emph{twisted} Drinfeld and Heisenberg doubles. There are two different points of view on twisting used here. One is the \emph{Drinfeld twist} of an ordinary Hopf algebra (see \ref{drinfeldtwist}), the other one considers a commutative Hopf algebra as a quasi-Hopf algebra with respect to any 3-cycle (see \ref{twisteddoublesdefn}). We close this section by discussing the example of the twisted Heisenberg double of a group in \ref{twistedgroups}, including an adaptation of the groupoid interpretation of \cite{Wil} for the twisted Heisenberg double.

\subsection{The Drinfeld and Heisenberg double of a Quasi-Hopf Algebra}\label{quasidoubles}

Even though every monoidal category is equivalent to a strict one (cf. \ref{strictification}) it is still important to keep track of the associativity isomorphisms $\alpha\colon \otimes(\otimes \times \ide)\rightarrow \otimes (\ide \times \otimes)$ in certain situations. For instance, the class of quasi-Hopf algebras is closed under \emph{Drinfeld twist} $\Delta_F: m^3_{H\otimes H}(F\otimes \Delta \otimes F^{-1})$ of the coproduct for an arbitrary invertible element $F \in H\otimes H$ (see e.g. \cite[2.4]{Maj1}). If $F$ is a 2-cocycle and $H$ a Hopf algebra, then the twist $H_F$ is again a Hopf algebra (see e.g. \cite{GM}). 

In this section, we provide formulas for the Drinfeld and Heisenberg double of a quasi-Hopf algebra $H$, with 3-cycle $\phi$. For this, we leave the generality of working with quasi-Hopf algebras in braided monoidal categories $\cB$. However, via reconstruction theory, one can also obtain formulas in this most general case, if $\cB=\lmod{H}$ for a quasitriangular finite-dimensional Hopf algebra $H$ (or even a weakly quasitriangular pair $A$, $H$ and $\cB= \lcomod{A}$. For this, we need the following definition:

\begin{definition}
A \emph{dual paired object} for a quasi-bialgebra $B\in \cB$ is an object $C$ of $\cB$ together with a morphism $\ev\colon C\otimes B\to I$ such that there exist maps $\Delta_C\colon C\to C\otimes C$, $m_C\colon C\otimes C\to C$, $1_C\colon I \to C$, $\varepsilon_C\colon C\to I$, such that
\begin{align}
\begin{split}
\ev(m_C\otimes B)&= \ev(\ide\otimes \ev\otimes \ide)(\ide_{C\otimes C}\otimes \Delta_B),\\
\ev(C\otimes m_B)&=\ev(\ide\otimes \ev\otimes \ide)(\Delta_C\otimes \ide_{B\otimes B}),\\
\ev(1_C\otimes \ide_B)&=\varepsilon_B,\qquad\ev(\ide_C\otimes 1_B)=\varepsilon_C.
\end{split}
\end{align}
Using graphical calculus, these are the same conditions as in Figure \ref{dualityaxiomshopf}. Note however that $C$ is \emph{not} a quasi-bialgebra as the product is only associative up to a 3-\emph{cocycle} (obtained by duality from $\phi$). We refer an object with the structure of $C$ as a \emph{dual quasi-bialgebra}. If moreover $B$ has the structure of a quasi-Hopf algebra, and there exists a morphism $S_C\colon C\to C$ such that
$\ev(S_C\otimes \ide_B)=\ev(\ide_C\otimes S_B)$.
In this case, we say that $C$ is a \emph{dual quasi-Hopf algebra}.
\end{definition}

\begin{definition}\label{defquasidoubles}
Let $B$ be a quasi-bialgebra in $\lmod{H}$, where $H$ is a finite-dimensional Hopf algebra, with a dually paired object $C$. The \emph{braided Drinfeld double} of $B$, $C$ over $H$, denoted by $\Drin_H(C,B)$, is defined as the quasi-Hopf algebra obtained by reconstruction of the category $\leftsub{B^{\oop}}{\cal{YD}}^B(\cB)$ on $C\otimes H \otimes B$. The \emph{braided Heisenberg double} $\Heis_H(C,B)$ is defined by reconstruction of $\leftsub{B^{\oop}}{\cal{YD}}^B(\cB)$ on $C\otimes H\otimes B$.

Similarly, we can define $\Drin_A(C,B)$ and $\Heis_{A}(C,B)$ if we are given a weak quasitriangular structure on dually paired Hopf algebras $A$ and $H$.
\end{definition}

We only give explicit presentations of $\Drin_H(C,B)$ and $\Heis_H(C,B)$ in the case where $\cB=\Vect$ to simplify the exposition. Formulas for the general case can be obtained in a similar way, but involve several occurrences of the R-matrix.

\begin{proposition}\label{presentationsquasidoubles}
Let $B$ be a quasi-bialgebra in $\Vect$ with 3-cycle $\phi$, and dually paired object $C$. Then $\Drin_H(C,B)$ is generated as an algebra by elements of $C$, $H$ and $B^{\op{op}}$ subject to the relations (\ref{boso1})-(\ref{drincross}) and for $c,d\in C$ the relation
\begin{align}\label{quasicproduct}
\begin{split}
cd=&\phi_2^{(3)}c_{(2)}\phi^{-(2)}d_{(3)}\phi_1^{(1)}\ev(c_{(1)}\otimes \phi_2^{(2)})\ev(c_{(3)}\otimes \phi^{-(3)})\ev(c_{(4)}\otimes \phi_1^{(3)})\ev(d_{(1)}\otimes \phi_2^{(1)})\\&\ev(d_{(2)}\otimes \phi^{-(1)})\ev(d_{(4)}\otimes \phi_1^{(2)}).\end{split}
\end{align}
The algebra $\Drin_H(C,B)$ is a quasi-bialgebra with coproducts given by
\begin{align}
\Delta(h)&=h_{(1)}\otimes h_{(2)},\qquad \Delta(b)=b_{(1)}\otimes b_{(2)},\\
\begin{split}\Delta(c)&=\phi_2^{-(2)}c_{(2)}\phi^{(1)}\phi_1^{-(1)}\otimes \phi_2^{-(3)}\phi^{(3)}c_{(4)}\phi_1^{-(2)}\ev(c_{(1)}\otimes \phi_2^{-(1)})\ev(c_{(3)}\otimes \phi^{(2)})\\&\quad\ev(c_{(5)}\otimes \phi_1^{-(3)}).\end{split}
\end{align}
With these formulas, $\Delta$ gives a quasi-coaction with respect to the 3-cycle $\phi$ of $B$.
The counit is given by $\varepsilon(chb)=\varepsilon(c)\varepsilon(h)\varepsilon(c)$. If $B$ is a quasi-Hopf algebra, a formula for the antipode can be obtained by combining $\delta'$ from Figure \ref{rightquasicoaction} and in Figure \ref{yddualseqn}.

The braided Heisenberg double $\Heis_H(C,B)$ of the quasi-bialgebra $B$ is the algebra generated by $C$, $H$ and $B^{\op{op}}$ subject to the relations (\ref{boso1}), (\ref{boso2}) and the cross relation (\ref{heiscross}), as well as the relation (\ref{quasicproduct}) from above.
\end{proposition}

The Drinfeld double of a quasi-Hopf algebra was already introduced -- using a left module version -- in \cite{Maj3}.

\subsection{Twisted Hopf Algebras}\label{drinfeldtwist}

In this section, we study the Drinfeld and Heisenberg double for two notions of twist of an ordinary Hopf algebra. The first one is often referred to as a \emph{Drinfeld twist}, see e.g. \cite[4.2]{Maj1} for details. The second one generalizes, in the sense of \cite{Maj3}, the construction of the twisted Drinfeld double of a group algebra $\Drin^\omega(G)$.

\begin{definition}
Let $B$ be an ordinary bialgebra and $F=F^{(1)}\otimes F^{(2)}\in B\otimes B$ an invertible element. Then the \emph{Drinfeld twist} $B_F$ of $B$ is the quasi-bialgebra defined on the algebra $B$ with the quasi-coalgebra structure
\begin{align}
\Delta_F(b)&=F\Delta(b)F^{-1},&
\phi&=F_{23}((\ide \otimes \Delta)F)((\Delta\otimes \ide)F^{-1})F^{-1}_{
12},&
\end{align}
where the counit is not changed. If moreover $B$ is a Hopf algebra, then the same antipode gives an antipode for $B_F$ with respect to $a=S(F^{-(1)})F^{-(2)}$, and $b=F^{(1)}S(F^{(2)})$.
If $B$ is quasitriangular, then so is $B_F$ with universal R-matrix $F_{21}RF^{-1}$.
\end{definition}

For a proof that $B_F$ is indeed a quasi-bialgebra (or Hopf algebra) see \cite[Theorem 2.4.2]{Maj1}. In fact, it is shown more generally that any Drinfeld twist of a quasi-Hopf algebra is again a quasi-Hopf algebra. If $F$ is a 2-cycle (satisfying a dual condition to (\ref{2cocycle})), then $B_F$ is again a Hopf algebra. The Drinfeld twist is also referred to as a \emph{gauge transformation} in the literature. It is a basic result that the categories $\lmod{B}$ and $\lmod{B_F}$ are equivalent as monoidal categories.

Another point of view on twisting is to view a bialgebra $B$ as a quasi-bialgebra $B^\phi$ with respect to some 3-cycle $\phi$ which commutes with the two-fold coproduct $(\Delta\otimes \ide)\Delta$. For example, if $B$ is commutative, $B^\phi$ is a quasi-bialgebra for any 3-cycle. This way, the usual notion of a \emph{twisted} Drinfeld double of a group algebra can be obtained (see \cite{DPR,Maj3}). We will consider this example together with the corresponding Heisenberg double in the following section. Using either versions of twist, we can apply Proposition \ref{presentationsquasidoubles} to compute the Drinfeld and Heisenberg doubles of the corresponding twisted Hopf algebras. For a general definition, let $C, B$ be dually paired Hopf algebra.

\begin{definition}\label{twisteddoublesdefn}
The \emph{twisted Drinfeld double} $\Drin^\omega(C,B)$, with respect to a  3-cycle on $B$ which commutes with the two-fold coproducts, is defined as the bialgebra (respectively, Hopf algebra) $\Drin_k(C,B^\omega)$ using the notation of Definition \ref{defquasidoubles}.

The \emph{twisted Heisenberg double} $\Heis^\omega(C,B)$ is defined to be $\Heis_k(C,B^\omega)$.
\end{definition}

Hence, we can give a presentation for the twisted double using Proposition \ref{presentationsquasidoubles}. More generally, the same construction works if we start with dually paired braided Hopf algebras in $\cB=\lmod{H}$ (or $\lcomod{A}$ in the weakly quasitriangular case) and $\omega$ a 3-cycle on $B$ which commutes with all two-fold coproducts in $B\otimes B\otimes B\in \Alg(\cB)$ giving a definition of $\Drin^\omega_H(C,B):= \Drin_H(C,B^\omega)$ and its Heisenberg analogue. We will not use this general case in the present section.

As a direct corollary, we have that the monoidal category $\lmod{\Drin^\omega_H(C,B)}$ acts on the category $\lmod{\Heis^\omega_H(C,B)}$.

\subsection{The Twisted Heisenberg Double of a Group}\label{twistedgroups}

We finish this exposition by showing, as an example, how the categorical action of the monoidal category of the Drinfeld center on the Hopf center can be used to obtain an action of the category of $G^{\op{ad}}$-equivariant $\omega$-twisted vector bundles on the category of $\omega$-twisted $G^{\op{reg}}$-equivariant vector bundles on $G$.

Let $G$ be a finite group. Then 3-cycles in the sense of (\ref{3cocycle}) on the Hopf algebra $k[G]$ of functions on $G$ correspond to 3-cocycles in the group cohomology of $G$. That is, for such $\omega$,
\begin{align}
\omega(h,k,l)\omega(g,hk,l)\omega(g,h,k)&=\omega(g,h,kl)\omega(gh,k,l),&\forall g,h,k,l\in G,
\end{align}
where $\omega$ is normalized such that $\omega(g,h,k)=1$ as soon as one of the group elements equals $1$.

As $k[G]$ is commutative, $k^\omega[G]$, as introduced in the previous section, is a quasi-Hopf algebra for any 3-cocycle $\omega$. In \cite{Maj3}, the Drinfeld center of $\lcomod{k^\omega[G]}$ is shown to be equivalent to the category of modules over the quasi-Hopf algebra  $\Drin^{\omega}(G)$, which we refer to as the \emph{twisted} Drinfeld double of the group $G$. There are two different points of view on its representation category. The first one is:

\begin{proposition}[\cite{Maj3}]
The category $\lmod{\Drin^{\omega}(G)}$ is equivalent to the category of $G$-graded vector spaces which are $G^{\op{ad}}$-equivariant with respect to twisted representations of $G$. The representations are twisted by the cocycle $\tau(\omega)\in Z^2(G,\leftexp{\op{ad}}{k[G]})$ valued in $k[G]$ with the left adjoint action of $G$, which is defined by
\begin{align}
\tau(\omega)(g,h)(k)&=\frac{\omega(g,h,h^{-1}g^{-1}kgh)\omega(k,g,h)}{\omega(g,g^{-1}kg,h)}.
\end{align}
That is, the action a homogeneous element $v$ of degree $d$ is given by
\begin{align}
g\triangleright(h\triangleright v)&=\tau(\omega)(g,h)(d)gh\triangleright v.
\end{align}
The equivariance condition corresponds to the YD-condition that the degree of $g\triangleright v$ is $gdg^{-1}$. We refer to vector spaces $V$ with this structure as $\omega$-twisted $G^{\op{ad}}$-equivariant vector bundles over $G$.
\end{proposition}

Modules over the twisted Heisenberg double are twisted representations of $G$ (with respect to the same 2-cocycle $\tau(\omega)$. The only difference is that the equivariance condition is that the degree of $g\triangleright v$ is $gd$ if $v$ has degree $d$. Hence we speak of $G^{\op{reg}}$-equivariant vector bundles over $G$. The categorical action of $G^{\op{ad}}$-equivariant $\omega$-twisted vector bundles on $G^{\op{reg}}$-equivariant ones is now a direct consequence of the categorical action discussed before.

Another point of view on the category $\lmod{\Drin^\omega(G)}$ is given in \cite{Wil} where representations of the twisted Drinfeld double $\Drin^\omega(G)$ are identified with $\tau(\omega)$-twisted representations of the action groupoid $\un{G}^{\op{ad}}$ of $G$ acting on itself by conjugation. The 2-cocycle $\tau(\omega)$ is obtained by the transgression map
\[
\tau \colon Z^3(G,U(1))\longrightarrow Z^2(\un{G}^{\op{ad}}).
\]
However, the same map also produces 2-cocycles for the action groupoid $\un{G}^{\op{reg}}$ of $G$ acting on itself by the regular action. Hence, we can identify the category $\lmod{\Heis^\omega(G)}$ with the category of $\tau(\omega)$-twisted representations over $\un{G}^{\op{reg}}$.





\bibliographystyle{elsart-num-sort}
\bibliography{biblio2}





\end{document}

%% file: Graphics/dualhopf.pdf_tex
\begingroup%
  \makeatletter%
  \providecommand\color[2][]{%
    \errmessage{(Inkscape) Color is used for the text in Inkscape, but the package 'color.sty' is not loaded}%
    \renewcommand\color[2][]{}%
  }%
  \providecommand\transparent[1]{%
    \errmessage{(Inkscape) Transparency is used (non-zero) for the text in Inkscape, but the package 'transparent.sty' is not loaded}%
    \renewcommand\transparent[1]{}%
  }%
  \providecommand\rotatebox[2]{#2}%
  \ifx\svgwidth\undefined%
    \setlength{\unitlength}{384.15155029bp}%
    \ifx\svgscale\undefined%
      \relax%
    \else%
      \setlength{\unitlength}{\unitlength * \real{\svgscale}}%
    \fi%
  \else%
    \setlength{\unitlength}{\svgwidth}%
  \fi%
  \global\let\svgwidth\undefined%
  \global\let\svgscale\undefined%
  \makeatother%
  \begin{picture}(1,0.10514387)%
    \put(0,0){\includegraphics[width=\unitlength]{dualhopf.pdf}}%
    \put(0.11446459,0.04405695){\color[rgb]{0,0,0}\makebox(0,0)[lb]{\smash{$=$}}}%
    \put(0.47466821,0.04346194){\color[rgb]{0,0,0}\makebox(0,0)[lb]{\smash{$=$}}}%
    \put(0.30151877,0.03290063){\color[rgb]{0,0,0}\makebox(0,0)[lb]{\smash{,}}}%
    \put(0.62467996,0.03304922){\color[rgb]{0,0,0}\makebox(0,0)[lb]{\smash{,}}}%
    \put(0.7392181,0.0434618){\color[rgb]{0,0,0}\makebox(0,0)[lb]{\smash{$=$}}}%
    \put(0.81136217,0.03267739){\color[rgb]{0,0,0}\makebox(0,0)[lb]{\smash{,}}}%
    \put(0.8647672,0.0434618){\color[rgb]{0,0,0}\makebox(0,0)[lb]{\smash{$=$}}}%
    \put(-0.00044538,0.095376){\color[rgb]{0,0,0}\makebox(0,0)[lb]{\smash{$C$}}}%
    \put(0.06202996,0.095376){\color[rgb]{0,0,0}\makebox(0,0)[lb]{\smash{$C$}}}%
    \put(0.09326763,0.095376){\color[rgb]{0,0,0}\makebox(0,0)[lb]{\smash{$B$}}}%
    \put(0.15574298,0.095376){\color[rgb]{0,0,0}\makebox(0,0)[lb]{\smash{$C$}}}%
    \put(0.19739321,0.095376){\color[rgb]{0,0,0}\makebox(0,0)[lb]{\smash{$C$}}}%
    \put(0.35358156,0.095376){\color[rgb]{0,0,0}\makebox(0,0)[lb]{\smash{$C$}}}%
    \put(0.24960477,0.095376){\color[rgb]{0,0,0}\makebox(0,0)[lb]{\smash{$B$}}}%
    \put(0.40564435,0.095376){\color[rgb]{0,0,0}\makebox(0,0)[lb]{\smash{$B$}}}%
    \put(0.44729457,0.095376){\color[rgb]{0,0,0}\makebox(0,0)[lb]{\smash{$B$}}}%
    \put(0.50976985,0.095376){\color[rgb]{0,0,0}\makebox(0,0)[lb]{\smash{$C$}}}%
    \put(0.54100759,0.095376){\color[rgb]{0,0,0}\makebox(0,0)[lb]{\smash{$B$}}}%
    \put(0.60348287,0.095376){\color[rgb]{0,0,0}\makebox(0,0)[lb]{\smash{$B$}}}%
    \put(0.64513317,0.095376){\color[rgb]{0,0,0}\makebox(0,0)[lb]{\smash{$C$}}}%
    \put(0.78049641,0.095376){\color[rgb]{0,0,0}\makebox(0,0)[lb]{\smash{$C$}}}%
    \put(0.84297175,0.095376){\color[rgb]{0,0,0}\makebox(0,0)[lb]{\smash{$B$}}}%
    \put(0.97833497,0.095376){\color[rgb]{0,0,0}\makebox(0,0)[lb]{\smash{$B$}}}%
  \end{picture}%
\endgroup%

%% file: Graphics/braiding.pdf_tex

\begingroup
  \makeatletter
  \providecommand\color[2][]{%
    \errmessage{(Inkscape) Color is used for the text in Inkscape, but the package 'color.sty' is not loaded}
    \renewcommand\color[2][]{}%
  }
  \providecommand\transparent[1]{%
    \errmessage{(Inkscape) Transparency is used (non-zero) for the text in Inkscape, but the package 'transparent.sty' is not loaded}
    \renewcommand\transparent[1]{}%
  }
  \providecommand\rotatebox[2]{#2}
  \ifx\svgwidth\undefined
    \setlength{\unitlength}{24.80101968pt}
  \else
    \setlength{\unitlength}{\svgwidth}
  \fi
  \global\let\svgwidth\undefined
  \makeatother
  \begin{picture}(1,1.00305813)%
    \put(0,0){\includegraphics[width=\unitlength]{braiding.pdf}}%
  \end{picture}%
\endgroup

%% file: Graphics/inversebraiding.pdf_tex

\begingroup
  \makeatletter
  \providecommand\color[2][]{%
    \errmessage{(Inkscape) Color is used for the text in Inkscape, but the package 'color.sty' is not loaded}
    \renewcommand\color[2][]{}%
  }
  \providecommand\transparent[1]{%
    \errmessage{(Inkscape) Transparency is used (non-zero) for the text in Inkscape, but the package 'transparent.sty' is not loaded}
    \renewcommand\transparent[1]{}%
  }
  \providecommand\rotatebox[2]{#2}
  \ifx\svgwidth\undefined
    \setlength{\unitlength}{24.80101968pt}
  \else
    \setlength{\unitlength}{\svgwidth}
  \fi
  \global\let\svgwidth\undefined
  \makeatother
  \begin{picture}(1,1.00305813)%
    \put(0,0){\includegraphics[width=\unitlength]{inversebraiding.pdf}}%
  \end{picture}%
\endgroup

%% file: Graphics/braidedpairing.pdf_tex
\begingroup%
  \makeatletter%
  \providecommand\color[2][]{%
    \errmessage{(Inkscape) Color is used for the text in Inkscape, but the package 'color.sty' is not loaded}%
    \renewcommand\color[2][]{}%
  }%
  \providecommand\transparent[1]{%
    \errmessage{(Inkscape) Transparency is used (non-zero) for the text in Inkscape, but the package 'transparent.sty' is not loaded}%
    \renewcommand\transparent[1]{}%
  }%
  \providecommand\rotatebox[2]{#2}%
  \ifx\svgwidth\undefined%
    \setlength{\unitlength}{252.1789917bp}%
    \ifx\svgscale\undefined%
      \relax%
    \else%
      \setlength{\unitlength}{\unitlength * \real{\svgscale}}%
    \fi%
  \else%
    \setlength{\unitlength}{\svgwidth}%
  \fi%
  \global\let\svgwidth\undefined%
  \global\let\svgscale\undefined%
  \makeatother%
  \begin{picture}(1,0.41987912)%
    \put(0,0){\includegraphics[width=\unitlength]{braidedpairing.pdf}}%
    \put(0.26680128,0.21352709){\color[rgb]{0,0,0}\makebox(0,0)[lb]{\smash{$=$}}}%
    \put(0.41180748,0.00667619){\color[rgb]{0,0,0}\makebox(0,0)[lb]{\smash{(in $\overline{\cB}$)}}}%
    \put(0.63452422,0.21432821){\color[rgb]{0,0,0}\makebox(0,0)[lb]{\smash{$=$}}}%
    \put(0.79248947,0.00667619){\color[rgb]{0,0,0}\makebox(0,0)[lb]{\smash{(in $\cB$)}}}%
    \put(-0.00067846,0.3104649){\color[rgb]{0,0,0}\makebox(0,0)[lb]{\smash{$\leftexp{cop}{C}$}}}%
    \put(0.1046495,0.3104649){\color[rgb]{0,0,0}\makebox(0,0)[lb]{\smash{$\leftexp{cop}{B}$}}}%
    \put(0.19671152,0.3104649){\color[rgb]{0,0,0}\makebox(0,0)[lb]{\smash{$\leftexp{cop}{B}$}}}%
    \put(0.33245035,0.35881913){\color[rgb]{0,0,0}\makebox(0,0)[lb]{\smash{$\leftexp{cop}{C}$}}}%
    \put(0.47203376,0.358018){\color[rgb]{0,0,0}\makebox(0,0)[lb]{\smash{$\leftexp{cop}{B}$}}}%
    \put(0.56720426,0.35798586){\color[rgb]{0,0,0}\makebox(0,0)[lb]{\smash{$\leftexp{cop}{B}$}}}%
    \put(0.72825752,0.40499943){\color[rgb]{0,0,0}\makebox(0,0)[lb]{\smash{$C$}}}%
    \put(0.87246261,0.40499934){\color[rgb]{0,0,0}\makebox(0,0)[lb]{\smash{$B$}}}%
    \put(0.96699703,0.40419821){\color[rgb]{0,0,0}\makebox(0,0)[lb]{\smash{$B$}}}%
  \end{picture}%
\endgroup%

%% file: Graphics/graphantipode.pdf_tex
\begingroup%
  \makeatletter%
  \providecommand\color[2][]{%
    \errmessage{(Inkscape) Color is used for the text in Inkscape, but the package 'color.sty' is not loaded}%
    \renewcommand\color[2][]{}%
  }%
  \providecommand\transparent[1]{%
    \errmessage{(Inkscape) Transparency is used (non-zero) for the text in Inkscape, but the package 'transparent.sty' is not loaded}%
    \renewcommand\transparent[1]{}%
  }%
  \providecommand\rotatebox[2]{#2}%
  \ifx\svgwidth\undefined%
    \setlength{\unitlength}{109.65624854bp}%
    \ifx\svgscale\undefined%
      \relax%
    \else%
      \setlength{\unitlength}{\unitlength * \real{\svgscale}}%
    \fi%
  \else%
    \setlength{\unitlength}{\svgwidth}%
  \fi%
  \global\let\svgwidth\undefined%
  \global\let\svgscale\undefined%
  \makeatother%
  \begin{picture}(1,0.44551277)%
    \put(0,0){\includegraphics[width=\unitlength]{graphantipode.pdf}}%
    \put(0.29547773,0.25931266){\color[rgb]{0,0,0}\makebox(0,0)[lb]{\smash{$=$}}}%
    \put(0.05374831,0.21570783){\color[rgb]{0,0,0}\makebox(0,0)[lb]{\smash{$S$}}}%
    \put(0.67478008,0.28736018){\color[rgb]{0,0,0}\makebox(0,0)[lb]{\smash{$d_X$}}}%
    \put(0.67471501,0.08630308){\color[rgb]{0,0,0}\makebox(0,0)[lb]{\smash{$d_X^{\text{-1}}$}}}%
  \end{picture}%
\endgroup%

%% file: Graphics/graphinvantipode.pdf_tex
\begingroup%
  \makeatletter%
  \providecommand\color[2][]{%
    \errmessage{(Inkscape) Color is used for the text in Inkscape, but the package 'color.sty' is not loaded}%
    \renewcommand\color[2][]{}%
  }%
  \providecommand\transparent[1]{%
    \errmessage{(Inkscape) Transparency is used (non-zero) for the text in Inkscape, but the package 'transparent.sty' is not loaded}%
    \renewcommand\transparent[1]{}%
  }%
  \providecommand\rotatebox[2]{#2}%
  \ifx\svgwidth\undefined%
    \setlength{\unitlength}{112.10888944bp}%
    \ifx\svgscale\undefined%
      \relax%
    \else%
      \setlength{\unitlength}{\unitlength * \real{\svgscale}}%
    \fi%
  \else%
    \setlength{\unitlength}{\svgwidth}%
  \fi%
  \global\let\svgwidth\undefined%
  \global\let\svgscale\undefined%
  \makeatother%
  \begin{picture}(1,0.43576615)%
    \put(0,0){\includegraphics[width=\unitlength]{graphinvantipode.pdf}}%
    \put(0.31089074,0.2536396){\color[rgb]{0,0,0}\makebox(0,0)[lb]{\smash{$=$}}}%
    \put(0.68183137,0.10238328){\color[rgb]{0,0,0}\makebox(0,0)[lb]{\smash{$d_X^{\text{-1}}$}}}%
    \put(0.68253204,0.27674162){\color[rgb]{0,0,0}\makebox(0,0)[lb]{\smash{$d_X$}}}%
    \put(0.02468965,0.24000409){\color[rgb]{0,0,0}\makebox(0,0)[lb]{\smash{$S^{\text{-}1}$}}}%
  \end{picture}%
\endgroup%

%% file: Graphics/graphevl.pdf_tex
\begingroup%
  \makeatletter%
  \providecommand\color[2][]{%
    \errmessage{(Inkscape) Color is used for the text in Inkscape, but the package 'color.sty' is not loaded}%
    \renewcommand\color[2][]{}%
  }%
  \providecommand\transparent[1]{%
    \errmessage{(Inkscape) Transparency is used (non-zero) for the text in Inkscape, but the package 'transparent.sty' is not loaded}%
    \renewcommand\transparent[1]{}%
  }%
  \providecommand\rotatebox[2]{#2}%
  \ifx\svgwidth\undefined%
    \setlength{\unitlength}{32.80103651bp}%
    \ifx\svgscale\undefined%
      \relax%
    \else%
      \setlength{\unitlength}{\unitlength * \real{\svgscale}}%
    \fi%
  \else%
    \setlength{\unitlength}{\svgwidth}%
  \fi%
  \global\let\svgwidth\undefined%
  \global\let\svgscale\undefined%
  \makeatother%
  \begin{picture}(1,1.01902268)%
    \put(0,0){\includegraphics[width=\unitlength]{graphevl.pdf}}%
    \put(0.42995723,0.62299444){\color[rgb]{0,0,0}\makebox(0,0)[lb]{\smash{$a$}}}%
  \end{picture}%
\endgroup%

%% file: Graphics/graphcoevl.pdf_tex
\begingroup%
  \makeatletter%
  \providecommand\color[2][]{%
    \errmessage{(Inkscape) Color is used for the text in Inkscape, but the package 'color.sty' is not loaded}%
    \renewcommand\color[2][]{}%
  }%
  \providecommand\transparent[1]{%
    \errmessage{(Inkscape) Transparency is used (non-zero) for the text in Inkscape, but the package 'transparent.sty' is not loaded}%
    \renewcommand\transparent[1]{}%
  }%
  \providecommand\rotatebox[2]{#2}%
  \ifx\svgwidth\undefined%
    \setlength{\unitlength}{51.07504bp}%
    \ifx\svgscale\undefined%
      \relax%
    \else%
      \setlength{\unitlength}{\unitlength * \real{\svgscale}}%
    \fi%
  \else%
    \setlength{\unitlength}{\svgwidth}%
  \fi%
  \global\let\svgwidth\undefined%
  \global\let\svgscale\undefined%
  \makeatother%
  \begin{picture}(1,0.71267688)%
    \put(0,0){\includegraphics[width=\unitlength]{graphcoevl.pdf}}%
    \put(0.07805563,0.36908489){\color[rgb]{0,0,0}\makebox(0,0)[lb]{\smash{$b$}}}%
  \end{picture}%
\endgroup%

%% file: Graphics/graphevr.pdf_tex
\begingroup%
  \makeatletter%
  \providecommand\color[2][]{%
    \errmessage{(Inkscape) Color is used for the text in Inkscape, but the package 'color.sty' is not loaded}%
    \renewcommand\color[2][]{}%
  }%
  \providecommand\transparent[1]{%
    \errmessage{(Inkscape) Transparency is used (non-zero) for the text in Inkscape, but the package 'transparent.sty' is not loaded}%
    \renewcommand\transparent[1]{}%
  }%
  \providecommand\rotatebox[2]{#2}%
  \ifx\svgwidth\undefined%
    \setlength{\unitlength}{43.44012789bp}%
    \ifx\svgscale\undefined%
      \relax%
    \else%
      \setlength{\unitlength}{\unitlength * \real{\svgscale}}%
    \fi%
  \else%
    \setlength{\unitlength}{\svgwidth}%
  \fi%
  \global\let\svgwidth\undefined%
  \global\let\svgscale\undefined%
  \makeatother%
  \begin{picture}(1,1.51533184)%
    \put(0,0){\includegraphics[width=\unitlength]{graphevr.pdf}}%
    \put(0.30189475,0.42084572){\color[rgb]{0,0,0}\makebox(0,0)[lb]{\smash{$S^{\text{-}1}$}}}%
    \put(0.36148703,0.82032133){\color[rgb]{0,0,0}\makebox(0,0)[lb]{\smash{$a$}}}%
  \end{picture}%
\endgroup%

%% file: Graphics/graphcoevr.pdf_tex
\begingroup%
  \makeatletter%
  \providecommand\color[2][]{%
    \errmessage{(Inkscape) Color is used for the text in Inkscape, but the package 'color.sty' is not loaded}%
    \renewcommand\color[2][]{}%
  }%
  \providecommand\transparent[1]{%
    \errmessage{(Inkscape) Transparency is used (non-zero) for the text in Inkscape, but the package 'transparent.sty' is not loaded}%
    \renewcommand\transparent[1]{}%
  }%
  \providecommand\rotatebox[2]{#2}%
  \ifx\svgwidth\undefined%
    \setlength{\unitlength}{53.6617776bp}%
    \ifx\svgscale\undefined%
      \relax%
    \else%
      \setlength{\unitlength}{\unitlength * \real{\svgscale}}%
    \fi%
  \else%
    \setlength{\unitlength}{\svgwidth}%
  \fi%
  \global\let\svgwidth\undefined%
  \global\let\svgscale\undefined%
  \makeatother%
  \begin{picture}(1,1.1014866)%
    \put(0,0){\includegraphics[width=\unitlength]{graphcoevr.pdf}}%
    \put(0.0515809,0.59415134){\color[rgb]{0,0,0}\makebox(0,0)[lb]{\smash{$S^{\text{-}1}$}}}%
    \put(0.12234612,0.91049452){\color[rgb]{0,0,0}\makebox(0,0)[lb]{\smash{$b$}}}%
  \end{picture}%
\endgroup%

%% file: Graphics/opcop.pdf_tex
\begingroup%
  \makeatletter%
  \providecommand\color[2][]{%
    \errmessage{(Inkscape) Color is used for the text in Inkscape, but the package 'color.sty' is not loaded}%
    \renewcommand\color[2][]{}%
  }%
  \providecommand\transparent[1]{%
    \errmessage{(Inkscape) Transparency is used (non-zero) for the text in Inkscape, but the package 'transparent.sty' is not loaded}%
    \renewcommand\transparent[1]{}%
  }%
  \providecommand\rotatebox[2]{#2}%
  \ifx\svgwidth\undefined%
    \setlength{\unitlength}{136.65097968bp}%
    \ifx\svgscale\undefined%
      \relax%
    \else%
      \setlength{\unitlength}{\unitlength * \real{\svgscale}}%
    \fi%
  \else%
    \setlength{\unitlength}{\svgwidth}%
  \fi%
  \global\let\svgwidth\undefined%
  \global\let\svgscale\undefined%
  \makeatother%
  \begin{picture}(1,0.32657072)%
    \put(0,0){\includegraphics[width=\unitlength]{opcop.pdf}}%
    \put(0.03220599,0.19473153){\color[rgb]{0,0,0}\makebox(0,0)[lb]{\smash{$\text{cop}$}}}%
    \put(0.4650082,0.14580611){\color[rgb]{0,0,0}\makebox(0,0)[lb]{\smash{$=$}}}%
  \end{picture}%
\endgroup%

%% file: Graphics/graphrmodevl.pdf_tex
\begingroup%
  \makeatletter%
  \providecommand\color[2][]{%
    \errmessage{(Inkscape) Color is used for the text in Inkscape, but the package 'color.sty' is not loaded}%
    \renewcommand\color[2][]{}%
  }%
  \providecommand\transparent[1]{%
    \errmessage{(Inkscape) Transparency is used (non-zero) for the text in Inkscape, but the package 'transparent.sty' is not loaded}%
    \renewcommand\transparent[1]{}%
  }%
  \providecommand\rotatebox[2]{#2}%
  \ifx\svgwidth\undefined%
    \setlength{\unitlength}{68.43404879bp}%
    \ifx\svgscale\undefined%
      \relax%
    \else%
      \setlength{\unitlength}{\unitlength * \real{\svgscale}}%
    \fi%
  \else%
    \setlength{\unitlength}{\svgwidth}%
  \fi%
  \global\let\svgwidth\undefined%
  \global\let\svgscale\undefined%
  \makeatother%
  \begin{picture}(1,0.73486577)%
    \put(0,0){\includegraphics[width=\unitlength]{graphrmodevl.pdf}}%
    \put(0.55686121,0.33125743){\color[rgb]{0,0,0}\makebox(0,0)[lb]{\smash{$S^{\text{-}1}$}}}%
    \put(0.60951727,0.58510014){\color[rgb]{0,0,0}\makebox(0,0)[lb]{\smash{$b$}}}%
  \end{picture}%
\endgroup%

%% file: Graphics/graphrmodcoevl.pdf_tex
\begingroup%
  \makeatletter%
  \providecommand\color[2][]{%
    \errmessage{(Inkscape) Color is used for the text in Inkscape, but the package 'color.sty' is not loaded}%
    \renewcommand\color[2][]{}%
  }%
  \providecommand\transparent[1]{%
    \errmessage{(Inkscape) Transparency is used (non-zero) for the text in Inkscape, but the package 'transparent.sty' is not loaded}%
    \renewcommand\transparent[1]{}%
  }%
  \providecommand\rotatebox[2]{#2}%
  \ifx\svgwidth\undefined%
    \setlength{\unitlength}{40.41200289bp}%
    \ifx\svgscale\undefined%
      \relax%
    \else%
      \setlength{\unitlength}{\unitlength * \real{\svgscale}}%
    \fi%
  \else%
    \setlength{\unitlength}{\svgwidth}%
  \fi%
  \global\let\svgwidth\undefined%
  \global\let\svgscale\undefined%
  \makeatother%
  \begin{picture}(1,1.13827568)%
    \put(0,0){\includegraphics[width=\unitlength]{graphrmodcoevl.pdf}}%
    \put(0.80403022,0.33769481){\color[rgb]{0,0,0}\makebox(0,0)[lb]{\smash{$~$}}}%
    \put(0.14635124,0.32355336){\color[rgb]{0,0,0}\makebox(0,0)[lb]{\smash{$S^{\text{-}1}$}}}%
    \put(0.23500354,0.75296221){\color[rgb]{0,0,0}\makebox(0,0)[lb]{\smash{$a$}}}%
  \end{picture}%
\endgroup%

%% file: Graphics/graphrmodevr.pdf_tex
\begingroup%
  \makeatletter%
  \providecommand\color[2][]{%
    \errmessage{(Inkscape) Color is used for the text in Inkscape, but the package 'color.sty' is not loaded}%
    \renewcommand\color[2][]{}%
  }%
  \providecommand\transparent[1]{%
    \errmessage{(Inkscape) Transparency is used (non-zero) for the text in Inkscape, but the package 'transparent.sty' is not loaded}%
    \renewcommand\transparent[1]{}%
  }%
  \providecommand\rotatebox[2]{#2}%
  \ifx\svgwidth\undefined%
    \setlength{\unitlength}{51.211288bp}%
    \ifx\svgscale\undefined%
      \relax%
    \else%
      \setlength{\unitlength}{\unitlength * \real{\svgscale}}%
    \fi%
  \else%
    \setlength{\unitlength}{\svgwidth}%
  \fi%
  \global\let\svgwidth\undefined%
  \global\let\svgscale\undefined%
  \makeatother%
  \begin{picture}(1,0.97295364)%
    \put(0,0){\includegraphics[width=\unitlength]{graphrmodevr.pdf}}%
    \put(0.82275306,0.33610563){\color[rgb]{0,0,0}\makebox(0,0)[lb]{\smash{$b$}}}%
  \end{picture}%
\endgroup%

%% file: Graphics/graphrmodcoevr.pdf_tex
\begingroup%
  \makeatletter%
  \providecommand\color[2][]{%
    \errmessage{(Inkscape) Color is used for the text in Inkscape, but the package 'color.sty' is not loaded}%
    \renewcommand\color[2][]{}%
  }%
  \providecommand\transparent[1]{%
    \errmessage{(Inkscape) Transparency is used (non-zero) for the text in Inkscape, but the package 'transparent.sty' is not loaded}%
    \renewcommand\transparent[1]{}%
  }%
  \providecommand\rotatebox[2]{#2}%
  \ifx\svgwidth\undefined%
    \setlength{\unitlength}{41.22013277bp}%
    \ifx\svgscale\undefined%
      \relax%
    \else%
      \setlength{\unitlength}{\unitlength * \real{\svgscale}}%
    \fi%
  \else%
    \setlength{\unitlength}{\svgwidth}%
  \fi%
  \global\let\svgwidth\undefined%
  \global\let\svgscale\undefined%
  \makeatother%
  \begin{picture}(1,1.18363355)%
    \put(0,0){\includegraphics[width=\unitlength]{graphrmodcoevr.pdf}}%
    \put(0.80787225,0.60713777){\color[rgb]{0,0,0}\makebox(0,0)[lb]{\smash{$~$}}}%
    \put(0.26940887,0.73804577){\color[rgb]{0,0,0}\makebox(0,0)[lb]{\smash{$a$}}}%
  \end{picture}%
\endgroup%

%% file: Graphics/quasiyd1.pdf_tex
\begingroup%
  \makeatletter%
  \providecommand\color[2][]{%
    \errmessage{(Inkscape) Color is used for the text in Inkscape, but the package 'color.sty' is not loaded}%
    \renewcommand\color[2][]{}%
  }%
  \providecommand\transparent[1]{%
    \errmessage{(Inkscape) Transparency is used (non-zero) for the text in Inkscape, but the package 'transparent.sty' is not loaded}%
    \renewcommand\transparent[1]{}%
  }%
  \providecommand\rotatebox[2]{#2}%
  \ifx\svgwidth\undefined%
    \setlength{\unitlength}{153.97823486bp}%
    \ifx\svgscale\undefined%
      \relax%
    \else%
      \setlength{\unitlength}{\unitlength * \real{\svgscale}}%
    \fi%
  \else%
    \setlength{\unitlength}{\svgwidth}%
  \fi%
  \global\let\svgwidth\undefined%
  \global\let\svgscale\undefined%
  \makeatother%
  \begin{picture}(1,0.72854583)%
    \put(0,0){\includegraphics[width=\unitlength]{quasiyd1.pdf}}%
    \put(0.60469552,0.59844218){\color[rgb]{0,0,0}\makebox(0,0)[lb]{\smash{$\phi$}}}%
    \put(0.74263441,0.43469833){\color[rgb]{0,0,0}\makebox(0,0)[lb]{\smash{$\phi^{\text{-}1}$}}}%
    \put(0.3374034,0.36566194){\color[rgb]{0,0,0}\makebox(0,0)[lb]{\smash{$=$}}}%
    \put(-0.00111116,0.00353007){\color[rgb]{0,0,0}\makebox(0,0)[lb]{\smash{$C$}}}%
    \put(0.15108919,0.00353007){\color[rgb]{0,0,0}\makebox(0,0)[lb]{\smash{$C$}}}%
    \put(0.23506179,0.00353007){\color[rgb]{0,0,0}\makebox(0,0)[lb]{\smash{$V$}}}%
    \put(0.20882036,0.70417651){\color[rgb]{0,0,0}\makebox(0,0)[lb]{\smash{$V$}}}%
    \put(0.46336229,0.70417651){\color[rgb]{0,0,0}\makebox(0,0)[lb]{\smash{$V$}}}%
    \put(0.5985058,0.00353007){\color[rgb]{0,0,0}\makebox(0,0)[lb]{\smash{$C$}}}%
    \put(0.74939394,0.0032117){\color[rgb]{0,0,0}\makebox(0,0)[lb]{\smash{$C$}}}%
    \put(0.93570817,0.00353007){\color[rgb]{0,0,0}\makebox(0,0)[lb]{\smash{$V$}}}%
    \put(0.44399011,0.42407529){\color[rgb]{0,0,0}\makebox(0,0)[lb]{\smash{$G$}}}%
    \put(0.58148638,0.29659922){\color[rgb]{0,0,0}\makebox(0,0)[lb]{\smash{$G$}}}%
    \put(0.66925384,0.22571708){\color[rgb]{0,0,0}\makebox(0,0)[lb]{\smash{$G$}}}%
    \put(0.81806895,0.13312508){\color[rgb]{0,0,0}\makebox(0,0)[lb]{\smash{$G$}}}%
    \put(0.65552275,0.11883877){\color[rgb]{0,0,0}\makebox(0,0)[lb]{\smash{$G$}}}%
    \put(0.91648069,0.28151427){\color[rgb]{0,0,0}\makebox(0,0)[lb]{\smash{$\phi$}}}%
  \end{picture}%
\endgroup%

%% file: Graphics/quasiyd2.pdf_tex
\begingroup%
  \makeatletter%
  \providecommand\color[2][]{%
    \errmessage{(Inkscape) Color is used for the text in Inkscape, but the package 'color.sty' is not loaded}%
    \renewcommand\color[2][]{}%
  }%
  \providecommand\transparent[1]{%
    \errmessage{(Inkscape) Transparency is used (non-zero) for the text in Inkscape, but the package 'transparent.sty' is not loaded}%
    \renewcommand\transparent[1]{}%
  }%
  \providecommand\rotatebox[2]{#2}%
  \ifx\svgwidth\undefined%
    \setlength{\unitlength}{184.68482361bp}%
    \ifx\svgscale\undefined%
      \relax%
    \else%
      \setlength{\unitlength}{\unitlength * \real{\svgscale}}%
    \fi%
  \else%
    \setlength{\unitlength}{\svgwidth}%
  \fi%
  \global\let\svgwidth\undefined%
  \global\let\svgscale\undefined%
  \makeatother%
  \begin{picture}(1,0.63071531)%
    \put(0,0){\includegraphics[width=\unitlength]{quasiyd2.pdf}}%
    \put(0.44734414,0.60992833){\color[rgb]{0,0,0}\makebox(0,0)[lb]{\smash{$W$}}}%
    \put(0.51578289,0.00407234){\color[rgb]{0,0,0}\makebox(0,0)[lb]{\smash{$C$}}}%
    \put(0.67715538,0.00306635){\color[rgb]{0,0,0}\makebox(0,0)[lb]{\smash{$V$}}}%
    \put(0.84115551,0.00407234){\color[rgb]{0,0,0}\makebox(0,0)[lb]{\smash{$W$}}}%
    \put(0.2775721,0.61039775){\color[rgb]{0,0,0}\makebox(0,0)[lb]{\smash{$V$}}}%
    \put(-0.00092641,0.0026777){\color[rgb]{0,0,0}\makebox(0,0)[lb]{\smash{$C$}}}%
    \put(0.07614998,0.00286529){\color[rgb]{0,0,0}\makebox(0,0)[lb]{\smash{$V\otimes W$}}}%
    \put(0.07280252,0.60975701){\color[rgb]{0,0,0}\makebox(0,0)[lb]{\smash{$V\otimes W$}}}%
    \put(0.22150335,0.30344362){\color[rgb]{0,0,0}\makebox(0,0)[lb]{\smash{$=$}}}%
    \put(0.43119286,0.35797737){\color[rgb]{0,0,0}\makebox(0,0)[lb]{\smash{$G$}}}%
    \put(0.5784708,0.25401648){\color[rgb]{0,0,0}\makebox(0,0)[lb]{\smash{$G$}}}%
    \put(0.5741391,0.09369954){\color[rgb]{0,0,0}\makebox(0,0)[lb]{\smash{$G$}}}%
    \put(0.52096007,0.51844382){\color[rgb]{0,0,0}\makebox(0,0)[lb]{\smash{$\phi^{\text{-}1}$}}}%
    \put(0.69517369,0.38753509){\color[rgb]{0,0,0}\makebox(0,0)[lb]{\smash{$\phi$}}}%
    \put(0.81122823,0.25854158){\color[rgb]{0,0,0}\makebox(0,0)[lb]{\smash{$\phi^{\text{-}1}$}}}%
  \end{picture}%
\endgroup%

%% file: Graphics/ydcond3.pdf_tex
\begingroup%
  \makeatletter%
  \providecommand\color[2][]{%
    \errmessage{(Inkscape) Color is used for the text in Inkscape, but the package 'color.sty' is not loaded}%
    \renewcommand\color[2][]{}%
  }%
  \providecommand\transparent[1]{%
    \errmessage{(Inkscape) Transparency is used (non-zero) for the text in Inkscape, but the package 'transparent.sty' is not loaded}%
    \renewcommand\transparent[1]{}%
  }%
  \providecommand\rotatebox[2]{#2}%
  \ifx\svgwidth\undefined%
    \setlength{\unitlength}{165.44649891bp}%
    \ifx\svgscale\undefined%
      \relax%
    \else%
      \setlength{\unitlength}{\unitlength * \real{\svgscale}}%
    \fi%
  \else%
    \setlength{\unitlength}{\svgwidth}%
  \fi%
  \global\let\svgwidth\undefined%
  \global\let\svgscale\undefined%
  \makeatother%
  \begin{picture}(1,0.52659443)%
    \put(0,0){\includegraphics[width=\unitlength]{ydcond3.pdf}}%
    \put(0.89960847,0.48875597){\color[rgb]{0,0,0}\makebox(0,0)[lb]{\smash{$B$}}}%
    \put(0.73249838,0.48879426){\color[rgb]{0,0,0}\makebox(0,0)[lb]{\smash{$V$}}}%
    \put(0.56357266,0.00498178){\color[rgb]{0,0,0}\makebox(0,0)[lb]{\smash{$C$}}}%
    \put(0.73249838,0.00525426){\color[rgb]{0,0,0}\makebox(0,0)[lb]{\smash{$V$}}}%
    \put(0.07488398,0.00525426){\color[rgb]{0,0,0}\makebox(0,0)[lb]{\smash{$C$}}}%
    \put(0.25379378,0.00525426){\color[rgb]{0,0,0}\makebox(0,0)[lb]{\smash{$V$}}}%
    \put(0.10053427,0.48879426){\color[rgb]{0,0,0}\makebox(0,0)[lb]{\smash{$V$}}}%
    \put(0.27313538,0.48879426){\color[rgb]{0,0,0}\makebox(0,0)[lb]{\smash{$B$}}}%
    \put(0.37467878,0.22284726){\color[rgb]{0,0,0}\makebox(0,0)[lb]{\smash{$=$}}}%
    \put(0.0162116,0.16976927){\color[rgb]{0,0,0}\makebox(0,0)[lb]{\smash{$G$}}}%
    \put(0.64135962,0.16976927){\color[rgb]{0,0,0}\makebox(0,0)[lb]{\smash{$G$}}}%
  \end{picture}%
\endgroup%

%% file: Graphics/hopfcond3.pdf_tex
\begingroup%
  \makeatletter%
  \providecommand\color[2][]{%
    \errmessage{(Inkscape) Color is used for the text in Inkscape, but the package 'color.sty' is not loaded}%
    \renewcommand\color[2][]{}%
  }%
  \providecommand\transparent[1]{%
    \errmessage{(Inkscape) Transparency is used (non-zero) for the text in Inkscape, but the package 'transparent.sty' is not loaded}%
    \renewcommand\transparent[1]{}%
  }%
  \providecommand\rotatebox[2]{#2}%
  \ifx\svgwidth\undefined%
    \setlength{\unitlength}{141.20059575bp}%
    \ifx\svgscale\undefined%
      \relax%
    \else%
      \setlength{\unitlength}{\unitlength * \real{\svgscale}}%
    \fi%
  \else%
    \setlength{\unitlength}{\svgwidth}%
  \fi%
  \global\let\svgwidth\undefined%
  \global\let\svgscale\undefined%
  \makeatother%
  \begin{picture}(1,0.62101946)%
    \put(0,0){\includegraphics[width=\unitlength]{hopfcond3.pdf}}%
    \put(-0.00201951,0.00817997){\color[rgb]{0,0,0}\makebox(0,0)[lb]{\smash{$C$}}}%
    \put(0.17361714,0.00817997){\color[rgb]{0,0,0}\makebox(0,0)[lb]{\smash{$V$}}}%
    \put(0.00364619,0.57474982){\color[rgb]{0,0,0}\makebox(0,0)[lb]{\smash{$V$}}}%
    \put(0.17361714,0.57474982){\color[rgb]{0,0,0}\makebox(0,0)[lb]{\smash{$B$}}}%
    \put(0.26426832,0.30016583){\color[rgb]{0,0,0}\makebox(0,0)[lb]{\smash{$=$}}}%
    \put(0.87933497,0.57672853){\color[rgb]{0,0,0}\makebox(0,0)[lb]{\smash{$B$}}}%
    \put(0.68353,0.57474982){\color[rgb]{0,0,0}\makebox(0,0)[lb]{\smash{$V$}}}%
    \put(0.48559763,0.00583722){\color[rgb]{0,0,0}\makebox(0,0)[lb]{\smash{$C$}}}%
    \put(0.68353,0.00817997){\color[rgb]{0,0,0}\makebox(0,0)[lb]{\smash{$V$}}}%
    \put(0.57674193,0.20357473){\color[rgb]{0,0,0}\makebox(0,0)[lb]{\smash{$G$}}}%
  \end{picture}%
\endgroup%

%% file: Graphics/ydcond.pdf_tex
\begingroup%
  \makeatletter%
  \providecommand\color[2][]{%
    \errmessage{(Inkscape) Color is used for the text in Inkscape, but the package 'color.sty' is not loaded}%
    \renewcommand\color[2][]{}%
  }%
  \providecommand\transparent[1]{%
    \errmessage{(Inkscape) Transparency is used (non-zero) for the text in Inkscape, but the package 'transparent.sty' is not loaded}%
    \renewcommand\transparent[1]{}%
  }%
  \providecommand\rotatebox[2]{#2}%
  \ifx\svgwidth\undefined%
    \setlength{\unitlength}{138.32110577bp}%
    \ifx\svgscale\undefined%
      \relax%
    \else%
      \setlength{\unitlength}{\unitlength * \real{\svgscale}}%
    \fi%
  \else%
    \setlength{\unitlength}{\svgwidth}%
  \fi%
  \global\let\svgwidth\undefined%
  \global\let\svgscale\undefined%
  \makeatother%
  \begin{picture}(1,0.6339474)%
    \put(0,0){\includegraphics[width=\unitlength]{ydcond.pdf}}%
    \put(0.52123492,0.58873444){\color[rgb]{0,0,0}\makebox(0,0)[lb]{\smash{$B$}}}%
    \put(0.89971817,0.58671464){\color[rgb]{0,0,0}\makebox(0,0)[lb]{\smash{$V$}}}%
    \put(0.52415604,0.00595873){\color[rgb]{0,0,0}\makebox(0,0)[lb]{\smash{$C$}}}%
    \put(0.89971817,0.00835026){\color[rgb]{0,0,0}\makebox(0,0)[lb]{\smash{$V$}}}%
    \put(0.05530618,0.00835026){\color[rgb]{0,0,0}\makebox(0,0)[lb]{\smash{$C$}}}%
    \put(0.24079597,0.00835026){\color[rgb]{0,0,0}\makebox(0,0)[lb]{\smash{$V$}}}%
    \put(0.06108982,0.58671464){\color[rgb]{0,0,0}\makebox(0,0)[lb]{\smash{$B$}}}%
    \put(0.24079597,0.58671464){\color[rgb]{0,0,0}\makebox(0,0)[lb]{\smash{$V$}}}%
    \put(0.30978651,0.30331609){\color[rgb]{0,0,0}\makebox(0,0)[lb]{\smash{$=$}}}%
    \put(0.20821117,0.2980789){\color[rgb]{0,0,0}\makebox(0,0)[lb]{\smash{$G$}}}%
    \put(0.43749122,0.29766571){\color[rgb]{0,0,0}\makebox(0,0)[lb]{\smash{$G$}}}%
  \end{picture}%
\endgroup%

%% file: Graphics/hopfcond.pdf_tex
\begingroup%
  \makeatletter%
  \providecommand\color[2][]{%
    \errmessage{(Inkscape) Color is used for the text in Inkscape, but the package 'color.sty' is not loaded}%
    \renewcommand\color[2][]{}%
  }%
  \providecommand\transparent[1]{%
    \errmessage{(Inkscape) Transparency is used (non-zero) for the text in Inkscape, but the package 'transparent.sty' is not loaded}%
    \renewcommand\transparent[1]{}%
  }%
  \providecommand\rotatebox[2]{#2}%
  \ifx\svgwidth\undefined%
    \setlength{\unitlength}{138.3859375bp}%
    \ifx\svgscale\undefined%
      \relax%
    \else%
      \setlength{\unitlength}{\unitlength * \real{\svgscale}}%
    \fi%
  \else%
    \setlength{\unitlength}{\svgwidth}%
  \fi%
  \global\let\svgwidth\undefined%
  \global\let\svgscale\undefined%
  \makeatother%
  \begin{picture}(1,0.63365041)%
    \put(0,0){\includegraphics[width=\unitlength]{hopfcond.pdf}}%
    \put(0.52145921,0.58845863){\color[rgb]{0,0,0}\makebox(0,0)[lb]{\smash{$B$}}}%
    \put(0.89976515,0.58643977){\color[rgb]{0,0,0}\makebox(0,0)[lb]{\smash{$V$}}}%
    \put(0.52437897,0.00595594){\color[rgb]{0,0,0}\makebox(0,0)[lb]{\smash{$C$}}}%
    \put(0.89976515,0.00834635){\color[rgb]{0,0,0}\makebox(0,0)[lb]{\smash{$V$}}}%
    \put(-0.00206059,0.00834635){\color[rgb]{0,0,0}\makebox(0,0)[lb]{\smash{$C$}}}%
    \put(0.17714837,0.00834635){\color[rgb]{0,0,0}\makebox(0,0)[lb]{\smash{$V$}}}%
    \put(0.00372035,0.58643977){\color[rgb]{0,0,0}\makebox(0,0)[lb]{\smash{$B$}}}%
    \put(0.17714837,0.58643977){\color[rgb]{0,0,0}\makebox(0,0)[lb]{\smash{$V$}}}%
    \put(0.26964332,0.30317399){\color[rgb]{0,0,0}\makebox(0,0)[lb]{\smash{$=$}}}%
    \put(0.4377551,0.29793925){\color[rgb]{0,0,0}\makebox(0,0)[lb]{\smash{$G$}}}%
  \end{picture}%
\endgroup%

%% file: Graphics/ydbraiding3b.pdf_tex
\begingroup%
  \makeatletter%
  \providecommand\color[2][]{%
    \errmessage{(Inkscape) Color is used for the text in Inkscape, but the package 'color.sty' is not loaded}%
    \renewcommand\color[2][]{}%
  }%
  \providecommand\transparent[1]{%
    \errmessage{(Inkscape) Transparency is used (non-zero) for the text in Inkscape, but the package 'transparent.sty' is not loaded}%
    \renewcommand\transparent[1]{}%
  }%
  \providecommand\rotatebox[2]{#2}%
  \ifx\svgwidth\undefined%
    \setlength{\unitlength}{48.65589871bp}%
    \ifx\svgscale\undefined%
      \relax%
    \else%
      \setlength{\unitlength}{\unitlength * \real{\svgscale}}%
    \fi%
  \else%
    \setlength{\unitlength}{\svgwidth}%
  \fi%
  \global\let\svgwidth\undefined%
  \global\let\svgscale\undefined%
  \makeatother%
  \begin{picture}(1,1.50257127)%
    \put(0,0){\includegraphics[width=\unitlength]{ydbraiding3b.pdf}}%
  \end{picture}%
\endgroup%

%% file: Graphics/ydbraiding1.pdf_tex
\begingroup%
  \makeatletter%
  \providecommand\color[2][]{%
    \errmessage{(Inkscape) Color is used for the text in Inkscape, but the package 'color.sty' is not loaded}%
    \renewcommand\color[2][]{}%
  }%
  \providecommand\transparent[1]{%
    \errmessage{(Inkscape) Transparency is used (non-zero) for the text in Inkscape, but the package 'transparent.sty' is not loaded}%
    \renewcommand\transparent[1]{}%
  }%
  \providecommand\rotatebox[2]{#2}%
  \ifx\svgwidth\undefined%
    \setlength{\unitlength}{48.65816368bp}%
    \ifx\svgscale\undefined%
      \relax%
    \else%
      \setlength{\unitlength}{\unitlength * \real{\svgscale}}%
    \fi%
  \else%
    \setlength{\unitlength}{\svgwidth}%
  \fi%
  \global\let\svgwidth\undefined%
  \global\let\svgscale\undefined%
  \makeatother%
  \begin{picture}(1,1.49684966)%
    \put(0,0){\includegraphics[width=\unitlength]{ydbraiding1.pdf}}%
  \end{picture}%
\endgroup%

%% file: Graphics/quasiydright.pdf_tex
\begingroup%
  \makeatletter%
  \providecommand\color[2][]{%
    \errmessage{(Inkscape) Color is used for the text in Inkscape, but the package 'color.sty' is not loaded}%
    \renewcommand\color[2][]{}%
  }%
  \providecommand\transparent[1]{%
    \errmessage{(Inkscape) Transparency is used (non-zero) for the text in Inkscape, but the package 'transparent.sty' is not loaded}%
    \renewcommand\transparent[1]{}%
  }%
  \providecommand\rotatebox[2]{#2}%
  \ifx\svgwidth\undefined%
    \setlength{\unitlength}{162.1868988bp}%
    \ifx\svgscale\undefined%
      \relax%
    \else%
      \setlength{\unitlength}{\unitlength * \real{\svgscale}}%
    \fi%
  \else%
    \setlength{\unitlength}{\svgwidth}%
  \fi%
  \global\let\svgwidth\undefined%
  \global\let\svgscale\undefined%
  \makeatother%
  \begin{picture}(1,0.89278751)%
    \put(0,0){\includegraphics[width=\unitlength]{quasiydright.pdf}}%
    \put(0.48972445,0.86802872){\color[rgb]{0,0,0}\makebox(0,0)[lb]{\smash{$V$}}}%
    \put(0.93816378,0.00304914){\color[rgb]{0,0,0}\makebox(0,0)[lb]{\smash{$C$}}}%
    \put(0.53619856,0.32137372){\color[rgb]{0,0,0}\makebox(0,0)[lb]{\smash{$b$}}}%
    \put(0.67504049,0.30112753){\color[rgb]{0,0,0}\makebox(0,0)[lb]{\smash{$a$}}}%
    \put(0.86342388,0.00353731){\color[rgb]{0,0,0}\makebox(0,0)[lb]{\smash{$V$}}}%
    \put(-0.00105492,0.44181101){\color[rgb]{0,0,0}\makebox(0,0)[lb]{\smash{$\delta'=$}}}%
    \put(0.47133273,0.5971715){\color[rgb]{0,0,0}\makebox(0,0)[lb]{\smash{$G$}}}%
    \put(0.60451242,0.48985038){\color[rgb]{0,0,0}\makebox(0,0)[lb]{\smash{$G$}}}%
    \put(0.65378817,0.43310071){\color[rgb]{0,0,0}\makebox(0,0)[lb]{\smash{$G$}}}%
    \put(0.6886666,0.36673576){\color[rgb]{0,0,0}\makebox(0,0)[lb]{\smash{$G$}}}%
    \put(0.93992792,0.32587955){\color[rgb]{0,0,0}\makebox(0,0)[lb]{\smash{$G$}}}%
    \put(0.62270882,0.75284826){\color[rgb]{0,0,0}\makebox(0,0)[lb]{\smash{$\phi$}}}%
    \put(0.7573032,0.6171222){\color[rgb]{0,0,0}\makebox(0,0)[lb]{\smash{$\phi^{\text{-}1}$}}}%
    \put(0.88159418,0.60975343){\color[rgb]{0,0,0}\makebox(0,0)[lb]{\smash{$\phi$}}}%
    \put(0.59091028,0.28148632){\color[rgb]{0,0,0}\makebox(0,0)[lb]{\smash{$S$}}}%
    \put(0.39492988,0.44126247){\color[rgb]{0,0,0}\makebox(0,0)[lb]{\smash{$=$}}}%
    \put(0.17329582,0.86965159){\color[rgb]{0,0,0}\makebox(0,0)[lb]{\smash{$V$}}}%
    \put(0.17291091,0.00592819){\color[rgb]{0,0,0}\makebox(0,0)[lb]{\smash{$V$}}}%
    \put(0.32088833,0.00592819){\color[rgb]{0,0,0}\makebox(0,0)[lb]{\smash{$C$}}}%
  \end{picture}%
\endgroup%

%% file: Graphics/deltadash.pdf_tex
\begingroup%
  \makeatletter%
  \providecommand\color[2][]{%
    \errmessage{(Inkscape) Color is used for the text in Inkscape, but the package 'color.sty' is not loaded}%
    \renewcommand\color[2][]{}%
  }%
  \providecommand\transparent[1]{%
    \errmessage{(Inkscape) Transparency is used (non-zero) for the text in Inkscape, but the package 'transparent.sty' is not loaded}%
    \renewcommand\transparent[1]{}%
  }%
  \providecommand\rotatebox[2]{#2}%
  \ifx\svgwidth\undefined%
    \setlength{\unitlength}{24.80667668bp}%
    \ifx\svgscale\undefined%
      \relax%
    \else%
      \setlength{\unitlength}{\unitlength * \real{\svgscale}}%
    \fi%
  \else%
    \setlength{\unitlength}{\svgwidth}%
  \fi%
  \global\let\svgwidth\undefined%
  \global\let\svgscale\undefined%
  \makeatother%
  \begin{picture}(1,1.00109966)%
    \put(0,0){\includegraphics[width=\unitlength]{deltadash.pdf}}%
  \end{picture}%
\endgroup%

%% file: Graphics/ydduals.pdf_tex
\begingroup%
  \makeatletter%
  \providecommand\color[2][]{%
    \errmessage{(Inkscape) Color is used for the text in Inkscape, but the package 'color.sty' is not loaded}%
    \renewcommand\color[2][]{}%
  }%
  \providecommand\transparent[1]{%
    \errmessage{(Inkscape) Transparency is used (non-zero) for the text in Inkscape, but the package 'transparent.sty' is not loaded}%
    \renewcommand\transparent[1]{}%
  }%
  \providecommand\rotatebox[2]{#2}%
  \ifx\svgwidth\undefined%
    \setlength{\unitlength}{212.66274414bp}%
    \ifx\svgscale\undefined%
      \relax%
    \else%
      \setlength{\unitlength}{\unitlength * \real{\svgscale}}%
    \fi%
  \else%
    \setlength{\unitlength}{\svgwidth}%
  \fi%
  \global\let\svgwidth\undefined%
  \global\let\svgscale\undefined%
  \makeatother%
  \begin{picture}(1,0.97967182)%
    \put(0,0){\includegraphics[width=\unitlength]{ydduals.pdf}}%
    \put(-0.00080453,0.48383901){\color[rgb]{0,0,0}\makebox(0,0)[lb]{\smash{$\delta^*:=$}}}%
    \put(0.43865722,0.96159077){\color[rgb]{0,0,0}\makebox(0,0)[lb]{\smash{$V^*$}}}%
    \put(0.73893144,0.00232542){\color[rgb]{0,0,0}\makebox(0,0)[lb]{\smash{$B$}}}%
    \put(0.85178619,0.00232542){\color[rgb]{0,0,0}\makebox(0,0)[lb]{\smash{$V^*$}}}%
    \put(0.93952793,0.90169845){\color[rgb]{0,0,0}\makebox(0,0)[lb]{\smash{$\phi$}}}%
    \put(0.77741828,0.6299803){\color[rgb]{0,0,0}\makebox(0,0)[lb]{\smash{$\phi^{\text{-}1}$}}}%
    \put(0.79174199,0.31619718){\color[rgb]{0,0,0}\makebox(0,0)[lb]{\smash{$\phi$}}}%
    \put(0.79334347,0.23885062){\color[rgb]{0,0,0}\makebox(0,0)[lb]{\smash{$G$}}}%
    \put(0.65791778,0.38932362){\color[rgb]{0,0,0}\makebox(0,0)[lb]{\smash{$G$}}}%
    \put(0.64951624,0.14470489){\color[rgb]{0,0,0}\makebox(0,0)[lb]{\smash{$S^{\text{-}1}$}}}%
    \put(0.66414577,0.21189356){\color[rgb]{0,0,0}\makebox(0,0)[lb]{\smash{$b$}}}%
    \put(0.76645379,0.78868114){\color[rgb]{0,0,0}\makebox(0,0)[lb]{\smash{$a$}}}%
    \put(0.74922365,0.71650227){\color[rgb]{0,0,0}\makebox(0,0)[lb]{\smash{$S^{\text{-}1}$}}}%
    \put(0.33255976,0.48484398){\color[rgb]{0,0,0}\makebox(0,0)[lb]{\smash{$=$}}}%
    \put(0.25381469,0.96202724){\color[rgb]{0,0,0}\makebox(0,0)[lb]{\smash{$V^*$}}}%
    \put(0.16323735,0.002364){\color[rgb]{0,0,0}\makebox(0,0)[lb]{\smash{$B$}}}%
    \put(0.27609209,0.002364){\color[rgb]{0,0,0}\makebox(0,0)[lb]{\smash{$V^*$}}}%
  \end{picture}%
\endgroup%

%% file: Graphics/ydbraiding4b.pdf_tex
\begingroup%
  \makeatletter%
  \providecommand\color[2][]{%
    \errmessage{(Inkscape) Color is used for the text in Inkscape, but the package 'color.sty' is not loaded}%
    \renewcommand\color[2][]{}%
  }%
  \providecommand\transparent[1]{%
    \errmessage{(Inkscape) Transparency is used (non-zero) for the text in Inkscape, but the package 'transparent.sty' is not loaded}%
    \renewcommand\transparent[1]{}%
  }%
  \providecommand\rotatebox[2]{#2}%
  \ifx\svgwidth\undefined%
    \setlength{\unitlength}{48.805737bp}%
    \ifx\svgscale\undefined%
      \relax%
    \else%
      \setlength{\unitlength}{\unitlength * \real{\svgscale}}%
    \fi%
  \else%
    \setlength{\unitlength}{\svgwidth}%
  \fi%
  \global\let\svgwidth\undefined%
  \global\let\svgscale\undefined%
  \makeatother%
  \begin{picture}(1,1.43196747)%
    \put(0,0){\includegraphics[width=\unitlength]{ydbraiding4b.pdf}}%
    \put(0.41707641,1.03808748){\color[rgb]{0,0,0}\makebox(0,0)[lb]{\smash{$S$}}}%
  \end{picture}%
\endgroup%

%% file: Graphics/altydcond3.pdf_tex
\begingroup%
  \makeatletter%
  \providecommand\color[2][]{%
    \errmessage{(Inkscape) Color is used for the text in Inkscape, but the package 'color.sty' is not loaded}%
    \renewcommand\color[2][]{}%
  }%
  \providecommand\transparent[1]{%
    \errmessage{(Inkscape) Transparency is used (non-zero) for the text in Inkscape, but the package 'transparent.sty' is not loaded}%
    \renewcommand\transparent[1]{}%
  }%
  \providecommand\rotatebox[2]{#2}%
  \ifx\svgwidth\undefined%
    \setlength{\unitlength}{161.04819336bp}%
    \ifx\svgscale\undefined%
      \relax%
    \else%
      \setlength{\unitlength}{\unitlength * \real{\svgscale}}%
    \fi%
  \else%
    \setlength{\unitlength}{\svgwidth}%
  \fi%
  \global\let\svgwidth\undefined%
  \global\let\svgscale\undefined%
  \makeatother%
  \begin{picture}(1,0.5541412)%
    \put(0,0){\includegraphics[width=\unitlength]{altydcond3.pdf}}%
    \put(0.28066133,0.25663634){\color[rgb]{0,0,0}\makebox(0,0)[lb]{\smash{$=$}}}%
    \put(0.81169748,0.17388698){\color[rgb]{0,0,0}\makebox(0,0)[lb]{\smash{$S^{\text{-}1}$}}}%
  \end{picture}%
\endgroup%

%% file: Graphics/altydcond4.pdf_tex
\begingroup%
  \makeatletter%
  \providecommand\color[2][]{%
    \errmessage{(Inkscape) Color is used for the text in Inkscape, but the package 'color.sty' is not loaded}%
    \renewcommand\color[2][]{}%
  }%
  \providecommand\transparent[1]{%
    \errmessage{(Inkscape) Transparency is used (non-zero) for the text in Inkscape, but the package 'transparent.sty' is not loaded}%
    \renewcommand\transparent[1]{}%
  }%
  \providecommand\rotatebox[2]{#2}%
  \ifx\svgwidth\undefined%
    \setlength{\unitlength}{145.00008bp}%
    \ifx\svgscale\undefined%
      \relax%
    \else%
      \setlength{\unitlength}{\unitlength * \real{\svgscale}}%
    \fi%
  \else%
    \setlength{\unitlength}{\svgwidth}%
  \fi%
  \global\let\svgwidth\undefined%
  \global\let\svgscale\undefined%
  \makeatother%
  \begin{picture}(1,0.61350884)%
    \put(0,0){\includegraphics[width=\unitlength]{altydcond4.pdf}}%
    \put(0.39448282,0.28355813){\color[rgb]{0,0,0}\makebox(0,0)[lb]{\smash{$=$}}}%
    \put(0.5304295,0.39991212){\color[rgb]{0,0,0}\makebox(0,0)[lb]{\smash{$S^{\text{-}1}$}}}%
  \end{picture}%
\endgroup%

%% file: Graphics/rmatrix.pdf_tex
\begingroup%
  \makeatletter%
  \providecommand\color[2][]{%
    \errmessage{(Inkscape) Color is used for the text in Inkscape, but the package 'color.sty' is not loaded}%
    \renewcommand\color[2][]{}%
  }%
  \providecommand\transparent[1]{%
    \errmessage{(Inkscape) Transparency is used (non-zero) for the text in Inkscape, but the package 'transparent.sty' is not loaded}%
    \renewcommand\transparent[1]{}%
  }%
  \providecommand\rotatebox[2]{#2}%
  \ifx\svgwidth\undefined%
    \setlength{\unitlength}{388.09213422bp}%
    \ifx\svgscale\undefined%
      \relax%
    \else%
      \setlength{\unitlength}{\unitlength * \real{\svgscale}}%
    \fi%
  \else%
    \setlength{\unitlength}{\svgwidth}%
  \fi%
  \global\let\svgwidth\undefined%
  \global\let\svgscale\undefined%
  \makeatother%
  \begin{picture}(1,0.17224429)%
    \put(0,0){\includegraphics[width=\unitlength]{rmatrix.pdf}}%
    \put(0.00874255,0.14059489){\color[rgb]{0,0,0}\makebox(0,0)[lb]{\smash{$R$}}}%
    \put(0.09203431,0.14137583){\color[rgb]{0,0,0}\makebox(0,0)[lb]{\smash{$R$}}}%
    \put(0.13138447,0.08381685){\color[rgb]{0,0,0}\makebox(0,0)[lb]{\smash{$=$}}}%
    \put(0.22319824,0.13028806){\color[rgb]{0,0,0}\makebox(0,0)[lb]{\smash{$R$}}}%
    \put(0.56559652,0.13122828){\color[rgb]{0,0,0}\makebox(0,0)[lb]{\smash{$R$}}}%
    \put(0.34296897,0.14153511){\color[rgb]{0,0,0}\makebox(0,0)[lb]{\smash{$R$}}}%
    \put(0.42542362,0.14153511){\color[rgb]{0,0,0}\makebox(0,0)[lb]{\smash{$R$}}}%
    \put(0.47489641,0.08381685){\color[rgb]{0,0,0}\makebox(0,0)[lb]{\smash{$=$}}}%
    \put(0.29953301,0.08381685){\color[rgb]{0,0,0}\makebox(0,0)[lb]{\smash{$,$}}}%
    \put(0.63700819,0.08308064){\color[rgb]{0,0,0}\makebox(0,0)[lb]{\smash{$,$}}}%
    \put(0.67868656,0.14059489){\color[rgb]{0,0,0}\makebox(0,0)[lb]{\smash{$R$}}}%
    \put(0.81684024,0.08324949){\color[rgb]{0,0,0}\makebox(0,0)[lb]{\smash{$=$}}}%
    \put(0.52024646,0.04052817){\color[rgb]{0,0,0}\makebox(0,0)[lb]{\smash{$\text{cop}$}}}%
    \put(0.96313094,0.14153511){\color[rgb]{0,0,0}\makebox(0,0)[lb]{\smash{$R$}}}%
    \put(0.86655598,0.10443052){\color[rgb]{0,0,0}\makebox(0,0)[lb]{\smash{$\text{cop}$}}}%
    \put(0.00855166,0.00472665){\color[rgb]{0,0,0}\makebox(0,0)[lb]{\smash{$B$}}}%
    \put(0.05956786,0.00212377){\color[rgb]{0,0,0}\makebox(0,0)[lb]{\smash{$B$}}}%
    \put(0.10017265,0.0031649){\color[rgb]{0,0,0}\makebox(0,0)[lb]{\smash{$B$}}}%
    \put(0.16882819,0.0031649){\color[rgb]{0,0,0}\makebox(0,0)[lb]{\smash{$B$}}}%
    \put(0.21984436,0.00524715){\color[rgb]{0,0,0}\makebox(0,0)[lb]{\smash{$B$}}}%
    \put(0.28335436,0.00420602){\color[rgb]{0,0,0}\makebox(0,0)[lb]{\smash{$B$}}}%
    \put(0.3333294,0.0031649){\color[rgb]{0,0,0}\makebox(0,0)[lb]{\smash{$B$}}}%
    \put(0.37913988,0.0031649){\color[rgb]{0,0,0}\makebox(0,0)[lb]{\smash{$B$}}}%
    \put(0.42651211,0.00264427){\color[rgb]{0,0,0}\makebox(0,0)[lb]{\smash{$B$}}}%
    \put(0.50380812,0.00420602){\color[rgb]{0,0,0}\makebox(0,0)[lb]{\smash{$B$}}}%
    \put(0.56575643,0.0036854){\color[rgb]{0,0,0}\makebox(0,0)[lb]{\smash{$B$}}}%
    \put(0.61625198,0.0036854){\color[rgb]{0,0,0}\makebox(0,0)[lb]{\smash{$B$}}}%
    \put(0.76478295,0.1561298){\color[rgb]{0,0,0}\makebox(0,0)[lb]{\smash{$B$}}}%
    \put(0.68357337,0.00420602){\color[rgb]{0,0,0}\makebox(0,0)[lb]{\smash{$B$}}}%
    \put(0.76738582,0.00212377){\color[rgb]{0,0,0}\makebox(0,0)[lb]{\smash{$B$}}}%
    \put(0.86966929,0.0031649){\color[rgb]{0,0,0}\makebox(0,0)[lb]{\smash{$B$}}}%
    \put(0.95764625,0.0031649){\color[rgb]{0,0,0}\makebox(0,0)[lb]{\smash{$B$}}}%
    \put(0.87539567,0.15560924){\color[rgb]{0,0,0}\makebox(0,0)[lb]{\smash{$B$}}}%
  \end{picture}%
\endgroup%

%% file: Graphics/ydcond4.pdf_tex
\begingroup%
  \makeatletter%
  \providecommand\color[2][]{%
    \errmessage{(Inkscape) Color is used for the text in Inkscape, but the package 'color.sty' is not loaded}%
    \renewcommand\color[2][]{}%
  }%
  \providecommand\transparent[1]{%
    \errmessage{(Inkscape) Transparency is used (non-zero) for the text in Inkscape, but the package 'transparent.sty' is not loaded}%
    \renewcommand\transparent[1]{}%
  }%
  \providecommand\rotatebox[2]{#2}%
  \ifx\svgwidth\undefined%
    \setlength{\unitlength}{188.607184bp}%
    \ifx\svgscale\undefined%
      \relax%
    \else%
      \setlength{\unitlength}{\unitlength * \real{\svgscale}}%
    \fi%
  \else%
    \setlength{\unitlength}{\svgwidth}%
  \fi%
  \global\let\svgwidth\undefined%
  \global\let\svgscale\undefined%
  \makeatother%
  \begin{picture}(1,0.46472016)%
    \put(0,0){\includegraphics[width=\unitlength]{ydcond4.pdf}}%
    \put(0.91193652,0.43001332){\color[rgb]{0,0,0}\makebox(0,0)[lb]{\smash{$B$}}}%
    \put(0.76534731,0.42853203){\color[rgb]{0,0,0}\makebox(0,0)[lb]{\smash{$V$}}}%
    \put(0.76534731,0.00437003){\color[rgb]{0,0,0}\makebox(0,0)[lb]{\smash{$V$}}}%
    \put(0.19272861,0.00437003){\color[rgb]{0,0,0}\makebox(0,0)[lb]{\smash{$V$}}}%
    \put(0.19727326,0.43156179){\color[rgb]{0,0,0}\makebox(0,0)[lb]{\smash{$V$}}}%
    \put(0.34997158,0.43156179){\color[rgb]{0,0,0}\makebox(0,0)[lb]{\smash{$B$}}}%
    \put(0.46843391,0.21220941){\color[rgb]{0,0,0}\makebox(0,0)[lb]{\smash{$=$}}}%
    \put(0.04033332,0.43156179){\color[rgb]{0,0,0}\makebox(0,0)[lb]{\smash{$C$}}}%
    \put(0.61840549,0.42853203){\color[rgb]{0,0,0}\makebox(0,0)[lb]{\smash{$C$}}}%
  \end{picture}%
\endgroup%

%% file: Graphics/hopfcond4.pdf_tex
\begingroup%
  \makeatletter%
  \providecommand\color[2][]{%
    \errmessage{(Inkscape) Color is used for the text in Inkscape, but the package 'color.sty' is not loaded}%
    \renewcommand\color[2][]{}%
  }%
  \providecommand\transparent[1]{%
    \errmessage{(Inkscape) Transparency is used (non-zero) for the text in Inkscape, but the package 'transparent.sty' is not loaded}%
    \renewcommand\transparent[1]{}%
  }%
  \providecommand\rotatebox[2]{#2}%
  \ifx\svgwidth\undefined%
    \setlength{\unitlength}{160.15472422bp}%
    \ifx\svgscale\undefined%
      \relax%
    \else%
      \setlength{\unitlength}{\unitlength * \real{\svgscale}}%
    \fi%
  \else%
    \setlength{\unitlength}{\svgwidth}%
  \fi%
  \global\let\svgwidth\undefined%
  \global\let\svgscale\undefined%
  \makeatother%
  \begin{picture}(1,0.54728052)%
    \put(0,0){\includegraphics[width=\unitlength]{hopfcond4.pdf}}%
    \put(0.72323128,0.00514639){\color[rgb]{0,0,0}\makebox(0,0)[lb]{\smash{$V$}}}%
    \put(0.72858331,0.50823137){\color[rgb]{0,0,0}\makebox(0,0)[lb]{\smash{$V$}}}%
    \put(0.90840942,0.50823137){\color[rgb]{0,0,0}\makebox(0,0)[lb]{\smash{$B$}}}%
    \put(0.37856458,0.2499097){\color[rgb]{0,0,0}\makebox(0,0)[lb]{\smash{$=$}}}%
    \put(0.54376204,0.50823137){\color[rgb]{0,0,0}\makebox(0,0)[lb]{\smash{$C$}}}%
    \put(-0.0010683,0.50466334){\color[rgb]{0,0,0}\makebox(0,0)[lb]{\smash{$C$}}}%
    \put(0.15235477,0.50466334){\color[rgb]{0,0,0}\makebox(0,0)[lb]{\smash{$V$}}}%
    \put(0.30220985,0.50287933){\color[rgb]{0,0,0}\makebox(0,0)[lb]{\smash{$B$}}}%
    \put(0.14700281,0.00871434){\color[rgb]{0,0,0}\makebox(0,0)[lb]{\smash{$V$}}}%
  \end{picture}%
\endgroup%

%% file: Graphics/weakqt1.pdf_tex
\begingroup%
  \makeatletter%
  \providecommand\color[2][]{%
    \errmessage{(Inkscape) Color is used for the text in Inkscape, but the package 'color.sty' is not loaded}%
    \renewcommand\color[2][]{}%
  }%
  \providecommand\transparent[1]{%
    \errmessage{(Inkscape) Transparency is used (non-zero) for the text in Inkscape, but the package 'transparent.sty' is not loaded}%
    \renewcommand\transparent[1]{}%
  }%
  \providecommand\rotatebox[2]{#2}%
  \ifx\svgwidth\undefined%
    \setlength{\unitlength}{72.77504bp}%
    \ifx\svgscale\undefined%
      \relax%
    \else%
      \setlength{\unitlength}{\unitlength * \real{\svgscale}}%
    \fi%
  \else%
    \setlength{\unitlength}{\svgwidth}%
  \fi%
  \global\let\svgwidth\undefined%
  \global\let\svgscale\undefined%
  \makeatother%
  \begin{picture}(1,0.51288189)%
    \put(0,0){\includegraphics[width=\unitlength]{weakqt1.pdf}}%
    \put(0.35684601,0.26975855){\color[rgb]{0,0,0}\makebox(0,0)[lb]{\smash{$=$}}}%
    \put(0.09162305,0.15446251){\color[rgb]{0,0,0}\makebox(0,0)[lb]{\smash{$R$}}}%
    \put(0.60539629,0.15168614){\color[rgb]{0,0,0}\makebox(0,0)[lb]{\smash{$R$}}}%
    \put(0.60561904,0.33045971){\color[rgb]{0,0,0}\makebox(0,0)[lb]{\smash{$S$}}}%
  \end{picture}%
\endgroup%

%% file: Graphics/weakqt2.pdf_tex
\begingroup%
  \makeatletter%
  \providecommand\color[2][]{%
    \errmessage{(Inkscape) Color is used for the text in Inkscape, but the package 'color.sty' is not loaded}%
    \renewcommand\color[2][]{}%
  }%
  \providecommand\transparent[1]{%
    \errmessage{(Inkscape) Transparency is used (non-zero) for the text in Inkscape, but the package 'transparent.sty' is not loaded}%
    \renewcommand\transparent[1]{}%
  }%
  \providecommand\rotatebox[2]{#2}%
  \ifx\svgwidth\undefined%
    \setlength{\unitlength}{72.77504bp}%
    \ifx\svgscale\undefined%
      \relax%
    \else%
      \setlength{\unitlength}{\unitlength * \real{\svgscale}}%
    \fi%
  \else%
    \setlength{\unitlength}{\svgwidth}%
  \fi%
  \global\let\svgwidth\undefined%
  \global\let\svgscale\undefined%
  \makeatother%
  \begin{picture}(1,0.51288189)%
    \put(0,0){\includegraphics[width=\unitlength]{weakqt2.pdf}}%
    \put(0.35684601,0.26975855){\color[rgb]{0,0,0}\makebox(0,0)[lb]{\smash{$=$}}}%
    \put(0.03911179,0.08670703){\color[rgb]{0,0,0}\makebox(0,0)[lb]{\smash{$R^{\oop}$}}}%
    \put(0.604358,0.17549533){\color[rgb]{0,0,0}\makebox(0,0)[lb]{\smash{$\ov{R}$}}}%
  \end{picture}%
\endgroup%

%% file: Graphics/dualrmatrixbraiding.pdf_tex
\begingroup%
  \makeatletter%
  \providecommand\color[2][]{%
    \errmessage{(Inkscape) Color is used for the text in Inkscape, but the package 'color.sty' is not loaded}%
    \renewcommand\color[2][]{}%
  }%
  \providecommand\transparent[1]{%
    \errmessage{(Inkscape) Transparency is used (non-zero) for the text in Inkscape, but the package 'transparent.sty' is not loaded}%
    \renewcommand\transparent[1]{}%
  }%
  \providecommand\rotatebox[2]{#2}%
  \ifx\svgwidth\undefined%
    \setlength{\unitlength}{43.197464bp}%
    \ifx\svgscale\undefined%
      \relax%
    \else%
      \setlength{\unitlength}{\unitlength * \real{\svgscale}}%
    \fi%
  \else%
    \setlength{\unitlength}{\svgwidth}%
  \fi%
  \global\let\svgwidth\undefined%
  \global\let\svgscale\undefined%
  \makeatother%
  \begin{picture}(1,0.95441387)%
    \put(0,0){\includegraphics[width=\unitlength]{dualrmatrixbraiding.pdf}}%
    \put(0.06209137,0.13030188){\color[rgb]{0,0,0}\makebox(0,0)[lb]{\smash{$R$}}}%
  \end{picture}%
\endgroup%

%% file: Graphics/dualrmatrixopbraiding.pdf_tex
\begingroup%
  \makeatletter%
  \providecommand\color[2][]{%
    \errmessage{(Inkscape) Color is used for the text in Inkscape, but the package 'color.sty' is not loaded}%
    \renewcommand\color[2][]{}%
  }%
  \providecommand\transparent[1]{%
    \errmessage{(Inkscape) Transparency is used (non-zero) for the text in Inkscape, but the package 'transparent.sty' is not loaded}%
    \renewcommand\transparent[1]{}%
  }%
  \providecommand\rotatebox[2]{#2}%
  \ifx\svgwidth\undefined%
    \setlength{\unitlength}{43.175072bp}%
    \ifx\svgscale\undefined%
      \relax%
    \else%
      \setlength{\unitlength}{\unitlength * \real{\svgscale}}%
    \fi%
  \else%
    \setlength{\unitlength}{\svgwidth}%
  \fi%
  \global\let\svgwidth\undefined%
  \global\let\svgscale\undefined%
  \makeatother%
  \begin{picture}(1,0.82567565)%
    \put(0,0){\includegraphics[width=\unitlength]{dualrmatrixopbraiding.pdf}}%
    \put(0.06212291,0.26649814){\color[rgb]{0,0,0}\makebox(0,0)[lb]{\smash{$\ov{R}$}}}%
  \end{picture}%
\endgroup%

%% file: Graphics/dualrmatrixbraiding2.pdf_tex
\begingroup%
  \makeatletter%
  \providecommand\color[2][]{%
    \errmessage{(Inkscape) Color is used for the text in Inkscape, but the package 'color.sty' is not loaded}%
    \renewcommand\color[2][]{}%
  }%
  \providecommand\transparent[1]{%
    \errmessage{(Inkscape) Transparency is used (non-zero) for the text in Inkscape, but the package 'transparent.sty' is not loaded}%
    \renewcommand\transparent[1]{}%
  }%
  \providecommand\rotatebox[2]{#2}%
  \ifx\svgwidth\undefined%
    \setlength{\unitlength}{49.395432bp}%
    \ifx\svgscale\undefined%
      \relax%
    \else%
      \setlength{\unitlength}{\unitlength * \real{\svgscale}}%
    \fi%
  \else%
    \setlength{\unitlength}{\svgwidth}%
  \fi%
  \global\let\svgwidth\undefined%
  \global\let\svgscale\undefined%
  \makeatother%
  \begin{picture}(1,1.16667182)%
    \put(0,0){\includegraphics[width=\unitlength]{dualrmatrixbraiding2.pdf}}%
    \put(0.05430024,0.42977073){\color[rgb]{0,0,0}\makebox(0,0)[lb]{\smash{$S$}}}%
    \put(0.05822613,0.15035115){\color[rgb]{0,0,0}\makebox(0,0)[lb]{\smash{$R$}}}%
  \end{picture}%
\endgroup%

%% file: Graphics/dualrmatrixbraiding3.pdf_tex
\begingroup%
  \makeatletter%
  \providecommand\color[2][]{%
    \errmessage{(Inkscape) Color is used for the text in Inkscape, but the package 'color.sty' is not loaded}%
    \renewcommand\color[2][]{}%
  }%
  \providecommand\transparent[1]{%
    \errmessage{(Inkscape) Transparency is used (non-zero) for the text in Inkscape, but the package 'transparent.sty' is not loaded}%
    \renewcommand\transparent[1]{}%
  }%
  \providecommand\rotatebox[2]{#2}%
  \ifx\svgwidth\undefined%
    \setlength{\unitlength}{49.217904bp}%
    \ifx\svgscale\undefined%
      \relax%
    \else%
      \setlength{\unitlength}{\unitlength * \real{\svgscale}}%
    \fi%
  \else%
    \setlength{\unitlength}{\svgwidth}%
  \fi%
  \global\let\svgwidth\undefined%
  \global\let\svgscale\undefined%
  \makeatother%
  \begin{picture}(1,1.32529534)%
    \put(0,0){\includegraphics[width=\unitlength]{dualrmatrixbraiding3.pdf}}%
    \put(0.05449633,0.58573626){\color[rgb]{0,0,0}\makebox(0,0)[lb]{\smash{$S$}}}%
    \put(0.05914125,0.29722535){\color[rgb]{0,0,0}\makebox(0,0)[lb]{\smash{$R$}}}%
  \end{picture}%
\endgroup%

%% file: Graphics/dualrmatrixopbraiding2.pdf_tex
\begingroup%
  \makeatletter%
  \providecommand\color[2][]{%
    \errmessage{(Inkscape) Color is used for the text in Inkscape, but the package 'color.sty' is not loaded}%
    \renewcommand\color[2][]{}%
  }%
  \providecommand\transparent[1]{%
    \errmessage{(Inkscape) Transparency is used (non-zero) for the text in Inkscape, but the package 'transparent.sty' is not loaded}%
    \renewcommand\transparent[1]{}%
  }%
  \providecommand\rotatebox[2]{#2}%
  \ifx\svgwidth\undefined%
    \setlength{\unitlength}{49.203656bp}%
    \ifx\svgscale\undefined%
      \relax%
    \else%
      \setlength{\unitlength}{\unitlength * \real{\svgscale}}%
    \fi%
  \else%
    \setlength{\unitlength}{\svgwidth}%
  \fi%
  \global\let\svgwidth\undefined%
  \global\let\svgscale\undefined%
  \makeatother%
  \begin{picture}(1,0.81294772)%
    \put(0,0){\includegraphics[width=\unitlength]{dualrmatrixopbraiding2.pdf}}%
    \put(0.05451132,0.25143921){\color[rgb]{0,0,0}\makebox(0,0)[lb]{\smash{$\ov{R}$}}}%
  \end{picture}%
\endgroup%

%% file: Graphics/dualrmatrixopbraiding3.pdf_tex
\begingroup%
  \makeatletter%
  \providecommand\color[2][]{%
    \errmessage{(Inkscape) Color is used for the text in Inkscape, but the package 'color.sty' is not loaded}%
    \renewcommand\color[2][]{}%
  }%
  \providecommand\transparent[1]{%
    \errmessage{(Inkscape) Transparency is used (non-zero) for the text in Inkscape, but the package 'transparent.sty' is not loaded}%
    \renewcommand\transparent[1]{}%
  }%
  \providecommand\rotatebox[2]{#2}%
  \ifx\svgwidth\undefined%
    \setlength{\unitlength}{49.075048bp}%
    \ifx\svgscale\undefined%
      \relax%
    \else%
      \setlength{\unitlength}{\unitlength * \real{\svgscale}}%
    \fi%
  \else%
    \setlength{\unitlength}{\svgwidth}%
  \fi%
  \global\let\svgwidth\undefined%
  \global\let\svgscale\undefined%
  \makeatother%
  \begin{picture}(1,0.81507816)%
    \put(0,0){\includegraphics[width=\unitlength]{dualrmatrixopbraiding3.pdf}}%
    \put(0.05465478,0.25209814){\color[rgb]{0,0,0}\makebox(0,0)[lb]{\smash{$\ov{R}$}}}%
  \end{picture}%
\endgroup%